\documentclass[11pt]{amsart}
 
\usepackage{fullpage}
\usepackage{amssymb}
\usepackage{amsmath}
\usepackage{amsxtra}
\usepackage{amscd}
\usepackage{graphicx}
\usepackage{amsfonts}
\usepackage{float}
\usepackage{enumerate}
\newtheorem*{theorem*}{Theorem}
\newtheorem{theorem}{Theorem}
\newtheorem{lemma}{Lemma} 
\newtheorem{corollary}{Corollary}
\newtheorem{prop}[theorem]{Proposition}

\newtheorem{definition}{Definition}

\usepackage{graphicx,color}
\usepackage{pstricks-add}
\usepackage{pgf,tikz,pgfplots}
\usepackage{mathrsfs}
\usetikzlibrary{arrows}
\pgfplotsset{compat=1.13}

\numberwithin{equation}{section}

\begin{document}

\title[]
  {Rail knotoids}
\author{Dimitrios Kodokostas}
\address{Department of Mathematics,
National Technical University of Athens,
Zografou campus,
{GR-15780} Athens, Greece.}
\email{dkodokostas@math.ntua.gr}

\author{Sofia Lambropoulou}
\address{Department of Mathematics,
National Technical University of Athens,
Zografou campus,
{GR-15780} Athens, Greece.}
\email{sofia@math.ntua.gr}
\urladdr{http://www.math.ntua.gr/$\sim$sofia}

\subjclass[2010]{57M27, 57M25}
\date{}


\keywords{rail arc, rail isotopy,  rail knotoid diagram, rail knotoid, knotoid handlebody of genus 2}

\begin{abstract} We work on the notions of rail arcs and rail isotopy  in $\mathbb{R}^3$, and we introduce the notions of rail knotoid diagrams and their equivalence. Our main result is that two rail arcs in $\mathbb{R}^3$ are rail isotopic if and only if their knotoid diagram projections  to the plane of two lines which we call rails, are equivalent. We also make a connection between the rail isotopy in $\mathbb{R}^3$ and the knot theory of the handlebody of genus $2$.
\end{abstract}

\dedicatory{Dedicated to Louis H. Kauffman for his 70th birthday}

\maketitle

\section*{Introduction} \label{section_introduction}

We study isotopies in $\mathbb{R}^3$ between arcs  that have their endpoints on two fixed parallel lines (we call them rails), that allow the endpoints to move freely on the rails but do not allow any other point of the arcs to  touch them. We call such arcs as rail arcs and such isotopies among them as rail isotopies.  As we remarked in \cite{KoLa1}, it turns out that rail isotopies are connected to the knot theory of the handlebody of genus~2. Rail arcs and rail isotopies are convenient renamings of what in \cite{GK1} are called open arcs and line isotopies respectively. It was proved in \cite{GK1} that rail isotopies can be studied diagrammatically with the notion of knotoid diagram which is a kind of generalization of  knot diagram. A planar knotoid diagram is what one gets by projecting a rail arc onto a plane perpendicular to the rails (keeping track of over/under data at crossings). 

Here we develop a new diagrammatic setting, which we call rail knotoid diagram, by projecting rail arcs onto the plane of the rails, and we prove that rail isotopy corresponds to (gives rise and comes from) an appropriately defined equivalence between such diagrams. Although we do not make use of any previous result on knotoids, the current article belongs in the general theory about them, and familiarizing with knotoids helps in putting the current work into context. Thus in \S \ref{section_knotoids} we recall some basic facts about knotoids. In \S \ref{section_rail_isotopies} we introduce the basic notions of rail arcs  and rail isotopy between such arcs and we develop the notion of triangle move for studying such isotopies. In \S \ref{section_handlebodies} we remark on the connection of the study of rail isotopy to the study of knot theory in the handlebody of genus~2.
In  \S \ref{rail_knotoids}   we introduce and study rail knotoid diagrams. We define a notion of equivalence between such diagrams and we prove that two rail arcs are rail isotopic if and only if their rail knotoid diagrams are equivalent.

We will be working in the piecewise linear (p.l.) category, thus all curves will be p.l. curves, all maps will be p.l. maps etc.  Due to the usual p.l. approximation theorems for the analogous smooth objects, our results hold in the smooth category as well.

\section{Rail isotopy} \label{section_rail_isotopies}

Henceforth $\mathbb{R}^3$ will be considered equipped with two given parallel lines $\ell_1,\ell_2$ (in this order) which we will call as \textit{rails}. We define:

\begin{definition} \rm
A \textit{rail arc} is any  oriented, connected, embedded arc $c$ in $\mathbb{R}^3$ with its interior in $\mathbb{R}^3-(\ell_1 \cup \ell_2)$, its first endpoint on $\ell_1$ and the last on $\ell_2$. We call two rail arcs  $c_1,c_2$  as \textit{rail isotopic}, if there exists an isotopy  of  $\mathbb{R}^3$ taking one onto the other (thus $c_1,c_2$ are connected by a homotopy of embeddings in $\mathbb{R}^3$), so that each rail maps onto itself (not necessarily pointwise) throughout the isotopy.  In particular, this implies that at each time throughout the isotopy, the image of the arc is a rail arc, and each endpoint remains on the same rail, but with the freedom to move up and down on it. We call such an isotopy as a \textit{rail isotopy} in $\mathbb{R}^3$.
\end{definition}

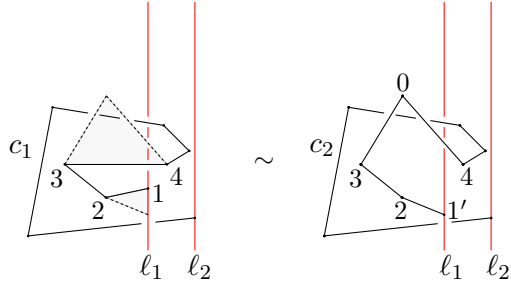
\begin{figure}[!h]
	\centering
	\psset{xunit=1.0cm,yunit=1.0cm,algebraic=true,dimen=middle,dotstyle=o,dotsize=5pt 0,linewidth=1.6pt,arrowsize=3pt 2,arrowinset=0.25}
	\begin{pspicture*}(-0.1039458738052424,0.06116059143102224)(6.715784802180208,3.933288375563131)
	\pspolygon[linewidth=0.pt,linecolor=lightgray,fillcolor=lightgray,fillstyle=solid,opacity=0.1](0.6445315826999247,1.6964695113948716)(1.1979962065470546,2.608469175390208)(2.00322258244179,1.6964695113948716)
	\pspolygon[linewidth=0.pt,linecolor=lightgray,fillcolor=lightgray,fillstyle=solid,opacity=0.1](1.1903889753849768,1.2576045696740883)(1.7507157590559494,1.377248967498395)(1.7507157590559494,1.0287192128094522)
	\psline[linewidth=0.4pt]{->}(-0.6586608091724795,5.749707869804877)(3.278727555591045,5.749707869804877)
	\psline[linewidth=0.4pt]{->}(-0.6586608091724795,5.749707869804877)(-0.6586608091724795,4.350591910714466)
	\psline[linewidth=0.4pt]{->}(-0.6586608091724795,5.749707869804877)(0.8749628356663519,4.738168870017557)
	\psline[linewidth=0.4pt,linestyle=dashed,dash=1pt 1pt](0.6445315826999247,1.6964695113948716)(1.1979962065470546,2.608469175390208)
	\psline[linewidth=0.4pt,linestyle=dashed,dash=1pt 1pt](1.1979962065470546,2.608469175390208)(2.00322258244179,1.6964695113948716)
	\psline[linewidth=0.4pt](0.6445315826999247,1.6964695113948716)(2.00322258244179,1.6964695113948716)
	\psline[linewidth=0.4pt](0.6445315826999247,1.6964695113948716)(1.1903889753849768,1.2576045696740883)
	\psline[linewidth=0.4pt](2.00322258244179,1.6964695113948716)(2.2979448314727335,1.8773675685835884)
	\psline[linewidth=0.4pt](2.2979448314727335,1.8773675685835884)(1.9586628065743275,2.216649593481997)
	\psline[linewidth=0.4pt](0.48114431104901,2.457430385345384)(0.15709644872051595,0.7463966880611413)
	\psline[linewidth=0.4pt,linecolor=red](1.7507157590559494,3.8583368107323617)(1.7507157590559494,2.332738656267294)
	\psline[linewidth=0.4pt,linecolor=red](1.7507157590559494,2.1947604305853257)(1.7507157590559494,2.05248087175696)
	\psline[linewidth=0.4pt,linecolor=red](1.7507157590559494,1.6146976138235303)(1.7507157590559494,0.5530732133349612)
	\psline[linewidth=0.4pt,linecolor=red](1.7507157590559494,1.9119991082678132)(1.7507157590559494,1.7552530021503596)
	\psline[linewidth=0.4pt](0.48114431104901,2.457430385345384)(0.9528997528226542,2.380551720760049)
	\psline[linewidth=0.4pt](1.1054122664252921,2.3556978296544338)(1.3771607537773767,2.311412890974834)
	\psline[linewidth=0.4pt](1.567414459270753,2.2804085834129504)(1.9586628065743275,2.216649593481997)
	\psline[linewidth=0.4pt,linecolor=red](2.3745569016110837,3.8583368107323617)(2.3745569016110837,0.5530732133349612)
	\psline[linewidth=0.4pt](4.581919947463449,1.6964695113948716)(5.1353845713105795,2.608469175390208)
	\psline[linewidth=0.4pt](5.1353845713105795,2.608469175390208)(5.940610947205315,1.6964695113948716)
	\psline[linewidth=0.4pt](4.581919947463449,1.6964695113948716)(5.127777340148501,1.2576045696740883)
	\psline[linewidth=0.4pt](5.940610947205315,1.6964695113948716)(6.235333196236258,1.8773675685835884)
	\psline[linewidth=0.4pt](6.235333196236258,1.8773675685835884)(5.896051171337852,2.216649593481997)
	\psline[linewidth=0.4pt](4.418532675812535,2.457430385345384)(4.094484813484041,0.7463966880611413)
	\psline[linewidth=0.4pt,linecolor=red](5.688104123819474,3.8583368107323617)(5.688104123819474,2.332738656267294)
	\psline[linewidth=0.4pt,linecolor=red](5.688104123819474,2.1947604305853257)(5.688104123819474,2.05248087175696)
	\psline[linewidth=0.4pt](4.418532675812535,2.457430385345384)(4.890288117586179,2.380551720760049)
	\psline[linewidth=0.4pt](5.0428006311888165,2.3556978296544338)(5.314549118540901,2.311412890974834)
	\psline[linewidth=0.4pt](5.504802824034278,2.2804085834129504)(5.896051171337852,2.216649593481997)
	\psline[linewidth=0.4pt,linecolor=red](6.311945266374608,3.8583368107323617)(6.311945266374608,0.5530732133349612)
	\psline[linewidth=0.4pt,linecolor=red](5.688104123819474,1.9119991082678132)(5.688104123819474,0.5530732133349612)
	\psline[linewidth=0.4pt](0.15709644872051595,0.7463966880611413)(1.6643261105468938,0.9090436340637899)
	\psline[linewidth=0.4pt](1.8494231157593275,0.929017672111554)(2.3745569016110837,0.9856854837097546)
	\psline[linewidth=0.4pt](4.094484813484041,0.7463966880611413)(5.601714475310418,0.9090436340637899)
	\psline[linewidth=0.4pt](5.786811480522852,0.929017672111554)(6.311945266374608,0.9856854837097546)
	\psline[linewidth=0.4pt](5.127777340148501,1.2576045696740883)(5.688104123819474,1.0287192128094522)
	\psline[linewidth=0.4pt,linestyle=dashed,dash=1pt 1pt](1.1903889753849768,1.2576045696740883)(1.7507157590559494,1.0287192128094522)
	\begin{scriptsize}
	\psdots[dotsize=1pt 0,dotstyle=*](-0.6586608091724795,5.749707869804877)
	\psdots[dotsize=1pt 0,dotstyle=*](3.278727555591045,5.749707869804877)
	\psdots[dotsize=1pt 0,dotstyle=*](-0.6586608091724795,4.350591910714466)
	\psdots[dotsize=1pt 0,dotstyle=*](0.8749628356663519,4.738168870017557)
	\psdots[dotsize=1pt 0,dotstyle=*](0.6445315826999247,1.6964695113948716)
	\psdots[dotsize=1pt 0,dotstyle=*](2.00322258244179,1.6964695113948716)
	\psdots[dotsize=1pt 0,dotstyle=*](1.1903889753849768,1.2576045696740883)
	\psdots[dotsize=1pt 0,dotstyle=*](2.2979448314727335,1.8773675685835884)
	\psdots[dotsize=1pt 0,dotstyle=*](1.9586628065743275,2.216649593481997)
	\psdots[dotsize=1pt 0,dotstyle=*](0.48114431104901,2.457430385345384)
	\psdots[dotsize=1pt 0,dotstyle=*](0.15709644872051595,0.7463966880611413)
	\psdots[dotsize=1pt 0,dotstyle=*](4.581919947463449,1.6964695113948716)
	\psdots[dotsize=1pt 0,dotstyle=*](5.1353845713105795,2.608469175390208)
	\psdots[dotsize=1pt 0,dotstyle=*](5.940610947205315,1.6964695113948716)
	\psdots[dotsize=1pt 0,dotstyle=*](4.581919947463449,1.6964695113948716)
	\psdots[dotsize=1pt 0,dotstyle=*](5.1353845713105795,2.608469175390208)
	\psdots[dotsize=1pt 0,dotstyle=*](5.1353845713105795,2.608469175390208)
	\psdots[dotsize=1pt 0,dotstyle=*](5.940610947205315,1.6964695113948716)
	\psdots[dotsize=1pt 0,dotstyle=*](4.581919947463449,1.6964695113948716)
	\psdots[dotsize=1pt 0,dotstyle=*](5.940610947205315,1.6964695113948716)
	\psdots[dotsize=1pt 0,dotstyle=*](4.581919947463449,1.6964695113948716)
	\psdots[dotsize=1pt 0,dotstyle=*](5.127777340148501,1.2576045696740883)
	\psdots[dotsize=1pt 0,dotstyle=*](5.940610947205315,1.6964695113948716)
	\psdots[dotsize=1pt 0,dotstyle=*](6.235333196236258,1.8773675685835884)
	\psdots[dotsize=1pt 0,dotstyle=*](6.235333196236258,1.8773675685835884)
	\psdots[dotsize=1pt 0,dotstyle=*](5.896051171337852,2.216649593481997)
	\psdots[dotsize=1pt 0,dotstyle=*](4.418532675812535,2.457430385345384)
	\psdots[dotsize=1pt 0,dotstyle=*](4.094484813484041,0.7463966880611413)
	\psdots[dotsize=1pt 0,dotstyle=*](4.418532675812535,2.457430385345384)
	\psdots[dotsize=1pt 0,dotstyle=*](5.896051171337852,2.216649593481997)
	\psdots[dotsize=1pt 0,dotstyle=*](5.127777340148501,1.2576045696740883)
	\psdots[dotsize=1pt 0,dotstyle=*](2.3745569016110837,0.9856854837097546)
	\psdots[dotsize=1pt 0,dotstyle=*](1.7507157590559494,1.377248967498395)
	\psdots[dotsize=1pt 0,dotstyle=*](4.094484813484041,0.7463966880611413)
	\psdots[dotsize=1pt 0,dotstyle=*](6.311945266374608,0.9856854837097546)
	\psdots[dotsize=1pt 0,dotstyle=*](5.127777340148501,1.2576045696740883)
	\psdots[dotsize=1pt 0,dotstyle=*](5.688104123819474,1.0287192128094522)
	\end{scriptsize}\small
	\rput[tl](1.8,1.42){$1$}
	\rput[tl](1.0,1.2){$2$}
	\rput[tl](5.0529875305472896,2.9){$0$}
	\psline[linewidth=0.4pt](1.1903889753849768,1.2576045696740883)(1.7507157590559494,1.377248967498395)
	\rput[tl](0.45,1.63){$3$}
	\rput[tl](2.04,1.65){$4$}
	\rput[tl](5.725343012737813,1.2333491421717862){$1'$}
	\rput[tl](5.023754683495527,1.2){$2$}
	\rput[tl](4.427404603639585,1.63){$3$}
	\rput[tl](5.9,1.65){$4$}
	\rput[tl](3.159083157766739,1.8){$\sim$}
	\normalsize
	\rput[tl](1.636336276366481,0.5){$\ell_1$}
	\rput[tl](2.278065319242304,0.5){$\ell_2$}
	\rput[tl](5.628108458322873,0.5){$\ell_1$}
	\rput[tl](6.25,0.5){$\ell_2$}
	\rput[tl](-0.1,2.016789443159007){$c_1$}
	\rput[tl](3.9,2.034329151390064){$c_2$}
	\end{pspicture*}
	\caption{Two triangle moves in $\mathbb{R}^3$ which are compatible for the rail arc $c_1$ and  transform it to  the rail arc $c_2$.}
	\label{figure_triangle_moves}
\end{figure}

Similarly to the case of isotopy between p.l. knots in $\mathbb{R}^3$,  rail isotopy between p.l. rail arcs can be effected via a finite sequence of  \textit{triangle moves} or \textit{elementary moves}: a rail arc $c$ is modified so that it either replaces its edge $AB$ by two new edges $AC,CB$, or vice versa, where the triangle $\Delta=ABC$ does not intersect $c$ or the rails at any other point (see Figure \ref{figure_triangle_moves}) or else it replaces edge $AB$ with $CB$, where $A,C$ lie on the same rail and the triangle $\Delta=ABC$ has  no other common points with the arc $c$, and no other common point with the rails other than segment $AC$. When it will be necessary to distinguish the second kind of move from the first one, we will be calling it as a \textit{space slide move}. Each triangle move, which from now on will be denoted $\Delta$ as the triangle itself, is actually the restriction on an edge (or on two consecutive edges) of the result of an isotopy $h_\Delta$ in $\mathbb{R}^3$ which fixes all points in the complement of the interior of a closed $3$-ball neighborhood of $\Delta$ (with the vertices of $\Delta$ on the boundary of the ball if we wish). This can readily be made to keep the rails fixed (although not necessarily pointwise) which means it is a rail isotopy. So instead of rail isotopies between arcs, we can indistinguishably use, if we wish, triangle moves. We talk in general about triangle moves in $\mathbb{R}^3$ without any reference to rail arcs, as long as the triangles $ABC $ do not intersect the rails as explained above. For a triangle move $\Delta$ replacing edge $AB$ by edges $BC,CA$, its \textit{inverse} move $\Delta^{-1}$ is the triangle move replacing edges $AC,CB$ by edge $AB$; the inverse of $\Delta^{-1}$ is $\Delta$.

If $\Delta_1,\ldots,\Delta_k$ are the triangles for a sequence of triangle moves in $\mathbb{R}^3$ and $h_{\Delta_1},\ldots, h_{\Delta_k}$ corresponding isotopies in space, instead of writing $h_{\Delta_k}\circ \cdots \circ h_{\Delta_1}$ we can write $\Delta_k\circ \cdots \circ \Delta_1$, and say that $\Delta_k\circ \cdots \circ \Delta_1$ is the \textit{composite move} of the triangle moves $\Delta_1, \ldots, \Delta_k$ (in this order). We call the last ones as \textit{submoves} of the former. If $\Delta_1,\ldots,\Delta_k$ are the triangles for a sequence of triangle moves for getting $c_2$ from $c_1$, we say $\Delta=\Delta_k\circ \cdots \circ \Delta_1$ is a \textit{composite move} for $c_1$ and  write $c_1 \overset{\Delta}{\sim} c_2$. Let us note that a triangle move in $\mathbb{R}^3$ performed in an edge of a rail arc is not necessary a triangle move of the arc itself as the triangle of the move may interfere with the rest of the arc. Similarly, whenever we get a rail arc $c_2$ from another rail arc $c_1$ via a sequence of triangle moves $\Delta_1,\ldots,\Delta_k$, it is not always true that we can replace any $\Delta_i$ by other moves $\delta_j$ in $\mathbb{R}^3$ that compose to it and still claim that we get $c_2$ form $c_1$, since for example the $\delta_j$'s may interfere with the rest of $(\Delta_{i-1}\circ\cdots\circ \Delta_1)(c_1)$.  But clearly, we can do so in case the triangles of the $\delta_j$'s are just part of $\Delta_i$. Also, two moves on distinct edges of an arc $c$ whose triangles have no common points other than possibly a common vertex of $c$, can be performed in either order and we can think of them as being performed simultaneously. We call two such moves as \textit{compatible for $c$} and extend the definition to any finite number of moves when each is performed on its own edge or pair of edges of $c$ and any two of their triangles have nothing else in common other than possibly a vertex of $c$. Such moves can be performed in any order; rigorously, this means that their composition is defined in any order, and the result is always the same.

 Finally, applying to an arc a sequence of moves, and then performing in reverse order their inverses, we return to the original arc. Thus:

\begin{lemma}\label{lemma_triangle_moves}
	For triangle moves $\Delta_i$ and rail arcs $c_1,c_2$ so that $c_1 \overset{\Delta_k \circ \cdots \circ  \Delta_1}{\sim} c_2$, the following are true:
	\begin{equation*}\label{equation_composition}
	\begin{split} 
	 \hspace{-9ex} & \bullet \left[  \text{For some } i, \Delta_i=\delta_l \circ \cdots \circ  \delta_1, \text{ and } \ \delta_j \subseteq \Delta_i, \forall j \right]   \Rightarrow c_1 \overset{\Delta_k \circ \cdots \circ \left( \delta_l \circ \cdots \circ  \delta_1 \right)  \circ \cdots \circ \Delta_1}{\sim} c_2 \\
	 \hspace{-9ex}	& \bullet \left[ \Delta_1,\ldots, \Delta_k \text{ compatible for } c_1 \text{ and } \tau \text{ is a permutation of } \{1,\ldots,k \} \right] \Rightarrow c_1 \overset{\Delta_{\tau(1)} \circ \cdots \circ \Delta_{\tau(k)}}{\sim} c_2  \\
	  \hspace{-9ex}	& \bullet  c_2 \overset{\Delta_1^{-1} \circ \cdots \circ \Delta_k^{-1}}{\sim} c_1  
		\end{split}
	\end{equation*}\normalsize 
\end{lemma}\normalsize

Any rail arc without its endpoints lies in $\mathbb{R}^3-(\ell_1 \cup \ell_2)$ and  we call it as \textit{open rail arc}. The boundary of an open rail arc as a subset of $\mathbb{R}^3$ consists of two points, one on $\ell_1$ and the other on $\ell_2$. We call an isotopy of $\mathbb{R}^3-(\ell_1 \cup \ell_2)$ as \textit{rail isotopy}, if it moves one open rail arc onto another, keeping at all times the image of the open rail arc as an open rail arc (the boundary points of the images of the rail arc considered as subsets of $\mathbb{R}^3$,  are onto the rails). 

Clearly, the definitions imply that each rail arc corresponds to a unique open rail arc and vice versa (the first lies inside $\mathbb{R}^3$, and the second inside $\mathbb{R}^3-(\ell_1 \cup \ell_2)$), and also that a rail isotopy of rail arcs gives rise to a rail isotopy of open rail arcs and vice versa (the first takes place inside $\mathbb{R}^3$, and the second inside $\mathbb{R}^3-(\ell_1 \cup \ell_2)$).

\section{Rail isotopy and handlebodies} \label{section_handlebodies}

 The space $\mathbb{R}^3-(\ell_1 \cup \ell_2)$ is homeomorphic to the interior $\left( \mathcal{H}_2\right)^o $ of a handlebody $\mathcal{H}_2=A \times [0,1]$ of genus 2, where $A$ is an annular thickening of a figure eight plane curve, with boundary three circles, say $c_0,c_1,c_2$.  Let $h:\mathbb{R}^3-(\ell_1 \cup \ell_2)\rightarrow \left( \mathcal{H}_2\right)^o$ be a homeomorphism of the two manifolds, sending points close to $\ell_i$ to points close to $c_i \times [0,1]$. Then it is meaningful to repeat the definitions for rails, rail arcs, open rail arcs and rail isotopies, replacing (i) $\mathbb{R}^3$ by $\mathcal{H}_2-\left(  (c_0 \times [0,1]) \cup (A\times 0) \cup (A \times 1) \right) $, (ii) $\mathbb{R}^3-(\ell_1 \cup \ell_2)$ by $\left( \mathcal{H}_2\right)^o $ and (iii)   $\ell_i$ by $c_i\times [0,1]$ for $ i=1,2$.
 
 \begin{figure}[!h]
 	\centering
 	\psset{xunit=1.0cm,yunit=1.0cm,algebraic=true,dimen=middle,dotstyle=o,dotsize=5pt 0,linewidth=1.6pt,arrowsize=3pt 2,arrowinset=0.25}
 	\begin{pspicture*}(3.54009372196738,2.091493423557577)(11,7.3865635576372926)
 	\psaxes[labelFontSize=\scriptstyle,xAxis=true,yAxis=true,Dx=0.5,Dy=0.5,ticksize=-2pt 0,subticks=2]{->}(0,0)(3.54009372196738,2.091493423557577)(9.201685647360373,7.3865635576372926)
 	\psline[linewidth=0.8pt](3.8430382495594717,3.433065009275547)(3.8430382495594717,6.089371484643429)
 	\psline[linewidth=0.8pt](7.9723189951167654,3.433065009275547)(7.9723189951167654,6.089371484643429)
 	\parametricplot[linewidth=0.8pt]{1.3048037387004574}{3.4944870324672994}{1.*0.22605357698322218*cos(t)+0.*0.22605357698322218*sin(t)+4.664694613240124|0.*0.22605357698322218*cos(t)+1.*0.22605357698322218*sin(t)+4.961491675954828}
 	\parametricplot[linewidth=0.8pt,linestyle=dotted]{1.3992565315392431}{1.5210216797027574}{1.*2.7554543609281943*cos(t)+0.*2.7554543609281943*sin(t)+4.741308651960328|0.*2.7554543609281943*cos(t)+1.*2.7554543609281943*sin(t)+2.439896659641171}
 	\parametricplot[linewidth=0.8pt]{-1.231600170046141}{0.4156879139013022}{1.*0.5406984685006947*cos(t)+0.*0.5406984685006947*sin(t)+6.574629249428455|0.*0.5406984685006947*cos(t)+1.*0.5406984685006947*sin(t)+4.942736449718199}
 	\parametricplot[linewidth=0.8pt,linestyle=dotted]{1.5214082662210313}{1.9504817421980185}{1.*0.8816485718272703*cos(t)+0.*0.8816485718272703*sin(t)+6.920840689547784|0.*0.8816485718272703*cos(t)+1.*0.8816485718272703*sin(t)+4.31136484443712}
 	\parametricplot[linewidth=0.8pt]{4.772453968767467}{5.1947208373222304}{1.*3.132986786053057*cos(t)+0.*3.132986786053057*sin(t)+4.110214474216897|0.*3.132986786053057*cos(t)+1.*3.132986786053057*sin(t)+7.300980158359168}
 	\parametricplot[linewidth=0.8pt]{1.1953744642837782}{1.4015376813291982}{1.*3.4560472019835107*cos(t)+0.*3.4560472019835107*sin(t)+4.783774349999232|0.*3.4560472019835107*cos(t)+1.*3.4560472019835107*sin(t)+1.7112205469269897}
 	\parametricplot[linewidth=0.8pt]{1.1436894175271268}{4.826292135688437}{1.*0.12319058887240927*cos(t)+0.*0.12319058887240927*sin(t)+4.284282679922288|0.*0.12319058887240927*cos(t)+1.*0.12319058887240927*sin(t)+4.296035592177741}
 	\parametricplot[linewidth=0.8pt]{4.417783347125596}{5.8251760246802755}{1.*0.39564394901207567*cos(t)+0.*0.39564394901207567*sin(t)+4.450193233080985|0.*0.39564394901207567*cos(t)+1.*0.39564394901207567*sin(t)+4.78675803508732}
 	\parametricplot[linewidth=0.8pt]{0.021067184328270528}{1.8077082372963895}{1.*0.2855361301616859*cos(t)+0.*0.2855361301616859*sin(t)+4.519587171071026|0.*0.2855361301616859*cos(t)+1.*0.2855361301616859*sin(t)+4.605803747820963}
 	\parametricplot[linewidth=0.8pt]{4.107668549026615}{5.224867179713746}{1.*0.3788437659190997*cos(t)+0.*0.3788437659190997*sin(t)+6.568773680919235|0.*0.3788437659190997*cos(t)+1.*0.3788437659190997*sin(t)+4.763020140588033}
 	\parametricplot[linewidth=0.8pt]{5.681727366699875}{5.740076894650666}{1.*5.788259093820678*cos(t)+0.*5.788259093820678*sin(t)+1.5809025154846623|0.*5.788259093820678*cos(t)+1.*5.788259093820678*sin(t)+7.726618532262496}
 	\parametricplot[linewidth=0.8pt]{4.147288176279135}{4.674148774858113}{1.*0.6780348524712857*cos(t)+0.*0.6780348524712857*sin(t)+6.562190006289699|0.*0.6780348524712857*cos(t)+1.*0.6780348524712857*sin(t)+5.412787618893904}
 	\parametricplot[linewidth=0.8pt]{5.121467949776788}{5.419233131648573}{1.*3.457582585561001*cos(t)+0.*3.457582585561001*sin(t)+4.188135209050859|0.*3.457582585561001*cos(t)+1.*3.457582585561001*sin(t)+7.69770697552369}
 	\psline[linewidth=0.8pt](4.805059939360503,6.089371484643429)(4.805059939360503,5.420283349334116)
 	\psline[linewidth=0.8pt](4.805059939360503,3.433065009275547)(4.805059939360503,4.210672185820206)
 	\psline[linewidth=0.8pt](4.805059939360503,5.272167696658945)(4.805059939360503,4.340273381910983)
 	\psline[linewidth=0.8pt](5.27908911177121,6.089371484643429)(5.27908911177121,5.414111863805983)
 	\psline[linewidth=0.8pt](5.27908911177121,5.210452841377621)(5.27908911177121,4.482217549058021)
 	\psline[linewidth=0.8pt](5.27908911177121,4.327930410854716)(5.27908911177121,3.433065009275547)
 	\psline[linewidth=0.8pt](6.536268132905109,6.089371484643429)(6.536268132905109,5.3832544361653225)
 	\psline[linewidth=0.8pt](6.536268132905109,3.433065009275547)(6.536268132905109,4.303244468742188)
 	\psline[linewidth=0.8pt](6.536268132905109,4.439017150361096)(6.536268132905109,5.191938384793224)
 	\psline[linewidth=0.8pt](7.010297305315627,6.089371484643429)(7.010297305315627,5.5)
 	\psline[linewidth=0.8pt](7.010297305315627,4.5)(7.010297305315627,3.433065009275547)
 	\psline[linewidth=0.8pt](7.010297305315627,5.278339182187075)(7.010297305315627,4.698219542542648)
 	\parametricplot[linewidth=0.8pt,linestyle=dotted]{0.0}{3.141592653589793}{1.*2.0646403727786478*cos(t)+0.*0.8781148250354871*sin(t)+5.907678622338119|0.*2.0646403727786478*cos(t)+1.*0.8781148250354871*sin(t)+3.433065009275547}
 	\parametricplot[linewidth=0.8pt]{-3.141592653589793}{0.0}{1.*2.0646403727786478*cos(t)+0.*0.8781148250354871*sin(t)+5.907678622338119|0.*2.0646403727786478*cos(t)+1.*0.8781148250354871*sin(t)+3.433065009275547}
 	\parametricplot[linewidth=0.8pt,linestyle=dotted]{0.0}{3.141592653589793}{1.*0.23701458620536614*cos(t)+0.*0.18736480493772018*sin(t)+5.042074525565857|0.*0.23701458620536614*cos(t)+1.*0.18736480493772018*sin(t)+3.4330650092755475}
 	\parametricplot[linewidth=0.8pt]{-3.141592653589793}{0.0}{1.*0.23701458620536614*cos(t)+0.*0.18736480493772018*sin(t)+5.042074525565857|0.*0.23701458620536614*cos(t)+1.*0.18736480493772018*sin(t)+3.4330650092755475}
 	\parametricplot[linewidth=0.8pt,linestyle=dotted]{0.0}{3.141592653589793}{1.*0.2370145862052594*cos(t)+0.*0.1873648049376358*sin(t)+6.773282719110368|0.*0.2370145862052594*cos(t)+1.*0.1873648049376358*sin(t)+3.4330650092755475}
 	\parametricplot[linewidth=0.8pt]{-3.141592653589793}{0.0}{1.*0.2370145862052594*cos(t)+0.*0.1873648049376358*sin(t)+6.773282719110368|0.*0.2370145862052594*cos(t)+1.*0.1873648049376358*sin(t)+3.4330650092755475}
 	\psline[linewidth=0.8pt,linecolor=red](5.0420745255658606,5.778229509965787)(5.0420745255658606,5.377082950637194)
 	\psline[linewidth=0.8pt,linecolor=red](5.0420745255658606,5.216624326905757)(5.0420745255658606,4.401988237192306)
 	\psline[linewidth=0.8pt,linecolor=red](5.0420745255658606,3.180034102622129)(5.0420745255658606,2.7727160577654035)
 	\psline[linewidth=0.8pt,linecolor=red](5.0420745255658606,2.5)(5.0420745255658606,2.2666542444585627)
 	\psline[linewidth=0.8pt,linecolor=red](6.77328271911038,5.809086937606448)(6.77328271911038,5.407940378277855)
 	\psline[linewidth=0.8pt,linecolor=red](6.77328271911038,5.235138783490154)(6.77328271911038,4.556275375395611)
 	\psline[linewidth=0.8pt,linecolor=red](6.77328271911038,3.1553481605096003)(6.77328271911038,2.8961457683280476)
 	\psline[linewidth=0.8pt,linecolor=red](6.77328271911038,2.519685151111983)(6.77328271911038,2.2419683023460335)
 	\psline[linewidth=0.8pt,linecolor=red](5.0420745255658606,7.339773541724954)(5.0420745255658606,6.089371484643429)
 	\psline[linewidth=0.8pt,linecolor=red](6.77328271911038,7.347571877710344)(6.77328271911038,6.089371484643429)
 	\psline[linewidth=0.8pt,linecolor=red](5.0420745255658606,4.2723870411015294)(5.0420745255658606,3.433065009275547)
 	\psline[linewidth=0.8pt,linecolor=red](6.77328271911038,4.364959324023513)(6.77328271911038,3.433065009275547)
 	\parametricplot[linewidth=0.8pt]{-4.420169340701285}{1.2771690642823572}{1.*0.2370145862053656*cos(t)+0.*0.18736480493771993*sin(t)+5.042074525565866|0.*0.2370145862053656*cos(t)+1.*0.18736480493771993*sin(t)+6.089371484643429}
 	\parametricplot[linewidth=0.8pt]{-4.3802177485623925}{1.2453923791926385}{1.*0.23701458620533894*cos(t)+0.*0.1873648049376988*sin(t)+6.7732827191103695|0.*0.23701458620533894*cos(t)+1.*0.1873648049376988*sin(t)+6.089371484643429}
 	\parametricplot[linewidth=0.8pt]{-4.214887167162374}{1.0646583645245382}{1.*2.0646403727786655*cos(t)+0.*0.8781148250354948*sin(t)+5.907678622338119|0.*2.0646403727786655*cos(t)+1.*0.8781148250354948*sin(t)+6.089371484643429}
 	\parametricplot[linewidth=0.8pt]{1.193222464394093}{1.9600639305605871}{1.*2.0646403727786655*cos(t)+0.*0.8781148250354948*sin(t)+5.907678622338119|0.*2.0646403727786655*cos(t)+1.*0.8781148250354948*sin(t)+6.089371484643429}
 	\rput[tl](3.9,2.9){$c_0$}
 	\rput[tl](4.6,3.25){$c_1$}
 	\rput[tl](6.9,3.25){$c_2$}
 	\rput[tl](8.016338577581124,5.078256105961894){$\mathcal{H}_2=A\times[0,1]$}
 	\rput[tl](5.801611157730421,3.261243821366056){$A$}
 	\rput[tl](4.1,5.093852777932673){$\gamma$}
 	\rput[tl](4.7,2.45){$\ell_1$}
 	\rput[tl](6.82,2.45){$\ell_2$}
 	\begin{scriptsize}
 	\psdots[dotsize=2pt 0,dotstyle=*](4.805059939360503,4.611818745148797)
 	\psdots[dotsize=2pt 0,dotstyle=*](6.536268132905109,4.735248455711441)
 	\end{scriptsize}
 	\end{pspicture*}
 	\caption{The handlebody $\mathcal{H}_2$ of genus 2 standardly embedded in $\mathbb{R}^3$, a knot in $\mathcal{H}_2$ with two points on its boundary, and two `rail' lines $\ell_1,\ell_2$ in $\mathbb{R}^3$ through the two `holes' of the handlebody.}
 	\label{figure_handlebody_of_genus_2}
 \end{figure}
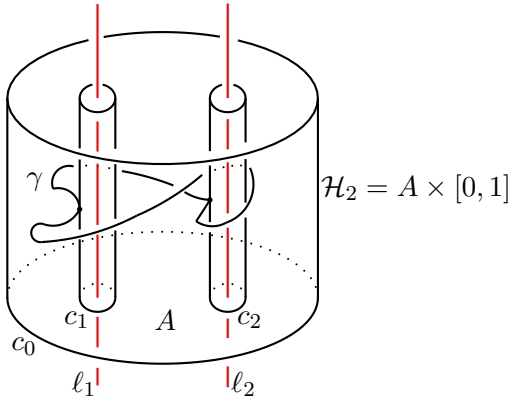

With these definitions, rail isotopy in $\mathcal{H}_2-\left(  (c_0 \times [0,1]) \cup (A\times 0) \cup (A \times 1) \right) $ corresponds by $h^{-1}$ to rail isotopy in $\mathbb{R}^3-(\ell_1 \cup \ell_2)$. But the first rail isotopy, clearly corresponds (gives rise in the natural way, and vice versa) to rail isotopy in $\mathcal{H}_2$, where now we allow the isotopies to touch the part $(c_0 \times [0,1]) \cup (A\times 0) \cup (A \times 1)$ of the boundary of the handlebody of genus two, retaining the rest of the above mentioned properties about keeping endpoints on the cylinders $c_i \times [0,1]$. And as already said, the second rail isotopy above, corresponds (gives rise in a natural way, and vice versa) to the rail isotopy of $\mathbb{R}^3$. Summarizing: 	

\begin{prop}
	Rail isotopies in $\mathcal{H}_2$ are in one to one correspondence  with rail isotopies in $\mathbb{R}^3$.
\end{prop}

Let us now notice that any knot in $\mathcal{H}_2$ is isotopic to some knot $k$ with a unique point on $c_1\times [0,1]$ and a unique point on $c_2 \times [0,1]$. Let us call $k$ as \textit{rail knot}, its points on $c_i\times [0,1], \ i=1,2$ as rail points  and the two arcs of $k$ with endpoints the rail point as its corresponding rail arcs. 
For two rail knots $k_1,k_2$, we can modify any isotopy of one onto the other so as at each stage its image is a rail knot. We get then a rail isotopy of the two rail arcs of $k_1$ to the two rail arcs of $k_2$. Throughout this isotopy, the images of the two arcs are disjoint except for the two rail points of the first arc which have the same images as the two rail points of the second. Let us in general, call  two rail isotopies of two rail arcs with the same endpoints, as \textit{matching} whenever these two properties hold.   Then the converse of the first observation about knots holds, and summarizing we get:  

\begin{prop}
	Knot isotopies in  $\mathcal{H}_2$ are in one to one correspondence with matching rail isotopies of two rail arcs with the same endpoints.
\end{prop}

\section{A reminder on knotoids} \label{section_knotoids}

In this section we recall some facts about knotoids. A \textit{knotoid diagram} $K$ in an oriented surface $\Sigma$  is an immersion of the unit interval $[0,1]$ in $\Sigma$ with a finite number of double points each of which is a transversal self-intersection endowed with over/under data. These are the  crossings of $K$. The images of $0$ and $1$ are two distinct points called the \textit{endpoints} of $K$ and are specifically called  \textit{leg} and  \textit{head}, respectively, so that $K$  is naturally oriented from its leg to its head. The trivial knotoid diagram is assumed to be an immersed arc without any self-intersections. See Figure~\ref{figure_knotoids_examples} for some examples.

\begin{figure}[!h]
	\centering
	\psset{xunit=1.0cm,yunit=1.0cm,algebraic=true,dimen=middle,dotstyle=o,dotsize=5pt 0,linewidth=1.6pt,arrowsize=3pt 2,arrowinset=0.25}
	\begin{pspicture*}(1.5548788849060198,2.884231732728732)(9.993244060037942,4.231105013402756)
	\psaxes[labelFontSize=\scriptstyle,xAxis=true,yAxis=true,Dx=1.,Dy=1.,ticksize=-2pt 0,subticks=2]{->}(0,0)(1.5548788849060198,2.884231732728732)(9.993244060037942,4.231105013402756)
	\parametricplot[linewidth=0.8pt]{-1.013690829904336}{0.7595360487821852}{1.*0.26133569415365393*cos(t)+0.*0.26133569415365393*sin(t)+7.231218944945993|0.*0.26133569415365393*cos(t)+1.*0.26133569415365393*sin(t)+3.5384858157264754}
	\parametricplot[linewidth=0.8pt]{1.5020661116134124}{2.785872906422819}{1.*0.6788608434088035*cos(t)+0.*0.6788608434088035*sin(t)+6.578822907889188|0.*0.6788608434088035*cos(t)+1.*0.6788608434088035*sin(t)+3.1503465667286434}
	\parametricplot[linewidth=0.8pt]{1.180418092077292}{2.509348563790416}{1.*0.4311110845775459*cos(t)+0.*0.4311110845775459*sin(t)+6.712447871732008|0.*0.4311110845775459*cos(t)+1.*0.4311110845775459*sin(t)+3.634926066911217}
	\parametricplot[linewidth=0.8pt]{3.537022663596186}{5.211678691778128}{1.*0.7345446073989543*cos(t)+0.*0.7345446073989543*sin(t)+7.017693896503618|0.*0.7345446073989543*cos(t)+1.*0.7345446073989543*sin(t)+3.9615404429725958}
	\parametricplot[linewidth=0.8pt]{0.3081239163679511}{1.065842404640644}{1.*0.761862779759634*cos(t)+0.*0.761862779759634*sin(t)+6.507937958787707|0.*0.761862779759634*cos(t)+1.*0.761862779759634*sin(t)+3.366822751138694}
	\parametricplot[linewidth=0.8pt]{0.625912904886897}{1.2657556832867782}{1.*0.5398485944695817*cos(t)+0.*0.5398485944695817*sin(t)+6.983219212130641|0.*0.5398485944695817*cos(t)+1.*0.5398485944695817*sin(t)+3.4021742001819897}
	\parametricplot[linewidth=0.8pt]{2.5646116478235923}{3.1321998662402377}{1.*0.6138247968713376*cos(t)+0.*0.6138247968713376*sin(t)+9.275389887207874|0.*0.6138247968713376*cos(t)+1.*0.6138247968713376*sin(t)+3.0105276965645897}
	\parametricplot[linewidth=0.8pt]{2.052276850041885}{2.453375620735589}{1.*1.7246980124416833*cos(t)+0.*1.7246980124416833*sin(t)+10.205358198678107|0.*1.7246980124416833*cos(t)+1.*1.7246980124416833*sin(t)+2.4314339098565987}
	\parametricplot[linewidth=0.8pt]{0.5945562282967374}{2.037159523224833}{1.*0.34172561323901396*cos(t)+0.*0.34172561323901396*sin(t)+9.560318335433536|0.*0.34172561323901396*cos(t)+1.*0.34172561323901396*sin(t)+3.6548189354489367}
	\parametricplot[linewidth=0.8pt]{-0.8844060200507915}{0.5085656213377389}{1.*0.397367096706076*cos(t)+0.*0.397367096706076*sin(t)+9.496325018219947|0.*0.397367096706076*cos(t)+1.*0.397367096706076*sin(t)+3.6527455938203652}
	\parametricplot[linewidth=0.8pt]{0.00976561142166362}{0.8712950698113682}{1.*1.1507194630952235*cos(t)+0.*1.1507194630952235*sin(t)+8.578864647549722|0.*1.1507194630952235*cos(t)+1.*1.1507194630952235*sin(t)+2.949175408800748}
	\parametricplot[linewidth=0.8pt]{0.8000569774650561}{2.4142309181627706}{1.*0.23655597867502653*cos(t)+0.*0.23655597867502653*sin(t)+8.9935066167198|0.*0.23655597867502653*cos(t)+1.*0.23655597867502653*sin(t)+3.7903472191250236}
	\parametricplot[linewidth=0.8pt]{2.499793470687477}{4.376976836726559}{1.*0.4297599582844793*cos(t)+0.*0.4297599582844793*sin(t)+9.161061899385682|0.*0.4297599582844793*cos(t)+1.*0.4297599582844793*sin(t)+3.6903631423821137}
	\parametricplot[linewidth=0.8pt]{4.380158001139283}{4.911854802855721}{1.*1.1053601351364213*cos(t)+0.*1.1053601351364213*sin(t)+9.380119136377663|0.*1.1053601351364213*cos(t)+1.*1.1053601351364213*sin(t)+4.3294674634347}
	\parametricplot[linewidth=0.8pt]{0.05377884529915093}{1.0888961143680977}{1.*0.6265904606399518*cos(t)+0.*0.6265904606399518*sin(t)+3.934323055206616|0.*0.6265904606399518*cos(t)+1.*0.6265904606399518*sin(t)+3.403437843266223}
	\parametricplot[linewidth=0.8pt]{1.061451961642927}{2.68916245835109}{1.*0.3133582952018022*cos(t)+0.*0.3133582952018022*sin(t)+4.071930010601838|0.*0.3133582952018022*cos(t)+1.*0.3133582952018022*sin(t)+3.6850876101429573}
	\parametricplot[linewidth=0.8pt]{2.8806953056028495}{4.245328415352342}{1.*0.4693753465785539*cos(t)+0.*0.4693753465785539*sin(t)+4.243590678090111|0.*0.4693753465785539*cos(t)+1.*0.4693753465785539*sin(t)+3.7009987141351943}
	\parametricplot[linewidth=0.8pt]{4.398764324452543}{5.336422144022633}{1.*1.0222801819678498*cos(t)+0.*1.0222801819678498*sin(t)+4.347630164885868|0.*1.0222801819678498*cos(t)+1.*1.0222801819678498*sin(t)+4.254310475969408}
	\parametricplot[linewidth=0.8pt]{-0.8226576128239484}{0.9839233647142973}{1.*0.39261105774853094*cos(t)+0.*0.39261105774853094*sin(t)+4.677877932999696|0.*0.39261105774853094*cos(t)+1.*0.39261105774853094*sin(t)+3.7124677992966917}
	\parametricplot[linewidth=0.8pt]{0.9122041334284211}{2.76679456943174}{1.*0.2777263079603046*cos(t)+0.*0.2777263079603046*sin(t)+4.725320566549362|0.*0.2777263079603046*cos(t)+1.*0.2777263079603046*sin(t)+3.819744624179272}
	\parametricplot[linewidth=0.8pt]{5.018389246503913}{5.717779966748461}{1.*0.5944581179287135*cos(t)+0.*0.5944581179287135*sin(t)+3.8717960428580342|0.*0.5944581179287135*cos(t)+1.*0.5944581179287135*sin(t)+4.05984256365547}
	\parametricplot[linewidth=0.8pt]{1.383088606265577}{2.273151477266394}{1.*0.4250527587292638*cos(t)+0.*0.4250527587292638*sin(t)+1.9846971923266015|0.*0.4250527587292638*cos(t)+1.*0.4250527587292638*sin(t)+3.210827315011941}
	\parametricplot[linewidth=0.8pt]{4.37175047172312}{4.658236369346498}{1.*1.6633287707623736*cos(t)+0.*1.6633287707623736*sin(t)+2.6197149851800363|0.*1.6633287707623736*cos(t)+1.*1.6633287707623736*sin(t)+5.196170299565205}
	\parametricplot[linewidth=0.8pt]{4.587352322685157}{5.659554732863946}{1.*0.5171490902681593*cos(t)+0.*0.5171490902681593*sin(t)+2.594179640812658|0.*0.5171490902681593*cos(t)+1.*0.5171490902681593*sin(t)+4.048391550175069}
	\psline[linewidth=0.8pt](4.27454322404783,3.622606743971146)(4.27454322404783,3.622606743971146)
	\psline[linewidth=0.8pt](4.27454322404783,3.622606743971146)(4.27454322404783,3.622606743971146)
	\psline[linewidth=0.8pt](5.952501992602721,3.5170898593655413)(6.037721857489821,3.5602998542754243)
	\psline[linewidth=0.8pt](6.037721857489821,3.5602998542754243)(6.032532878335625,3.464892288921176)
	\psline[linewidth=0.8pt](8.623406905465249,3.1109903369109597)(8.68564283016517,3.180760742603273)
	\psline[linewidth=0.8pt](8.68564283016517,3.180760742603273)(8.7149478115771,3.091977647938969)
	\psline[linewidth=0.8pt](1.7786866506390717,3.6372883142320966)(1.877747076669504,3.622204903687225)
	\psline[linewidth=0.8pt](1.877747076669504,3.622204903687225)(1.815154246946719,3.5439577635075974)
	\psline[linewidth=0.8pt](4.134131244639585,3.581879414455725)(4.230866298014159,3.58608206948836)
	\psline[linewidth=0.8pt](4.230866298014159,3.58608206948836)(4.186138377348477,3.5002057283132184)
	\begin{scriptsize}
	\psdots[dotsize=2pt 0,dotstyle=*](5.942461422020274,3.386770135808418)
	\psdots[dotsize=2pt 0,dotstyle=*](7.233920208381669,3.597873975886729)
	\psdots[dotsize=2pt 0,dotstyle=*](8.661592167316302,3.0162931375751327)
	\psdots[dotsize=2pt 0,dotstyle=*](9.729529240653733,2.960412709319109)
	\psdots[dotsize=2pt 0,dotstyle=*](4.560007632958618,3.437118914037305)
	\psdots[dotsize=2pt 0,dotstyle=*](4.050874842181515,3.4929993422933285)
	\psdots[dotsize=2pt 0,dotstyle=*](1.7101057919014326,3.535279793066744)
	\psdots[dotsize=2pt 0,dotstyle=*](3.013982451208654,3.7463836331450553)
	\end{scriptsize}
	\end{pspicture*}
	\caption{Examples of knotoid diagrams.}
	\label{figure_knotoids_examples}
\end{figure}
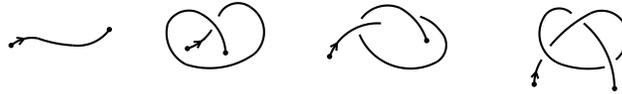

On the set of knotoid diagrams in $\Sigma$  the usual local  moves are allowed away from the endpoints. Namely the three Reidemeister moves together with planar isotopy, 
 which includes also the {\it swing moves} for the endpoints, whereby  an endpoint can be pulled within its region, without crossing any other arc of the diagram. See Figure~\ref{figure_knotoid_reidemeister_moves}. All these moves generate an equivalence relation in the set of knotoid diagrams in $\Sigma$ and the equivalence classes are called {\it knotoids}.  

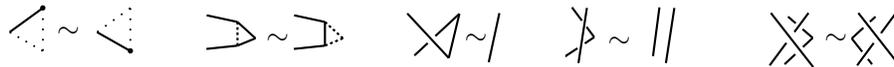
\begin{figure}[!h]
	\centering
	\psset{xunit=1.0cm,yunit=1.0cm,algebraic=true,dimen=middle,dotstyle=o,dotsize=5pt 0,linewidth=1.6pt,arrowsize=3pt 2,arrowinset=0.25}
	\begin{pspicture*}(0.5384030150416896,3.1453136553469774)(12.654373332312707,4.143559998126994)
	\psaxes[labelFontSize=\scriptstyle,xAxis=true,yAxis=true,Dx=0.5,Dy=0.5,ticksize=-2pt 0,subticks=2]{->}(0,0)(0.5384030150416896,3.1453136553469774)(12.654373332312707,4.143559998126994)
	\psline[linewidth=0.8pt](5.9745722612878644,3.965679933097659)(6.526162649653523,3.3735348686786497)
	\psline[linewidth=0.8pt](6.526162649653523,3.3735348686786497)(6.667560574354975,3.960880094361609)
	\psline[linewidth=0.8pt](3.307331385732172,3.9408697524405505)(3.7172423389348435,3.8524046011265356)
	\psline[linewidth=0.8pt](3.7172423389348435,3.8524046011265356)(3.9565311345834555,3.634869332355069)
	\psline[linewidth=0.8pt](3.7172423389348435,3.536978461407909)(3.303925328269058,3.4934714076536157)
	\psline[linewidth=0.8pt](3.9565311345834555,3.634869332355069)(3.7172423389348435,3.536978461407909)
	\psline[linewidth=0.8pt,linestyle=dashed,dash=1pt 1pt](3.7172423389348435,3.8524046011265356)(3.7172423389348435,3.536978461407909)
	\psplot[linewidth=0.4pt]{0.5384030150416896}{12.654373332312707}{(-4.746898439789713-0.*x)/-1.}
	\psline[linewidth=0.4pt]{->}(1.0815132602437787,4.746898439789713)(2.244092436353957,4.746898439789713)
	\psline[linewidth=0.8pt](4.469910561842351,3.9408697524405505)(4.8798215150450215,3.8524046011265356)
	\psline[linewidth=0.8pt](4.8798215150450215,3.8524046011265356)(4.8798215150450215,3.536978461407909)
	\psline[linewidth=0.8pt](4.8798215150450215,3.536978461407909)(4.466504504379237,3.4934714076536157)
	\psline[linewidth=0.8pt](8.362459203633898,4.038946323308519)(8.250182376971203,3.3057149008760947)
	\psline[linewidth=0.8pt](9.241271955480865,3.397784582893436)(9.326648420264412,4.04657670636543)
	\psline[linewidth=0.8pt](9.525038379744077,4.038946323308519)(9.412761553081381,3.3057149008760947)
	\psline[linewidth=0.8pt,linestyle=dashed,dash=1pt 1pt](4.8798215150450215,3.536978461407909)(5.1191103106936335,3.634869332355069)
	\psline[linewidth=0.8pt,linestyle=dashed,dash=1pt 1pt](5.1191103106936335,3.634869332355069)(4.8798215150450215,3.8524046011265356)
	\psline[linewidth=0.8pt](0.6874544057571992,3.7163936640801003)(1.1303566211355798,4.0380042209841225)
	\psline[linewidth=0.8pt,linestyle=dotted](0.6874544057571992,3.7163936640801003)(1.1303566211355798,3.46572549171585)
	\psline[linewidth=0.8pt,linestyle=dotted](1.1303566211355798,3.46572549171585)(1.1303566211355798,4.0380042209841225)
	\psline[linewidth=0.8pt,linestyle=dotted](1.8500335818673777,3.7163936640801003)(2.2929357972457582,4.0380042209841225)
	\psline[linewidth=0.8pt,linestyle=dotted](2.2929357972457582,4.0380042209841225)(2.2929357972457582,3.46572549171585)
	\psline[linewidth=0.8pt](2.2929357972457582,3.46572549171585)(1.8500335818673777,3.7163936640801003)
	\psline[linewidth=0.8pt](10.808425647559389,3.9714254956195245)(11.372154533214061,3.2512788554829157)
	\psline[linewidth=0.8pt](11.971004823669567,3.9714254956195245)(12.53473370932424,3.2512788554829157)
	\psline[linewidth=0.8pt](8.078692779370687,3.397784582893436)(8.078692779370687,3.397784582893436)
	\psline[linewidth=0.8pt](10.792245420888882,3.2741700046536466)(11.013968697562659,3.5736734806642088)
	\psline[linewidth=0.8pt](11.121057589998268,3.718328999055165)(11.308425647559389,3.9714254956195245)
	\psline[linewidth=0.8pt](11.234755387049098,3.5186046545900247)(11.364524150157152,3.6404283913853406)
	\psline[linewidth=0.8pt](11.364524150157152,3.6404283913853406)(11.229548465403681,3.7565702596615833)
	\psline[linewidth=0.8pt](11.138571110489373,3.834853099936686)(11.036417678710011,3.922752564491021)
	\psline[linewidth=0.8pt](12.198996854820189,3.922752564491021)(12.124975654764112,3.8547477117830202)
	\psline[linewidth=0.8pt](12.037883132557099,3.7747339570726215)(11.88339019802705,3.63279800832843)
	\psline[linewidth=0.8pt](11.88339019802705,3.63279800832843)(12.022395910027385,3.4559051302775488)
	\psline[linewidth=0.8pt](12.098641358292456,3.3588783492352303)(12.153214556478726,3.289430770767467)
	\psline[linewidth=0.8pt](11.95482459699906,3.2741700046536466)(12.175230503981282,3.5718939796225038)
	\psline[linewidth=0.8pt](12.271879729007138,3.702447610631487)(12.471004823669567,3.9714254956195245)
	\rput[tl](1.32,3.78){$\sim$}
	\rput[tl](4.1,3.7){$\sim$}
	\rput[tl](6.74,3.7){$\sim$}
	\rput[tl](8.645,3.65){$\sim$}
	\rput[tl](11.525,3.7){$\sim$}
	\psline[linewidth=0.8pt](6.246618136190322,3.5705516516998377)(6.069338585233445,3.406165158994369)
	\psline[linewidth=0.8pt](6.347887475681171,3.6644559483186256)(6.667560574354975,3.960880094361609)
	\psline[linewidth=0.8pt](8.078692779370687,3.397784582893436)(8.237077984082978,3.5023188180035483)
	\psline[linewidth=0.8pt](8.334494939311382,3.566614008454295)(8.4602119322162,3.6495872237714755)
	\psline[linewidth=0.8pt](8.164069244154234,4.04657670636543)(8.296422810105888,3.8691521884397426)
	\psline[linewidth=0.8pt](8.358184261030946,3.7863588343054655)(8.4602119322162,3.6495872237714755)
	\psline[linewidth=0.8pt](10.990635380368548,3.289430770767467)(11.132684966049267,3.422783443039162)
	\psline[linewidth=0.8pt](7.198879121504373,3.969556255505894)(7.057227517918134,3.3157796235694037)
	\begin{scriptsize}
	\psdots[dotsize=2pt 0,dotstyle=*](1.1303566211355798,4.0380042209841225)
	\psdots[dotsize=2pt 0,dotstyle=*](2.2929357972457582,3.46572549171585)
	\end{scriptsize}
	\end{pspicture*}
	\caption{Moves for knotoid diagrams.}
	\label{figure_knotoid_reidemeister_moves}
\end{figure}
	
	The moves consisting of pulling the arc adjacent to an endpoint over or under another arc, as shown in Figure~\ref{figure_forbidden_moves}, are the {\it forbidden moves} in the theory. Notice that, if both forbidden moves were allowed, any knotoid diagram in any surface could be clearly turned into the trivial knotoid diagram. 
	
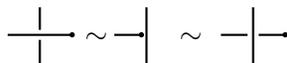
\begin{figure}[!h]
	\centering
	\psset{xunit=1.0cm,yunit=1.0cm,algebraic=true,dimen=middle,dotstyle=o,dotsize=5pt 0,linewidth=1.6pt,arrowsize=3pt 2,arrowinset=0.25}
	\begin{pspicture*}(0.4316028074167062,3.2447217946874587)(4.361643390934681,4.106678548434506)
	\psaxes[labelFontSize=\scriptstyle,xAxis=true,yAxis=true,Dx=0.5,Dy=0.5,ticksize=-2pt 0,subticks=2]{->}(0,0)(0.4316028074167062,3.2447217946874587)(4.361643390934681,4.106678548434506)
	\psline[linewidth=0.4pt]{->}(1.0815132602437787,4.746898439789713)(2.499093482488279,4.746898439789713)
	\rput[tl](1.6,3.75){$\sim$}
	\rput[tl](2.855,3.75){$\sim$}
	\psline[linewidth=0.8pt](0.5702392083690274,3.71488002400403)(1.4261682925094457,3.71488002400403)
	\psline[linewidth=0.8pt](0.9982037504392364,4.082567870008015)(0.9982037504392364,3.7811843896768793)
	\psline[linewidth=0.8pt](0.9982037504392364,3.6485756583311812)(0.9982037504392364,3.3170538299669308)
	\psline[linewidth=0.8pt](3.8333641949282367,4.082567870008015)(3.8333641949282367,3.3170538299669308)
	\psline[linewidth=0.8pt](3.4053996528580277,3.71488002400403)(3.7588764302724154,3.71488002400403)
	\psline[linewidth=0.8pt](3.907851959584055,3.71488002400403)(4.261328736998446,3.71488002400403)
	\psline[linewidth=0.8pt](2.4157839726837365,4.082567870008015)(2.4157839726837365,3.3170538299669308)
	\psline[linewidth=0.8pt](1.9878194306135275,3.71488002400403)(2.3484017423227126,3.71488002400403)
	\begin{scriptsize}
	\psdots[dotsize=1pt 0,dotstyle=*](1.0815132602437787,4.746898439789713)
	\psdots[dotsize=1pt 0,dotstyle=*](2.499093482488279,4.746898439789713)
	\psdots[dotsize=2pt 0,dotstyle=*](1.4261682925094457,3.71488002400403)
	\psdots[dotsize=2pt 0,dotstyle=*](4.261328736998446,3.71488002400403)
	\psdots[dotsize=2pt 0,dotstyle=*](2.3484017423227126,3.71488002400403)
	\end{scriptsize}
	\end{pspicture*}
	\caption{Forbidden knotoid moves.}
	\label{figure_forbidden_moves}
\end{figure}
	
The theory of knotoids was introduced by Turaev \cite{Tu} in 2010. The theory of  knotoids in the 2-sphere (spherical knotoids) extends classical knot theory and also proposes a new diagrammatic approach to  classical knot theory \cite{Tu}. 
 This approach promises reducing of the computational complexity of knot invariants, see  \cite{Tu,DST}.  In \cite{Tu} basic properties of knotoids were studied, including the introduction of several invariants of knotoids in the 2-sphere,  such as the complexity (or height \cite{GK1}) and the Jones/bracket polynomial. Knotoids in $S^2$ were classified by Bartholomew in \cite{Ba} up to 5 crossings by using Turaev's generalization of the bracket polynomial for knotoids. There is also a recent classification table for prime knotoids of positive complexity with up to 5 crossings \cite{KMT}, obtained by using the correspondence between  knotoids in $S^2$ and  knots in the thickened torus. New invariants for knotoids were introduced in \cite{Gthesis,GK1} in analogy with invariants from virtual knot theory. 

Planar knotoids surject to spherical knotoids, but do not inject \cite{Tu}.  For example, the first  two illustrations of Figure~\ref{figure_knotoids_examples} are equivalent as spherical knotoids but distinct as planar ones.  
 This means that planar knotoids provide a much richer combinatorial structure. This fact has interesting implications in the study of proteins.
 Indeed, recently knotoids have been studied in the field of biochemistry as they suggest new topological models for open protein chains \cite{GDBS,GGLDSK,DRGDSMRSS}.
 Such studies are enabled due to the following lifting of knotoids to open space curves (what we call rail arcs), proposed by G\"ug\"umc\"u and Kauffman in  \cite{GK1}. Namely,  an open space curve  in $\mathbb{R}^3$ projects to a planar knotoid diagram when projected along the two lines passing through its endpoints (what we call rails) and are perpendicular to a chosen projection plane, and this curve  can be viewed as a  lifting of the  diagram.  Figure~\ref{figure_projections_of_arcs} illustrates two such projections.  The method in \cite{GDBS,GGLDSK,DRGDSMRSS}  is to project an open protein chain to several planes and to consider all possible knotoid types obtained this way, choosing the dominant one for representing the protein. Then, the invariants introduced in \cite{Tu, Gthesis, GK1}  are used for determining its topological type. 

\begin{figure}[!h]
	\centering
	\psset{xunit=1.0cm,yunit=1.0cm,algebraic=true,dimen=middle,dotstyle=o,dotsize=5pt 0,linewidth=1.6pt,arrowsize=3pt 2,arrowinset=0.25}
	\begin{pspicture*}(3.190834007232368,2.1704652631105077)(8.32338166157489,6.674190853841482)
	\psaxes[labelFontSize=\scriptstyle,xAxis=true,yAxis=true,Dx=0.5,Dy=0.5,ticksize=-2pt 0,subticks=2]{->}(0,0)(3.190834007232368,2.1704652631105077)(8.32338166157489,6.674190853841482)
	\psline[linewidth=0.8pt](4.480477932711423,3.210771026470039)(5.250763258001653,4.215672783897377)
	\psline[linewidth=0.8pt](7.462275086307531,3.210771026470039)(8.23256041159776,4.215672783897377)
	\psline[linewidth=0.8pt](4.480477932711423,3.210771026470039)(7.462275086307531,3.210771026470039)
	\psline[linewidth=0.8pt,linecolor=red](5.7897163177173985,3.579614158436346)(5.7897163177173985,6.007590970387144)
	\psline[linewidth=0.8pt,linecolor=red](7.348194144442746,3.579614158436346)(7.348194144442746,4.56234264157902)
	\psline[linewidth=0.8pt,linecolor=red](7.348194144442746,4.750538252321563)(7.348194144442746,6.007590970387144)
	\parametricplot[linewidth=0.8pt]{1.3124896428446133}{2.0491738013646073}{1.*2.0441623491706924*cos(t)+0.*2.0441623491706924*sin(t)+6.730724779917389|0.*2.0441623491706924*cos(t)+1.*2.0441623491706924*sin(t)+3.123699349255933}
	\parametricplot[linewidth=0.8pt]{-0.897963385319354}{1.212895220174818}{1.*0.21069737658829024*cos(t)+0.*0.21069737658829024*sin(t)+7.336666585591969|0.*0.21069737658829024*cos(t)+1.*0.21069737658829024*sin(t)+4.843622530429749}
	\parametricplot[linewidth=0.8pt]{3.456529644292434}{5.190203848893322}{1.*0.643976505465421*cos(t)+0.*0.643976505465421*sin(t)+7.171847961037518|0.*0.643976505465421*cos(t)+1.*0.643976505465421*sin(t)+5.250697016198361}
	\parametricplot[linewidth=0.8pt]{1.850915222861518}{2.8256036869461454}{1.*1.1396477477429465*cos(t)+0.*1.1396477477429465*sin(t)+7.66327245733914|0.*1.1396477477429465*cos(t)+1.*1.1396477477429465*sin(t)+4.912363758490943}
	\parametricplot[linewidth=0.8pt]{-1.013690829904336}{0.7595360487821852}{1.*0.26133569415365393*cos(t)+0.*0.26133569415365393*sin(t)+7.231218944945993|0.*0.26133569415365393*cos(t)+1.*0.26133569415365393*sin(t)+3.5384858157264754}
	\parametricplot[linewidth=0.8pt]{0.42592965653394305}{1.8792143378973862}{1.*0.13285260566959992*cos(t)+0.*0.13285260566959992*sin(t)+7.18568062011666|0.*0.13285260566959992*cos(t)+1.*0.13285260566959992*sin(t)+3.790516595140854}
	\parametricplot[linewidth=0.8pt]{1.6219613144402374}{2.1081175146628013}{1.*2.41737740702649*cos(t)+0.*2.41737740702649*sin(t)+7.0270183312744265|0.*2.41737740702649*cos(t)+1.*2.41737740702649*sin(t)+1.502886623286431}
	\parametricplot[linewidth=0.8pt]{1.1181887922759286}{2.503745877034218}{1.*0.35393406681780265*cos(t)+0.*0.35393406681780265*sin(t)+6.721722502254331|0.*0.35393406681780265*cos(t)+1.*0.35393406681780265*sin(t)+3.715306198104689}
	\parametricplot[linewidth=0.8pt]{3.5397602144799913}{5.044968730984081}{1.*0.7394757223657821*cos(t)+0.*0.7394757223657821*sin(t)+7.127969499924894|0.*0.7394757223657821*cos(t)+1.*0.7394757223657821*sin(t)+4.015621785682471}
	\parametricplot[linewidth=0.8pt]{0.45843591282691365}{1.1506056541748808}{1.*0.9649778905080494*cos(t)+0.*0.9649778905080494*sin(t)+6.482854298407436|0.*0.9649778905080494*cos(t)+1.*0.9649778905080494*sin(t)+3.1525669947261172}
	\psline[linewidth=0.8pt](5.250763258001653,4.215672783897377)(5.718082577653644,4.215672783897378)
	\psline[linewidth=0.8pt](5.861350057781157,4.215672783897378)(7.276560404378989,4.215672783897377)
	\psline[linewidth=0.8pt](7.4198278845065015,4.215672783897377)(8.23256041159776,4.215672783897377)
	\psline[linewidth=0.8pt,linecolor=red](5.7897163177173985,3.028799957058603)(5.7897163177173985,2.227117008155029)
	\psline[linewidth=0.8pt,linecolor=red](7.348194144442743,3.028799957058602)(7.348194144442743,2.227117008155028)
	\psline[linewidth=0.8pt](5.973555288270382,4.985574415114078)(6.055658461648811,5.0531773034462235)
	\psline[linewidth=0.8pt](6.055658461648811,5.0531773034462235)(5.949869841797671,5.042229723821125)
	\psline[linewidth=0.8pt](5.958483004458589,3.6366050139586417)(6.02697826001612,3.703712237432131)
	\psline[linewidth=0.8pt](6.02697826001612,3.703712237432131)(5.932421377748845,3.6877756288811985)
	\psline[linewidth=0.8pt,linecolor=red](5.7897163177173985,5.966145079751118)(5.7897163177173985,6.1734378516285355)
	\psline[linewidth=0.8pt,linecolor=red](7.348194144442746,6.007590970387144)(7.348194144442746,6.214883742264562)
	\psline[linewidth=0.8pt]{->}(0.,7.532357133936048)(3.868811445485548,7.532357133936048)
	\psline[linewidth=0.8pt]{->}(3.868811445485548,7.532357133936048)(0.,7.532357133936048)
	\psline[linewidth=0.8pt](4.480477932711423,3.210771026470039)(3.806023814636113,5.633481213766879)
	\psline[linewidth=0.8pt](3.806023814636113,5.633481213766879)(4.576309139926343,6.638382971194217)
	\parametricplot[linewidth=0.8pt]{1.8293691007221915}{3.2706397293898215}{1.*0.26128939657191*cos(t)+0.*0.26128939657191*sin(t)+4.892243295439127|0.*0.26128939657191*cos(t)+1.*0.26128939657191*sin(t)+4.674075984483759}
	\parametricplot[linewidth=0.8pt]{4.498554911872299}{6.155893305355071}{1.*0.42704020139799914*cos(t)+0.*0.42704020139799914*sin(t)+4.552950742029711|0.*0.42704020139799914*cos(t)+1.*0.42704020139799914*sin(t)+4.741830220239119}
	\parametricplot[linewidth=0.8pt]{2.9320834961138402}{4.359412311086289}{1.*0.4147888324807016*cos(t)+0.*0.4147888324807016*sin(t)+4.6057186939789405|0.*0.4147888324807016*cos(t)+1.*0.4147888324807016*sin(t)+4.713732295111519}
	\parametricplot[linewidth=0.8pt]{2.145577323512249}{2.8077799445123093}{1.*0.9713077458815136*cos(t)+0.*0.9713077458815136*sin(t)+5.117691547180983|0.*0.9713077458815136*cos(t)+1.*0.9713077458815136*sin(t)+4.481753300972342}
	\psline[linewidth=0.8pt,linecolor=green](3.2737751765523413,4.958014368813275)(3.8155169824366526,5.097567892403155)
	\psline[linewidth=0.8pt,linecolor=green](4.58963912781445,5.296982949360193)(5.707957112671319,5.585063384699601)
	\psline[linewidth=0.8pt,linecolor=green](5.894806056582166,5.633195960436934)(7.786815461530115,6.120580497236117)
	\psline[linewidth=0.8pt,linecolor=green](3.4697826181764713,4.340771678285877)(4.0054700198086195,4.478765578015369)
	\psline[linewidth=0.8pt,linestyle=dotted,linecolor=green](4.58963912781445,5.296982949360193)(4.028099591921847,5.152329500027591)
	\psline[linewidth=0.8pt,linestyle=dotted,linecolor=green](4.2934132021705524,4.552940185283673)(4.633126526730202,4.640450860918803)
	\parametricplot[linewidth=0.8pt]{-0.1256914607471451}{1.7936841201714606}{1.*0.13993360620807846*cos(t)+0.*0.13993360620807846*sin(t)+4.856363216454292|0.*0.13993360620807846*cos(t)+1.*0.13993360620807846*sin(t)+4.790206958112255}
	\psline[linewidth=0.8pt,linecolor=green](4.633126526730202,4.640450860918803)(6.228337634804768,5.051379699312829)
	\psline[linewidth=0.8pt](4.576309139926343,6.638382971194217)(4.900901734336691,5.472412204430727)
	\psline[linewidth=0.8pt](4.944721827294024,5.315005817886624)(5.078791052921769,4.833415046881697)
	\psline[linewidth=0.8pt](5.113776188520828,4.7077447571640265)(5.250763258001653,4.215672783897377)
	\psline[linewidth=0.8pt](4.597416812804654,4.729965698565064)(4.65603974902441,4.785790054421361)
	\psline[linewidth=0.8pt](4.65603974902441,4.785790054421361)(4.6750735912359875,4.707108924482469)
	\begin{scriptsize}
	\psdots[dotsize=2pt 0,dotstyle=*](5.7897163177173985,3.579614158436346)
	\psdots[dotsize=2pt 0,dotstyle=*](7.348194144442746,3.579614158436346)
	\psdots[dotsize=2pt 0,dotstyle=*](7.348194144442746,6.007590970387144)
	\psdots[dotsize=2pt 0,dotstyle=*](5.7897163177173985,4.938390172463857)
	\psdots[dotsize=1pt 0,dotstyle=*](3.868811445485548,7.532357133936048)
	\psdots[dotsize=3pt 0,dotstyle=*,linecolor=darkgray](0.,7.532357133936048)
	\psdots[dotsize=2pt 0,dotstyle=*](4.58963912781445,5.296982949360193)
	\psdots[dotsize=1pt 0,dotstyle=*](-1.971208170456346,6.671785579751239)
	\psdots[dotsize=2pt 0,dotstyle=*](4.633126526730202,4.640450860918803)
	\end{scriptsize}
	\end{pspicture*}
	\caption{Projections of a rail arc as knotoid diagrams.}
\label{figure_projections_of_arcs}
\end{figure}
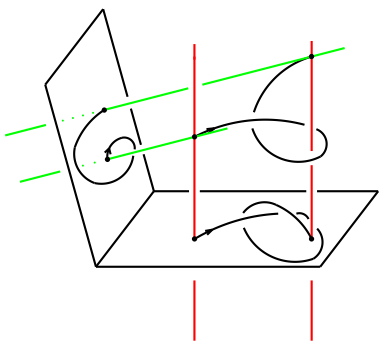

A  line isotopy in \cite{GK1}  between two open space curves is what we call rail isotopy, and G\"ug\"umc\"u and Kauffman have proved (in our new terminology) the following:
\begin{theorem*}
	Two rail arcs are rail isotopic if and only if their knotoid diagram projections in a plane perpendicular to the rails are equivalent.
\end{theorem*}
In \cite{KoLa1} we observed that this lifting of knotoids in 3-space is related to the knot theory of the handlebody of genus~2.  Some other recent works on knotoids include: the theory of braidoids \cite{Gthesis,GL1,GL2}, a study of biquandle coloring invariants  \cite{GN1}, the study of knots that are knotoid equivalent \cite{AdHen} and the construction of double branched covers of knotoids \cite{BBHL}. For a survey on the subject the interested reader may consult \cite{GKL}.

\section{Rail isotopy and rail knotoids} \label{rail_knotoids}

	We are now going to  investigate rail isotopy in $\mathbb{R}^3$ in a new diagrammatic setting by projecting the rail arcs to the plane $\pi$ defined by the rails $\ell_1,\ell_2$, which we can call as \textit{rail plane}. Figure \ref{figure_knotoids_rail_and_planar} conveys a feeling of the difference between the new setting (rightmost figure) and that of the usual planar knotoids (leftmost figure).

	Let us note that projections of rail arcs onto planes can be as bad as projections of knots to planes, but clearly any rail arc is rail isotopic to one with a generic projection to any given plane, meaning it has only a finite number of intersection points with itself and the rails, all of them double points. Thus from now on we can restrict attention to such projections. Keeping track of the over/under crossings at the  double points of the onto the plane $\pi$ of the rails $\ell_1,\ell_2$, we get a generic immersion of the unit interval in the plane with its endpoints on the two rails. This projection is actually just a planar knotoid diagram $c_{pr}$ on $\pi$, whose endpoints are on the rails (the leg on $\ell_1$ and the head on $\ell_2$).

\begin{figure}[!h]
	\centering
	\psset{xunit=1.0cm,yunit=1.0cm,algebraic=true,dimen=middle,dotstyle=o,dotsize=5pt 0,linewidth=1.6pt,arrowsize=3pt 2,arrowinset=0.25}
	\begin{pspicture*}(0.8597242462018629,1.8086164125766302)(11.482376605887998,6.307935331842192)
	\psaxes[labelFontSize=\scriptstyle,xAxis=true,yAxis=true,Dx=0.5,Dy=0.5,ticksize=-2pt 0,subticks=2]{->}(0,0)(0.8597242462018629,1.8086164125766302)(11.482376605887998,6.307935331842192)
	\psline[linewidth=0.8pt](4.82466144914252,3.098595068780185)(5.250763258001653,4.215672783897377)
	\psline[linewidth=0.8pt](7.807374111156448,3.098595068780185)(8.23347592001558,4.215672783897377)
	\psline[linewidth=0.8pt](4.82466144914252,3.098595068780185)(7.807374111156448,3.098595068780185)
	\psline[linewidth=0.8pt,linecolor=red](5.7897163177173985,3.579614158436346)(5.7897163177173985,6.007590970387144)
	\psline[linewidth=0.8pt,linecolor=red](7.348194144442746,3.579614158436346)(7.348194144442746,4.56234264157902)
	\psline[linewidth=0.8pt,linecolor=red](7.348194144442746,4.750538252321563)(7.348194144442746,6.007590970387144)
	\parametricplot[linewidth=0.8pt]{1.3124896428446133}{2.0491738013646073}{1.*2.0441623491706924*cos(t)+0.*2.0441623491706924*sin(t)+6.730724779917389|0.*2.0441623491706924*cos(t)+1.*2.0441623491706924*sin(t)+3.123699349255933}
	\parametricplot[linewidth=0.8pt]{-0.897963385319354}{1.212895220174818}{1.*0.21069737658829024*cos(t)+0.*0.21069737658829024*sin(t)+7.336666585591969|0.*0.21069737658829024*cos(t)+1.*0.21069737658829024*sin(t)+4.843622530429749}
	\parametricplot[linewidth=0.8pt]{3.456529644292434}{5.190203848893322}{1.*0.643976505465421*cos(t)+0.*0.643976505465421*sin(t)+7.171847961037518|0.*0.643976505465421*cos(t)+1.*0.643976505465421*sin(t)+5.250697016198361}
	\parametricplot[linewidth=0.8pt]{1.727268012851006}{2.949250896956655}{1.*0.9302669017999806*cos(t)+0.*0.9302669017999806*sin(t)+7.493161332981229|0.*0.9302669017999806*cos(t)+1.*0.9302669017999806*sin(t)+5.088688895717715}
	\psline[linewidth=0.8pt](5.250763258001653,4.215672783897377)(5.718082577653644,4.215672783897378)
	\psline[linewidth=0.8pt](5.861350057781157,4.215672783897378)(7.276560404378989,4.215672783897378)
	\psline[linewidth=0.8pt](7.4198278845065015,4.215672783897378)(8.23347592001558,4.215672783897377)
	\psline[linewidth=0.8pt,linecolor=red](5.789716317717398,2.9373759900107856)(5.7897163177173985,2.227117008155029)
	\psline[linewidth=0.8pt,linecolor=red](7.3481941444427425,2.9373759900107856)(7.348194144442743,2.227117008155029)
	\psline[linewidth=0.8pt](5.973555288270382,4.985574415114078)(6.055658461648811,5.0531773034462235)
	\psline[linewidth=0.8pt](6.055658461648811,5.0531773034462235)(5.949869841797671,5.042229723821125)
	\psline[linewidth=0.8pt,linecolor=red](5.7897163177173985,5.966145079751118)(5.7897163177173985,6.1734378516285355)
	\psline[linewidth=0.8pt,linecolor=red](7.348194144442746,6.007590970387144)(7.348194144442746,6.214883742264562)
	\psline[linewidth=0.8pt]{->}(0.,7.532357133936048)(3.868811445485548,7.532357133936048)
	\parametricplot[linewidth=0.8pt]{1.312489642844614}{2.049173801364607}{1.*2.0441623491706937*cos(t)+0.*2.0441623491706937*sin(t)+10.599536225402938|0.*2.0441623491706937*cos(t)+1.*2.0441623491706937*sin(t)+3.123699349255931}
	\parametricplot[linewidth=0.8pt]{-0.8979633853193576}{1.21289522017483}{1.*0.21069737658828933*cos(t)+0.*0.21069737658828933*sin(t)+11.205478031077519|0.*0.21069737658828933*cos(t)+1.*0.21069737658828933*sin(t)+4.843622530429749}
	\parametricplot[linewidth=0.8pt]{3.4565296442924307}{5.190203848893328}{1.*0.6439765054654198*cos(t)+0.*0.6439765054654198*sin(t)+11.040659406523064|0.*0.6439765054654198*cos(t)+1.*0.6439765054654198*sin(t)+5.250697016198358}
	\parametricplot[linewidth=0.8pt]{1.7272680128510032}{2.9492508969566544}{1.*0.9302669017999797*cos(t)+0.*0.9302669017999797*sin(t)+11.361972778466775|0.*0.9302669017999797*cos(t)+1.*0.9302669017999797*sin(t)+5.088688895717715}
	\psline[linewidth=0.8pt,linecolor=red](11.217005589928295,3.579614158436346)(11.217005589928295,4.56234264157902)
	\psline[linewidth=0.8pt,linecolor=red](11.217005589928295,4.750538252321563)(11.217005589928295,6.007590970387144)
	\psline[linewidth=0.8pt](9.84236673375593,4.985574415114078)(9.924469907134359,5.0531773034462235)
	\psline[linewidth=0.8pt](9.924469907134359,5.0531773034462235)(9.81868128728322,5.042229723821125)
	\psline[linewidth=0.8pt,linecolor=red](11.217005589928295,6.007590970387144)(11.217005589928295,6.214883742264562)
	\psline[linewidth=0.8pt]{->}(3.868811445485548,7.532357133936048)(0.,7.532357133936048)
	\parametricplot[linewidth=0.8pt]{-1.0136908299043377}{0.7595360487821827}{1.*0.2613356941536533*cos(t)+0.*0.2613356941536533*sin(t)+3.3624074994604447|0.*0.2613356941536533*cos(t)+1.*0.2613356941536533*sin(t)+3.5384858157264762}
	\parametricplot[linewidth=0.8pt]{1.6219613144402378}{2.1081175146628013}{1.*2.417377407026493*cos(t)+0.*2.417377407026493*sin(t)+3.1582068857888794|0.*2.417377407026493*cos(t)+1.*2.417377407026493*sin(t)+1.5028866232864277}
	\parametricplot[linewidth=0.8pt]{1.1181887922759253}{2.5037458770342216}{1.*0.3539340668178012*cos(t)+0.*0.3539340668178012*sin(t)+2.8529110567687828|0.*0.3539340668178012*cos(t)+1.*0.3539340668178012*sin(t)+3.715306198104691}
	\parametricplot[linewidth=0.8pt]{3.5397602144799905}{5.044968730984079}{1.*0.7394757223657815*cos(t)+0.*0.7394757223657815*sin(t)+3.259158054439347|0.*0.7394757223657815*cos(t)+1.*0.7394757223657815*sin(t)+4.015621785682471}
	\parametricplot[linewidth=0.8pt]{0.4584359128269141}{1.1506056541748801}{1.*0.9649778905080505*cos(t)+0.*0.9649778905080505*sin(t)+2.6140428529218873|0.*0.9649778905080505*cos(t)+1.*0.9649778905080505*sin(t)+3.1525669947261163}
	\psline[linewidth=0.8pt](2.0896715589730412,3.6366050139586417)(2.1581668145305724,3.703712237432131)
	\psline[linewidth=0.8pt](2.1581668145305724,3.703712237432131)(2.0636099322632973,3.6877756288811985)
	\psline[linewidth=0.8pt,linecolor=red](9.658527763202947,6.1734378516285355)(9.658527763202947,3.579614158436346)
	\parametricplot[linewidth=0.8pt]{0.7145938139315864}{1.1770747742420844}{1.*0.7407898062604519*cos(t)+0.*0.7407898062604519*sin(t)+2.992353926740024|0.*0.7407898062604519*cos(t)+1.*0.7407898062604519*sin(t)+3.2329902236120276}
	\psline[linewidth=0.8pt](1.381951812516105,4.215672783897377)(0.9558500036569724,3.098595068780185)
	\psline[linewidth=0.8pt](0.9558500036569724,3.098595068780185)(3.9385626656709,3.098595068780185)
	\psline[linewidth=0.8pt](3.9385626656709,3.098595068780185)(4.3646644745300325,4.215672783897377)
	\psline[linewidth=0.8pt](4.3646644745300325,4.215672783897377)(1.381951812516105,4.215672783897377)
	\begin{scriptsize}
	\psdots[dotsize=2pt 0,dotstyle=*](7.348194144442746,6.007590970387144)
	\psdots[dotsize=2pt 0,dotstyle=*](5.7897163177173985,4.938390172463857)
	\psdots[dotsize=1pt 0,dotstyle=*](3.868811445485548,7.532357133936048)
	\psdots[dotsize=3pt 0,dotstyle=*,linecolor=darkgray](0.,7.532357133936048)
	\psdots[dotsize=2pt 0,dotstyle=*](9.658527763202947,4.938390172463857)
	\psdots[dotsize=2pt 0,dotstyle=*](11.217005589928295,6.007590970387144)
	\psdots[dotsize=2pt 0,dotstyle=*](11.217005589928295,6.007590970387144)
	\psdots[dotsize=2pt 0,dotstyle=*](11.217005589928295,6.007590970387144)
	\psdots[dotsize=2pt 0,dotstyle=*](1.9209048722318505,3.579614158436346)
	\psdots[dotsize=2pt 0,dotstyle=*](3.479382698957198,3.579614158436346)
	\psdots[dotsize=2pt 0,dotstyle=*](3.479382698957198,3.579614158436346)
	\end{scriptsize}
	\rput[tl](5.72,2.13){$\ell_1$}
	\rput[tl](7.27,2.13){$\ell_2$}
	\rput[tl](9.6,3.5){$\ell_1$}
	\rput[tl](11.118526357978949,3.5){$\ell_2$}
	\end{pspicture*}
	\caption{In the middle we see a rail arc in space and a plane perpendicular to the rails, whereas on the left and right we see  the knotoid diagram projections of this rail arc on the perpendicular plane, and on the plane of the rails respectively.}
	\label{figure_knotoids_rail_and_planar}
\end{figure}
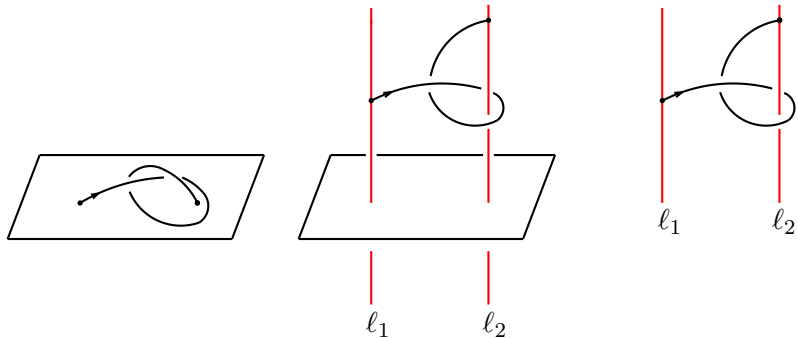

So first we define:

\begin{definition} \rm
	 A \textit{planar rail knotoid diagram} or just \textit{rail knotoid diagram} is an immersion of the unit interval in the rail plane $\pi$ of the rails with only a finite number of transversal intersection points with itself and the rails.  All intersection points,  except  for the endpoints, are double points with additional over/under data.  The endpoints, the first one on $\ell_1$ and the second one on $\ell_2$,  are trivalent.
	 
	 Two rail knotoid diagrams on $\pi$ are \textit{rail equivalent} whenever one can be obtained from the other via a finite sequence of the \textit{rail knotoid equivalence moves} defined locally in Figure \ref{figure_rail_moves}, which include the usual Reidemester moves $\Omega_1, \ \Omega_2, \ \Omega_3$ and their versions where parts of the rails are involved, along with slide \textit{slide moves} which involve the rails and the endpoints of the rail knotoids, and finally some planar rail isotopy moves or just\textit{ planar isotopies} of $\pi$.
	 
	 Clearly, equivalence between rail knotoid diagrams as defined is indeed an equivalence relation  in the set of all planar rail knotoid diagrams. We call the equivalence classes  as \textit{planar rail knotoids} or simply as rail knotoids.
	\end{definition}

	  Although rather obvious, it is no harm to call upon some more terminology: if a move $\delta_1$ is applied to a diagram $c_1$ resulting to a diagram $c_2$, and $\delta_2$ is applied to $c_2$ resulting to $c_3$, we call $\delta_2$ as successive to $\delta_1$, we say that they compose and write $\delta_2 \circ \delta_1$, and we also say that their composition is applied to $c_1$ resulting to $c_3$. Similarly we define the notion of a sequence of successive moves (or just a sequence of moves), and of their composition.

  We prefer the planar isotopies of Figure \ref{figure_rail_moves} instead of plane isotopies in $\pi$, so that the above equivalence relation which we'll try  to investigate can get an entirely diagrammatic setting.

\begin{figure}[!h]
	\centering
	\newrgbcolor{eqeqeq}{0.8784313725490196 0.8784313725490196 0.8784313725490196}
	\psset{xunit=1.0cm,yunit=1.0cm,algebraic=true,dimen=middle,dotstyle=o,dotsize=5pt 0,linewidth=1.6pt,arrowsize=3pt 2,arrowinset=0.25}
	\begin{pspicture*}(-2,-6.818205735001157)(12.5,3.1962160687602714)
	\pscircle[linewidth=0.4pt,linestyle=dotted,linecolor=eqeqeq,fillcolor=eqeqeq,fillstyle=solid,opacity=0.45](1.2019127851903506,2.4540743605874535){0.6055794355355599}
	\pscircle[linewidth=0.4pt,linestyle=dotted,linecolor=eqeqeq,fillcolor=eqeqeq,fillstyle=solid,opacity=0.45](2.800876210249418,2.4540743605874535){0.6055794355355599}
	\pscircle[linewidth=0.4pt,linestyle=dotted,linecolor=eqeqeq,fillcolor=eqeqeq,fillstyle=solid,opacity=0.45](4.399839635308485,2.4540743605874535){0.6055794355355599}
	\pscircle[linewidth=0.4pt,linestyle=dotted,linecolor=eqeqeq,fillcolor=eqeqeq,fillstyle=solid,opacity=0.45](5.998803060367552,2.4540743605874535){0.6055794355355599}
	\pscircle[linewidth=0.4pt,linestyle=dotted,linecolor=eqeqeq,fillcolor=eqeqeq,fillstyle=solid,opacity=0.45](7.59776648542662,2.4540743605874535){0.6055794355355599}
	\pscircle[linewidth=0.4pt,linestyle=dotted,linecolor=eqeqeq,fillcolor=eqeqeq,fillstyle=solid,opacity=0.45](9.196729910485686,2.4540743605874535){0.6055794355355599}
	\pscircle[linewidth=0.4pt,linestyle=dotted,linecolor=eqeqeq,fillcolor=eqeqeq,fillstyle=solid,opacity=0.45](1.9283105506386062,0.7357140552259755){0.6055794355355599}
	\pscircle[linewidth=0.4pt,linestyle=dotted,linecolor=eqeqeq,fillcolor=eqeqeq,fillstyle=solid,opacity=0.45](3.527273975697674,0.7357140552259755){0.6055794355355599}
	\pscircle[linewidth=0.4pt,linestyle=dotted,linecolor=eqeqeq,fillcolor=eqeqeq,fillstyle=solid,opacity=0.45](6.794124581268798,-4.456165556353921){0.6055794355355599}
	\pscircle[linewidth=0.4pt,linestyle=dotted,linecolor=eqeqeq,fillcolor=eqeqeq,fillstyle=solid,opacity=0.45](8.393088006327865,-4.456165556353921){0.6055794355355599}
	\pscircle[linewidth=0.4pt,linestyle=dotted,linecolor=eqeqeq,fillcolor=eqeqeq,fillstyle=solid,opacity=0.45](6.959618287821182,0.7573961006442667){0.6055794355355599}
	\pscircle[linewidth=0.4pt,linestyle=dotted,linecolor=eqeqeq,fillcolor=eqeqeq,fillstyle=solid,opacity=0.45](8.55858171288025,0.7573961006442667){0.6055794355355599}
	\pscircle[linewidth=0.4pt,linestyle=dotted,linecolor=eqeqeq,fillcolor=eqeqeq,fillstyle=solid,opacity=0.45](2.575932038113259,-0.9261513686166785){0.6055794355355599}
	\pscircle[linewidth=0.4pt,linestyle=dotted,linecolor=eqeqeq,fillcolor=eqeqeq,fillstyle=solid,opacity=0.45](4.1748954631723265,-0.9261513686166785){0.6055794355355599}
	\pscircle[linewidth=0.4pt,linestyle=dotted,linecolor=eqeqeq,fillcolor=eqeqeq,fillstyle=solid,opacity=0.45](6.0724903698992865,-0.793516348055761){0.6055794355355599}
	\pscircle[linewidth=0.4pt,linestyle=dotted,linecolor=eqeqeq,fillcolor=eqeqeq,fillstyle=solid,opacity=0.45](7.671453794958354,-0.793516348055761){0.6055794355355599}
	\pscircle[linewidth=0.4pt,linestyle=dotted,linecolor=eqeqeq,fillcolor=eqeqeq,fillstyle=solid,opacity=0.45](6.054168044256003,-2.7122237153944893){0.6055794355355599}
	\pscircle[linewidth=0.4pt,linestyle=dotted,linecolor=eqeqeq,fillcolor=eqeqeq,fillstyle=solid,opacity=0.45](7.6531314693150705,-2.7122237153944893){0.6055794355355599}
	\pscircle[linewidth=0.4pt,linestyle=dotted,linecolor=eqeqeq,fillcolor=eqeqeq,fillstyle=solid,opacity=0.45](9.252094894374137,-2.7122237153944893){0.6055794355355599}
	\pscircle[linewidth=0.4pt,linestyle=dotted,linecolor=eqeqeq,fillcolor=eqeqeq,fillstyle=solid,opacity=0.45](4.486114096850756,-6.012513502764692){0.6055794355355599}
	\pscircle[linewidth=0.4pt,linestyle=dotted,linecolor=eqeqeq,fillcolor=eqeqeq,fillstyle=solid,opacity=0.45](6.085077521909823,-6.012513502764692){0.6055794355355599}
	\pscircle[linewidth=0.4pt,linestyle=dotted,linecolor=eqeqeq,fillcolor=eqeqeq,fillstyle=solid,opacity=0.45](0.9781930543758244,-2.8406259353567025){0.6055794355355599}
	\pscircle[linewidth=0.4pt,linestyle=dotted,linecolor=eqeqeq,fillcolor=eqeqeq,fillstyle=solid,opacity=0.45](2.577156479434892,-2.8406259353567025){0.6055794355355599}
	\pscircle[linewidth=0.8pt,linecolor=eqeqeq,fillcolor=eqeqeq,fillstyle=solid,opacity=0.45](4.176119904493959,-2.8406259353567025){0.6055794355355599}
	\pscircle[linewidth=0.4pt,linestyle=dotted,linecolor=eqeqeq,fillcolor=eqeqeq,fillstyle=solid,opacity=0.45](1.8413084330996161,-4.496841871047476){0.6055794355355599}
	\pscircle[linewidth=0.4pt,linestyle=dotted,linecolor=eqeqeq,fillcolor=eqeqeq,fillstyle=solid,opacity=0.45](3.4402718581586837,-4.496841871047476){0.6055794355355599}
	\psline[linewidth=0.4pt](1.3227311151030459,0.7357140552259755)(1.9283105506386062,1.15722537516398)
	\psline[linewidth=0.4pt](1.9283105506386062,1.15722537516398)(2.5338899861741666,0.7357140552259755)
	\psline[linewidth=0.4pt](1.5675446466473142,1.2221032155296498)(1.653132278104357,1.1067128869641871)
	\psline[linewidth=0.4pt](1.7765456109401818,0.9403254687958016)(1.8528228390660708,0.8374875439032694)
	\psline[linewidth=0.4pt](2.0125703983283754,0.6221138615498806)(2.2890764546298983,0.24932489492230125)
	\psline[linewidth=0.4pt](2.2041184799509934,1.1075617982061026)(2.2890764546298983,1.2221032155296498)
	\psline[linewidth=0.4pt](1.5675446466473142,0.24932489492230125)(2.049870149557713,0.899602245853167)
	\psline[linewidth=0.4pt](2.9216945401621133,0.7357140552259755)(3.527273975697674,0.31420273528797105)
	\psline[linewidth=0.4pt](3.527273975697674,0.31420273528797105)(4.132853411233234,0.7357140552259755)
	\psline[linewidth=0.4pt](3.1665080717063816,1.2221032155296498)(3.451786264125138,0.8374875439032694)
	\psline[linewidth=0.4pt](3.385558738975869,0.544651778473185)(3.8880398796889657,1.2221032155296498)
	\psline[linewidth=0.4pt](3.2572598635841747,0.3716776099631304)(3.1665080717063816,0.24932489492230125)
	\psline[linewidth=0.4pt](6.037065243343325,1.849704889073963)(5.91524850033442,2.4693007518466326)
	\psline[linewidth=0.4pt](5.91524850033442,2.4693007518466326)(6.538302508736716,2.72915034636763)
	\psline[linewidth=0.4pt](5.613768576130681,2.632624235269829)(5.896019862401664,2.2601304115062484)
	\psline[linewidth=0.4pt](6.053459312263512,2.052353775781831)(6.185563473335151,1.8780127311659658)
	\psline[linewidth=0.4pt](5.613768576130681,2.632624235269829)(6.042656241063381,2.599646126840949)
	\psline[linewidth=0.4pt](6.337108495589661,2.577005050468609)(6.595548706118908,2.5571330187407826)
	\psline[linewidth=0.4pt](7.514211925393488,2.4693007518466326)(7.636028668402393,1.849704889073963)
	\psline[linewidth=0.4pt](7.514211925393488,2.4693007518466326)(8.137265933795783,2.72915034636763)
	\psline[linewidth=0.4pt](8.811695426248816,2.632624235269829)(9.383490323453286,1.8780127311659658)
	\psline[linewidth=0.4pt](8.811695426248816,2.632624235269829)(9.793475556237045,2.5571330187407826)
	\psline[linewidth=0.4pt](7.784526898394218,1.8780127311659658)(8.194512131177976,2.5571330187407826)
	\psline[linewidth=0.4pt](9.23499209346146,1.849704889073963)(9.200778042426252,2.0237276378890243)
	\psline[linewidth=0.4pt](9.148872881110874,2.287732588384257)(9.113175350452554,2.4693007518466326)
	\psline[linewidth=0.4pt](9.113175350452554,2.4693007518466326)(9.284705597001457,2.540838795760539)
	\psline[linewidth=0.4pt](9.544873727518096,2.6493439691649994)(9.73622935885485,2.72915034636763)
	\psline[linewidth=0.4pt](2.800876210249418,1.8484949250518938)(2.3914611766415104,2.4540743605874535)
	\psline[linewidth=0.4pt](4.399839635308485,1.8484949250518938)(4.8092546689163935,2.4540743605874535)
	\psline[linewidth=0.4pt](4.8092546689163935,2.4540743605874535)(3.990424601700578,2.4540743605874535)
	\psline[linewidth=0.4pt](1.2019127851903506,1.8484949250518938)(1.6113278187982583,2.4540743605874535)
	\psline[linewidth=0.4pt](0.792497751582443,2.4540743605874535)(1.0685014158182162,2.4540743605874535)
	\psline[linewidth=0.4pt](1.335324154562485,2.4540743605874535)(1.6113278187982583,2.4540743605874535)
	\psline[linewidth=0.4pt](1.2019127851903506,3.0596537961230132)(1.2019127851903506,1.8484949250518938)
	\psline[linewidth=0.4pt](2.800876210249418,1.8484949250518938)(2.800876210249418,3.0596537961230132)
	\psline[linewidth=0.4pt](4.399839635308485,1.8484949250518938)(4.399839635308486,2.2932918511635787)
	\psline[linewidth=0.4pt](4.399839635308486,2.6148568700113293)(4.399839635308485,3.0596537961230132)
	\psline[linewidth=0.4pt](3.8880398796889657,0.24932489492230125)(3.802755135622831,0.3643068668717411)
	\psline[linewidth=0.4pt](3.6890390457067843,0.5176203393350868)(3.6115338233874428,0.6221138615498806)
	\psline[linewidth=0.4pt](6.794124581268798,-4.456165556353921)(7.381952019835337,-4.31061365766528)
	\psline[linewidth=0.8pt,linecolor=red](6.794124581268798,-3.8505861208183596)(6.794124581268798,-5.061744991889482)
	\psline[linewidth=0.8pt,linecolor=red](8.393088006327865,-5.061744991889482)(8.393088006327865,-3.8505861208183596)
	\psline[linewidth=0.4pt](8.980915444894404,-4.31061365766528)(8.393088006327865,-4.005690626324728)
	\psline[linewidth=0.4pt](6.297386629642253,-4.109787352424411)(6.691510450497595,-4.259389112322574)
	\psline[linewidth=0.4pt](6.867134831921718,-4.3260527219594005)(6.967081380735694,-4.363990493814622)
	\psline[linewidth=0.4pt](7.146907137566169,-4.4322488640863655)(7.395588872874564,-4.526643628594447)
	\psline[linewidth=0.4pt](7.896350054701321,-4.109787352424411)(8.275271171934888,-4.253618460825053)
	\psline[linewidth=0.4pt](8.498989755297437,-4.338537696917726)(8.99455229793363,-4.526643628594447)
	\psline[linewidth=0.4pt](6.96042131165065,1.3629750037578878)(6.958815263991713,0.1518171975306455)
	\psline[linewidth=0.4pt](8.559384736709717,1.3629750037578878)(8.557778689050782,0.1518171975306455)
	\psline[linewidth=0.4pt](2.142380331405235,-1.3489508132167476)(3.1814807121525703,-0.9200475900492665)
	\psline[linewidth=0.4pt](3.1814807121525703,-0.9200475900492665)(2.5177827988334274,-0.32337021536556343)
	\psline[linewidth=0.4pt](4.116746223892495,-0.32337021536556343)(3.7413437564643024,-1.3489508132167476)
	\psline[linewidth=0.4pt](6.073293393728754,-0.18793744494213982)(6.071687346069819,-1.3990952511693822)
	\psline[linewidth=0.4pt](5.595980222008932,-1.1672335294925795)(5.965914461075204,-0.8771016576890154)
	\psline[linewidth=0.4pt](8.244192651267518,-0.9902310023579743)(8.147963942848708,-0.41979916661894245)
	\psline[linewidth=0.4pt](7.6722568187878215,-0.18793744494213982)(7.670650771128886,-1.3990952511693822)
	\psline[linewidth=0.4pt](7.1949436470679995,-1.1672335294925795)(7.526502790707905,-1.1113013283811553)
	\psline[linewidth=0.4pt](7.861773763104863,-1.0547429622053284)(8.244192651267518,-0.9902310023579743)
	\psline[linewidth=0.4pt](6.24620752059624,-0.6572735330535524)(6.549000517789641,-0.41979916661894245)
	\psline[linewidth=0.8pt,linecolor=red](6.054168044256003,-2.1066442798589353)(6.054168044256003,-3.317803150930043)
	\psline[linewidth=0.8pt,linecolor=red](7.6531314693150705,-2.1066442798589353)(7.6531314693150705,-3.317803150930043)
	\psline[linewidth=0.4pt](8.944977353685925,-2.7108432494062917)(9.632029257101959,-3.1837916667675366)
	\psline[linewidth=0.4pt](8.944977353685925,-2.7108432494062917)(9.857488078611562,-2.7272418784792833)
	\psline[linewidth=0.8pt,linecolor=red](9.252094894374137,-2.850204289967203)(9.252094894374137,-2.75719931436935)
	\psline[linewidth=0.8pt,linecolor=red](9.252094894374137,-2.5786861646899446)(9.252094894374137,-2.1066442798589353)
	\psline[linewidth=0.4pt](6.188535028924575,-2.7187771188667105)(6.659561228493426,-2.7272418784792833)
	\psline[linewidth=0.4pt](6.165659414816347,-2.9990025880005913)(6.434102406983824,-3.1837916667675366)
	\psline[linewidth=0.4pt](8.03306583204289,-3.1837916667675366)(7.842169443609521,-2.877307753512482)
	\psline[linewidth=0.4pt](7.842169443609521,-2.877307753512482)(8.258524653552493,-2.7272418784792833)
	\psline[linewidth=0.4pt](5.74705050356779,-2.7108432494062917)(5.962823054030592,-2.7147208739892372)
	\psline[linewidth=0.4pt](5.74705050356779,-2.7108432494062917)(5.940131001103015,-2.8437547736718916)
	\psline[linewidth=0.8pt,linecolor=red](9.252094894374137,-3.317803150930043)(9.252094894374137,-3.095282231407122)
	\psline[linewidth=0.8pt,linecolor=red](4.486917120680223,-5.406934599651071)(4.485311073021288,-6.618092405878313)
	\psline[linewidth=0.4pt](4.009603948960401,-6.386230684201511)(4.3795381880266735,-6.096098812397946)
	\psline[linewidth=0.4pt](6.657816378218986,-6.209228157066905)(6.561587669800177,-5.638796321327874)
	\psline[linewidth=0.8pt,linecolor=red](6.085880545739291,-5.406934599651071)(6.084274498080355,-6.618092405878313)
	\psline[linewidth=0.4pt](5.608567374019469,-6.386230684201511)(5.940126517659374,-6.330298483090086)
	\psline[linewidth=0.4pt](6.275397490056332,-6.27374011691426)(6.657816378218986,-6.209228157066905)
	\psline[linewidth=0.4pt](4.659831247547709,-5.876270687762483)(4.96262424474111,-5.638796321327874)
	\psline[linewidth=0.4pt](4.585534938101867,-2.8406259353567025)(3.7667048708860515,-2.8406259353567025)
	\psline[linewidth=0.4pt](0.5687780207679167,-2.8406259353567025)(0.84478168500369,-2.8406259353567025)
	\psline[linewidth=0.4pt](1.1116044237479588,-2.8406259353567025)(1.387608087983732,-2.8406259353567025)
	\psline[linewidth=0.8pt,linecolor=red](0.9781930543758244,-2.235046499821143)(0.9781930543758244,-3.446205370892262)
	\psline[linewidth=0.8pt,linecolor=red](2.577156479434892,-3.446205370892262)(2.577156479434892,-2.235046499821143)
	\psline[linewidth=0.8pt,linecolor=red](4.176119904493959,-3.446205370892262)(4.17611990449396,-3.0014084447805773)
	\psline[linewidth=0.8pt,linecolor=red](4.17611990449396,-2.6798434259328268)(4.176119904493959,-2.235046499821143)
	\psline[linewidth=0.4pt](1.2357289975640557,-4.496841871047476)(1.8413084330996161,-4.075330551109471)
	\psline[linewidth=0.4pt](1.8413084330996161,-4.075330551109471)(2.4468878686351765,-4.496841871047476)
	\psline[linewidth=0.4pt](1.480542529108324,-4.0104527107438015)(1.5661301605653668,-4.125843039309264)
	\psline[linewidth=0.4pt](1.6895434934011917,-4.2922304574776495)(1.7658207215270807,-4.395068382370182)
	\psline[linewidth=0.4pt](1.9255682807893852,-4.61044206472357)(2.202074337090908,-4.98323103135115)
	\psline[linewidth=0.8pt,linecolor=red](2.1171163624120033,-4.124994128067349)(2.202074337090908,-4.0104527107438015)
	\psline[linewidth=0.8pt,linecolor=red](1.480542529108324,-4.98323103135115)(1.962868032018723,-4.332953680420284)
	\psline[linewidth=0.4pt](2.834692422623123,-4.496841871047476)(3.4402718581586837,-4.9183531909854805)
	\psline[linewidth=0.4pt](3.4402718581586837,-4.9183531909854805)(4.045851293694244,-4.496841871047476)
	\psline[linewidth=0.4pt](3.0795059541673915,-4.0104527107438015)(3.364784146586148,-4.395068382370182)
	\psline[linewidth=0.8pt,linecolor=red](3.298556621436879,-4.687904147800266)(3.8010377621499756,-4.0104527107438015)
	\psline[linewidth=0.8pt,linecolor=red](3.1702577460451846,-4.860878316310321)(3.0795059541673915,-4.98323103135115)
	\psline[linewidth=0.4pt](3.8010377621499756,-4.98323103135115)(3.715753018083841,-4.868249059401711)
	\psline[linewidth=0.4pt](3.602036928167794,-4.714935586938364)(3.5245317058484527,-4.61044206472357)
	\begin{scriptsize}
	\psdots[dotsize=2pt 0,dotstyle=*](0.39261628860097386,6.117624065117547)
	\psdots[dotsize=2pt 0,dotstyle=*](1.9915797136600413,6.117624065117547)
	\psdots[dotsize=2pt 0,dotstyle=*](0.39261628860097386,4.306835161430932)
	\psdots[dotsize=2pt 0,dotstyle=*](1.1190140540492295,4.399263759756069)
	\psdots[dotsize=2pt 0,dotstyle=*](1.3227311151030459,0.7357140552259755)
	\psdots[dotsize=2pt 0,dotstyle=*](2.5338899861741666,0.7357140552259755)
	\psdots[dotsize=2pt 0,dotstyle=*](1.9283105506386062,1.15722537516398)
	\psdots[dotsize=2pt 0,dotstyle=*](2.9216945401621133,0.7357140552259755)
	\psdots[dotsize=2pt 0,dotstyle=*](4.132853411233234,0.7357140552259755)
	\psdots[dotsize=2pt 0,dotstyle=*](3.527273975697674,0.31420273528797105)
	\psdots[dotsize=2pt 0,dotstyle=*](6.037065243343325,1.849704889073963)
	\psdots[dotsize=2pt 0,dotstyle=*](6.538302508736716,2.72915034636763)
	\psdots[dotsize=2pt 0,dotstyle=*](5.613768576130681,2.632624235269829)
	\psdots[dotsize=2pt 0,dotstyle=*](5.91524850033442,2.4693007518466326)
	\psdots[dotsize=2pt 0,dotstyle=*](6.185563473335151,1.8780127311659658)
	\psdots[dotsize=2pt 0,dotstyle=*](6.595548706118908,2.5571330187407826)
	\psdots[dotsize=2pt 0,dotstyle=*](7.784526898394218,1.8780127311659658)
	\psdots[dotsize=2pt 0,dotstyle=*](7.636028668402393,1.849704889073963)
	\psdots[dotsize=2pt 0,dotstyle=*](8.194512131177976,2.5571330187407826)
	\psdots[dotsize=2pt 0,dotstyle=*](8.137265933795783,2.72915034636763)
	\psdots[dotsize=2pt 0,dotstyle=*](7.514211925393488,2.4693007518466326)
	\psdots[dotsize=2pt 0,dotstyle=*](9.23499209346146,1.849704889073963)
	\psdots[dotsize=2pt 0,dotstyle=*](9.383490323453286,1.8780127311659658)
	\psdots[dotsize=2pt 0,dotstyle=*](9.793475556237045,2.5571330187407826)
	\psdots[dotsize=2pt 0,dotstyle=*](9.73622935885485,2.72915034636763)
	\psdots[dotsize=2pt 0,dotstyle=*](9.113175350452554,2.4693007518466326)
	\psdots[dotsize=2pt 0,dotstyle=*](8.811695426248816,2.632624235269829)
	\psdots[dotsize=2pt 0,dotstyle=*](2.800876210249418,1.8484949250518938)
	\psdots[dotsize=2pt 0,dotstyle=*](2.3914611766415104,2.4540743605874535)
	\psdots[dotsize=2pt 0,dotstyle=*](3.990424601700578,2.4540743605874535)
	\psdots[dotsize=2pt 0,dotstyle=*](4.399839635308485,1.8484949250518938)
	\psdots[dotsize=2pt 0,dotstyle=*](4.8092546689163935,2.4540743605874535)
	\psdots[dotsize=2pt 0,dotstyle=*](0.792497751582443,2.4540743605874535)
	\psdots[dotsize=2pt 0,dotstyle=*](1.6113278187982583,2.4540743605874535)
	\psdots[dotsize=2pt 0,dotstyle=*](1.2019127851903506,1.8484949250518938)
	\psdots[dotsize=1pt 0,dotstyle=*](0.39261628860097386,4.416553163732603)
	\psdots[dotsize=1pt 0,dotstyle=*](20.062072706363153,1.6973975772734986)
	\psdots[dotsize=1pt 0,dotstyle=*](15.841905056907082,1.2397890369710258)
	\psdots[dotsize=1pt 0,dotstyle=*](15.910828812360071,-3.9520905746088704)
	\psdots[dotsize=2pt 0,dotstyle=*](6.794124581268798,-4.456165556353921)
	\psdots[dotsize=2pt 0,dotstyle=*](7.381952019835337,-4.31061365766528)
	\psdots[dotsize=2pt 0,dotstyle=*](8.980915444894404,-4.31061365766528)
	\psdots[dotsize=2pt 0,dotstyle=*](8.393088006327865,-4.005690626324728)
	\psdots[dotsize=1pt 0,dotstyle=*](16.683678791537552,-1.0041876569566486)
	\psdots[dotsize=2pt 0,dotstyle=*](8.558589445174942,0.7632272034920484)
	\psdots[dotsize=1pt 0,dotstyle=*](15.480337815186614,-3.2911004608386305)
	\psdots[dotsize=2pt 0,dotstyle=*](2.142380331405235,-1.3489508132167476)
	\psdots[dotsize=2pt 0,dotstyle=*](3.1814807121525703,-0.9200475900492665)
	\psdots[dotsize=2pt 0,dotstyle=*](3.1814807121525703,-0.9200475900492665)
	\psdots[dotsize=2pt 0,dotstyle=*](2.5177827988334274,-0.32337021536556343)
	\psdots[dotsize=2pt 0,dotstyle=*](4.116746223892495,-0.32337021536556343)
	\psdots[dotsize=2pt 0,dotstyle=*](3.7413437564643024,-1.3489508132167476)
	\psdots[dotsize=1pt 0,dotstyle=*](16.88593046737494,-0.05959323796191729)
	\psdots[dotsize=2pt 0,dotstyle=*](5.595980222008932,-1.1672335294925795)
	\psdots[dotsize=2pt 0,dotstyle=*](8.244192651267518,-0.9902310023579743)
	\psdots[dotsize=2pt 0,dotstyle=*](8.147963942848708,-0.41979916661894245)
	\psdots[dotsize=2pt 0,dotstyle=*](7.1949436470679995,-1.1672335294925795)
	\psdots[dotsize=2pt 0,dotstyle=*](8.244192651267518,-0.9902310023579743)
	\psdots[dotsize=2pt 0,dotstyle=*](6.549000517789641,-0.41979916661894245)
	\psdots[dotsize=2pt 0,dotstyle=*](8.244192651267518,-0.9902310023579743)
	\psdots[dotsize=1pt 0,dotstyle=*](13.220542554433694,-3.7091625840779257)
	\psdots[dotsize=1pt 0,dotstyle=*](13.275907538322144,-1.8640397536731463)
	\psdots[dotsize=2pt 0,dotstyle=*](8.944977353685925,-2.7108432494062917)
	\psdots[dotsize=2pt 0,dotstyle=*](9.632029257101959,-3.1837916667675366)
	\psdots[dotsize=2pt 0,dotstyle=*](8.944977353685925,-2.7108432494062917)
	\psdots[dotsize=2pt 0,dotstyle=*](9.857488078611562,-2.7272418784792833)
	\psdots[dotsize=2pt 0,dotstyle=*](6.659561228493426,-2.7272418784792833)
	\psdots[dotsize=2pt 0,dotstyle=*](6.434102406983824,-3.1837916667675366)
	\psdots[dotsize=2pt 0,dotstyle=*](8.03306583204289,-3.1837916667675366)
	\psdots[dotsize=2pt 0,dotstyle=*](7.842169443609521,-2.877307753512482)
	\psdots[dotsize=2pt 0,dotstyle=*](7.842169443609521,-2.877307753512482)
	\psdots[dotsize=2pt 0,dotstyle=*](8.258524653552493,-2.7272418784792833)
	\psdots[dotsize=2pt 0,dotstyle=*](5.74705050356779,-2.7108432494062917)
	\psdots[dotsize=2pt 0,dotstyle=*](5.74705050356779,-2.7108432494062917)
	\psdots[dotsize=2pt 0,dotstyle=*](5.74705050356779,-2.7108432494062917)
	\psdots[dotsize=1pt 0,dotstyle=*](14.419363940263123,-1.3420270484392156)
	\psdots[dotsize=1pt 0,dotstyle=*](17.849733146086084,-4.56110539738179)
	\psdots[dotsize=2pt 0,dotstyle=*](4.009603948960401,-6.386230684201511)
	\psdots[dotsize=2pt 0,dotstyle=*](6.657816378218986,-6.209228157066905)
	\psdots[dotsize=2pt 0,dotstyle=*](6.561587669800177,-5.638796321327874)
	\psdots[dotsize=2pt 0,dotstyle=*](5.608567374019469,-6.386230684201511)
	\psdots[dotsize=2pt 0,dotstyle=*](6.657816378218986,-6.209228157066905)
	\psdots[dotsize=2pt 0,dotstyle=*](4.96262424474111,-5.638796321327874)
	\psdots[dotsize=2pt 0,dotstyle=*](6.657816378218986,-6.209228157066905)
	\psdots[dotsize=1pt 0,dotstyle=*](12.9900434378369,1.6160449478863925)
	\psdots[dotsize=1pt 0,dotstyle=*](12.766323707022373,-3.6786553480577635)
	\psdots[dotsize=2pt 0,dotstyle=*](2.167741445826984,-2.8406259353567025)
	\psdots[dotsize=2pt 0,dotstyle=*](4.585534938101867,-2.8406259353567025)
	\psdots[dotsize=2pt 0,dotstyle=*](4.585534938101867,-2.8406259353567025)
	\psdots[dotsize=2pt 0,dotstyle=*](3.7667048708860515,-2.8406259353567025)
	\psdots[dotsize=2pt 0,dotstyle=*](0.5687780207679167,-2.8406259353567025)
	\psdots[dotsize=2pt 0,dotstyle=*](1.387608087983732,-2.8406259353567025)
	\psdots[dotsize=1pt 0,dotstyle=*](16.706276744145104,2.436350627539713)
	\psdots[dotsize=1pt 0,dotstyle=*](16.619274626606114,-2.7962052987337382)
	\psdots[dotsize=2pt 0,dotstyle=*](1.2357289975640557,-4.496841871047476)
	\psdots[dotsize=2pt 0,dotstyle=*](1.8413084330996161,-4.075330551109471)
	\psdots[dotsize=2pt 0,dotstyle=*](1.8413084330996161,-4.075330551109471)
	\psdots[dotsize=2pt 0,dotstyle=*](2.4468878686351765,-4.496841871047476)
	\psdots[dotsize=2pt 0,dotstyle=*](2.834692422623123,-4.496841871047476)
	\psdots[dotsize=2pt 0,dotstyle=*](3.4402718581586837,-4.9183531909854805)
	\psdots[dotsize=2pt 0,dotstyle=*](3.4402718581586837,-4.9183531909854805)
	\psdots[dotsize=2pt 0,dotstyle=*](4.045851293694244,-4.496841871047476)
	\psline[linewidth=0.4pt](0.9781930543758244,-3.2252452405218657)(1.387608087983732,-2.8406259353567025)
	\psline[linewidth=0.4pt](2.577156479434892,-3.2252452405218657)(2.167741445826984,-2.8406259353567025)
	\psline[linewidth=0.4pt](4.176119904493959,-3.2252452405218657)(4.585534938101867,-2.8406259353567025)
	\psdots[dotsize=2pt 0,dotstyle=*](0.9781930543758244,-3.2252452405218657)
	\psdots[dotsize=2pt 0,dotstyle=*](2.577156479434892,-3.2252452405218657)
	\psdots[dotsize=2pt 0,dotstyle=*](4.176119904493959,-3.2252452405218657)
	\end{scriptsize}
	\rput[tl](0.1,2.65){$\Omega_1$}
	\rput[tl](9.931152635914795,2.65){$\Omega_2$}
	\rput[tl](0.7,0.94){$\Omega_3$}
	\rput[tl](9.22,0.94){$planar \ isotopy \ 1$}
	\rput[tl](-1,-0.8){$planar \ isotopy \ 2$}
	\rput[tl](8.4,-0.6){$planar \ isotopy \ 3$}
	\rput[tl](-1.4,-2.6){$rail \ \Omega_1$}
	\rput[tl](10.033080593204987,-2.6){$rail \ \Omega_2$}
	\rput[tl](-0.45,-4.3){$rail \ \Omega_3$}
	\rput[tl](9.2,-4.3){$slide \ move$}
	\rput[tl](-0.15,-6){$rail \ planar \ isotopy \ 3$}
	\rput[tl](1.8788440099896417,2.45){$\sim$}
	\rput[tl](3.5,2.45){$\sim$}
	\rput[tl](6.73,2.45){$\sim$}
	\rput[tl](8.3,2.45){$\sim$}
	\rput[tl](2.6,0.9){$\sim$}
	\rput[tl](7.65,0.9){$\sim$}
	\rput[tl](6.75,-0.75){$\sim$}
	\rput[tl](3.23,-0.9){$\sim$}
	\rput[tl](6.74,-2.65){$\sim$}
	\rput[tl](8.385,-2.65){$\sim$}
		\rput[tl](1.65,-2.75){$\sim$}
		\rput[tl](3.25,-2.75){$\sim$}
		\rput[tl](2.46,-4.35){$\sim$}
		\rput[tl](7.43,-4.35){$\sim$}		
	\rput[tl](5.157,-6){$\sim$}
	\normalsize
	\end{pspicture*}
	\caption{The set of rail knotoid equivalence moves, which generate the equivalence in the set of planar rail knotoid diagrams. The moves are split in two parts: those which do not involve the rails, and those which do involve them. Dots denote vertices, and rails get a red colour. The moving part in 
		$\Omega_3$-moves can be on top, bottom or in the middle, and especially in the rail $\Omega_3$ moves, the rail part can be on top or bottom of the other non-moving part. The non-moving part in planar isotopies of type 3 can be either on top or bottom, whereas the non-moving arc in slide moves can only be on top or bottom of all other parts.}
	\label{figure_rail_moves}
\end{figure}
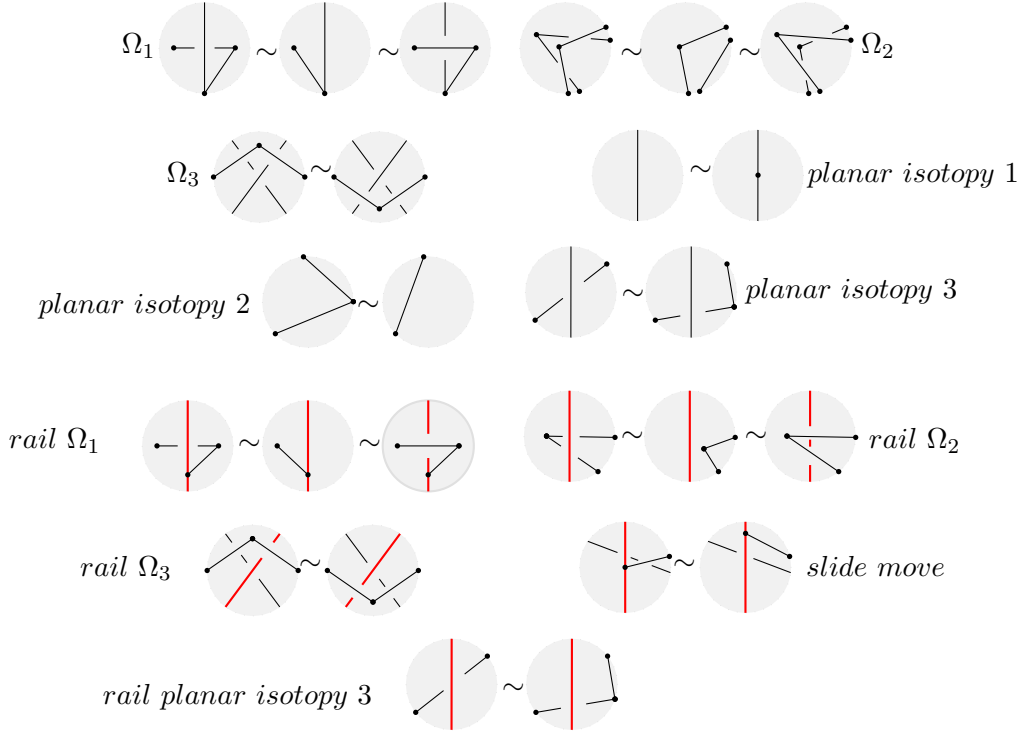

	For a triangle move $\Delta$ between two rail arcs $c_1,c_2$ in $\mathbb{R}^3$, let us denote $\Delta_{pr}$ the projection on $\pi$ of the edges of the triangle $\Delta$, keeping track of over/under data with respect to the rails and to the  projections $c_{1pr},c_{2pr}$ of $c_1,c_2$.
	Let us call $\Delta$ as \textit{nice} whenever $\Delta_{pr}$ is a finite composition of Reidermeister, slide or planar isotopy moves between $c_{1pr},c_{2pr}$. And let us call $\Delta=ABC$ which exchanges edge $AB$ with edges $AC,CB$ as \textit{good},  whenever the following entering-exiting condition is satisfied: for $\Delta_{pr}$ the entering and exiting edges (if any) $MA_{pr}, BN_{pr}$ of $c_{1pr}$ (and $c_{2pr}$) do not intersect the interior of $ABC_{pr}$. The following is true:

\begin{lemma}\label{lemma_nice_and_good_moves}
	Any triangle move $\Delta$ between two rail arcs $c_1,c_2$ is a composition of other triangle moves $\Delta_i$, each being a nice or good one,  with triangles satisfying $\Delta_i\subseteq \Delta$ for all  $i$.
\end{lemma}

\begin{proof}

	Let $\Delta=ABC$ be a move which, say for definiteness, replaces edge $AB$ of $c_1$ by edges $AC,CB$ of  $c_2$. If  on $\pi$ it happens for example that the entering edge, say $MA_{pr}$, intersects the interior of $ABC_{pr}$, then by  small triangle moves we put a new vertex on $c_1$ at a point  $B'$ nearby $A$  on edge $AB$. We take care that the triangle of each one of these moves is a tiny triangle part of $ABC$ which projects its vertices away from edges of $c_{pr}$ other than $AB$ and whose projection on $\pi$ contains no vertices of $c_{pr}$ other than $A,B$.
 	Let $\delta$ be the composition of these moves. Let $C'$ be a point close to $A$ on the segment $AC$, and let $\Delta_1$ be the triangle move replacing $AB'$ by $AC',C'B'$. Let $\Delta_2$ be the triangle move replacing $C'B',B'B$ by $C'B$. Let $\Delta_3$ be the triangle move replacing $C'B$ by $C'C,CB$.  Then $\Delta=   \Delta_3 \circ \Delta_2 \circ \Delta_1 \circ \delta$ (no question here about if these moves compose).
	
	Note now that each submove of $\delta$ projects on $\pi$ to a planar isotopy of type $2$ or $3$, thus $\delta$ is a nice move. Also, choosing $B',C'$ appropriately close to $A$,  the projection $\Delta_{1pr}$ is an $\Omega_1$ move, thus $\Delta_1$ is a nice move. Finally $\Delta_{2pr},\Delta_{3pr}$ are good moves with respect to their entering edge. Since the triangle of $\delta$ and all of $\Delta_i$'s is a subset of $ABC$ we would have finished, if only $\Delta_{2},\Delta_{3}$ were good with respect to their exiting edge as well.   
	
	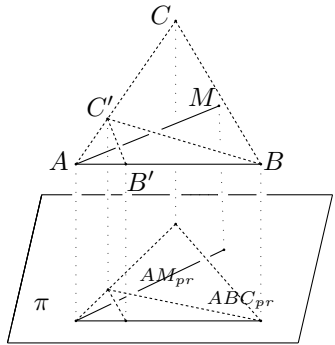
\begin{figure}[!h]
		\centering
		\psset{xunit=1.0cm,yunit=1.0cm,algebraic=true,dimen=middle,dotstyle=o,dotsize=5pt 0,linewidth=1.6pt,arrowsize=3pt 2,arrowinset=0.25}
		\begin{pspicture*}(0.9039745464411381,0.18595033906040837)(5.810601579135646,4.9)
		\psaxes[labelFontSize=\scriptstyle,xAxis=true,yAxis=true,Dx=1.,Dy=1.,ticksize=-2pt 0,subticks=2]{->}(0,0)(0.9039745464411381,0.18595033906040837)(5.810601579135646,4.695636757874054)
		\psline[linewidth=0.4pt]{->}(-2.160245098369651,4.880482899318799)(-0.5612816733105834,4.880482899318799)
		\psline[linewidth=0.4pt](-2.160245098369651,0.18595033906040837)(-2.160245098369651,4.695636757874054)
		\psline[linewidth=0.4pt]{->}(-2.160245098369651,4.880482899318799)(-2.160245098369651,3.481366940228388)
		\psline[linewidth=0.4pt]{->}(-2.160245098369651,4.880482899318799)(-0.5331218225677232,3.5228195290636695)
		\psline[linewidth=0.4pt](2.159921269396657,2.687579042046982)(4.618002623124626,2.687579042046982)
		\psline[linewidth=0.4pt,linestyle=dashed,dash=1pt 1pt](3.485736802706625,4.594043714032869)(4.618002623124626,2.687579042046982)
		\psline[linewidth=0.4pt,linestyle=dashed,dash=1pt 1pt](2.159921269396657,2.687579042046982)(3.485736802706625,4.594043714032869)
		\psline[linewidth=0.4pt](1.249459091049676,0.3100283475438989)(1.7050794441005366,2.2811246449738487)
		\psline[linewidth=0.4pt](1.249459091049676,0.3100283475438989)(5.115172351394384,0.3100283475438989)
		\psline[linewidth=0.4pt](5.115172351394384,0.3100283475438989)(5.570792704445244,2.2811246449738487)
		\psline[linewidth=0.4pt](1.3172267828429507,0.6032036753901218)(1.7050794441005366,2.2811246449738487)
		\psline[linewidth=0.4pt,linestyle=dashed,dash=1pt 1pt](2.159921269396657,0.6032036753901218)(3.485736802706625,1.8923926813588636)
		\psline[linewidth=0.4pt,linestyle=dashed,dash=1pt 1pt](3.485736802706625,1.8923926813588636)(4.618002623124626,0.6032036753901218)
		\psline[linewidth=0.4pt,linestyle=dashed,dash=1pt 1pt](2.5806621879779996,3.2925860563573863)(2.823051589627747,2.687579042046982)
		\psline[linewidth=0.4pt,linestyle=dashed,dash=1pt 1pt](2.5806621879779996,3.2925860563573863)(4.618002623124626,2.687579042046982)
		\psline[linewidth=0.4pt,linestyle=dashed,dash=1pt 1pt](2.5874038478937966,1.01887676609967)(2.823051589627747,0.6032036753901218)
		\psline[linewidth=0.4pt,linestyle=dashed,dash=1pt 1pt](2.5874038478937966,1.01887676609967)(4.618002623124626,0.6032036753901218)
		\psline[linewidth=0.4pt](2.159921269396657,2.687579042046982)(2.6876535180912806,2.9028633167852242)
		\psline[linewidth=0.4pt](3.29801180114528,3.151854240919169)(4.060793259552503,3.4630250163704948)
		\psline[linewidth=0.4pt](2.159921269396657,0.6032036753901218)(4.618002623124626,0.6032036753901218)
		\psline[linewidth=0.4pt](1.249459091049676,0.3100283475438989)(1.7050794441005366,2.2811246449738487)
		\psline[linewidth=0.4pt,linestyle=dotted](2.159921269396657,2.687579042046982)(2.159921269396657,0.6032036753901218)
		\psline[linewidth=0.4pt,linestyle=dotted](2.823051589627747,2.687579042046982)(2.823051589627747,0.6032036753901218)
		\psline[linewidth=0.4pt,linestyle=dotted](4.618002623124626,2.687579042046982)(4.618002623124626,0.6032036753901218)
		\psline[linewidth=0.4pt,linestyle=dotted](2.5806621879779996,3.2925860563573863)(2.582160812673773,2.787156097571605)
		\psline[linewidth=0.4pt,linestyle=dotted](2.582743287256272,2.5907092452965226)(2.5874038478937966,1.01887676609967)
		\psline[linewidth=0.4pt,linestyle=dotted](3.485736802706625,2.5908041856009985)(3.485736802706625,1.8923926813588636)
		\psline[linewidth=0.4pt,linestyle=dotted](3.485736802706625,3.384026355602877)(3.485736802706625,4.594043714032869)
		\psline[linewidth=0.4pt,linestyle=dotted](4.060793259552503,3.4630250163704948)(4.076978989599283,3.0084332283486037)
		\psline[linewidth=0.4pt,linestyle=dotted](4.092280926437204,2.578663622205184)(4.1289018780680955,1.5501289912163103)
		\psline[linewidth=0.4pt,linestyle=dotted](3.5631566878634113,2.2811246449738487)(4.047030970093326,2.2811246449738487)
		\psline[linewidth=0.4pt](1.7050794441005366,2.2811246449738487)(2.109774298329194,2.2811246449738487)
		\psline[linewidth=0.4pt](2.2350582033944772,2.2811246449738487)(2.5477008930182636,2.2811246449738487)
		\psline[linewidth=0.4pt](2.624440580337376,2.2811246449738487)(2.777955125344037,2.2811246449738487)
		\psline[linewidth=0.4pt](2.9050876640307264,2.2811246449738487)(3.4083169175498385,2.2811246449738487)
		\psline[linewidth=0.4pt](3.5631566878634113,2.2811246449738487)(4.047030970093326,2.2811246449738487)
		\psline[linewidth=0.4pt](4.559937709257037,2.2811246449738487)(4.182515769117702,2.2811246449738487)
		\psline[linewidth=0.4pt](4.695422508281413,2.2811246449738487)(5.570792704445244,2.2811246449738487)
		\psline[linewidth=0.4pt](2.159921269396657,0.6032036753901218)(2.6283296356558017,0.8284713789937908)
		\psline[linewidth=0.4pt](2.7325658926052965,0.8786008470534281)(2.8305364468373657,0.9257170033016431)
		\psline[linewidth=0.4pt](2.982773925982307,0.9989312962129258)(3.8093409258521698,1.3964452249719281)
		\psline[linewidth=0.4pt](3.9448346510437498,1.4616070848472649)(4.1289018780680955,1.5501289912163103)
		\psline[linewidth=0.4pt](2.7801685196387553,2.94060409358979)(3.1310919557550805,3.083760585579484)
		\begin{scriptsize}
		\psdots[dotsize=1pt 0,dotstyle=*](-2.160245098369651,4.880482899318799)
		\psdots[dotsize=1pt 0,dotstyle=*](-0.5612816733105834,4.880482899318799)
		\psdots[dotsize=1pt 0,dotstyle=*](-2.160245098369651,3.481366940228388)
		\psdots[dotsize=1pt 0,dotstyle=*](-0.5331218225677231,3.5228195290636695)
		\psdots[dotsize=1pt 0,dotstyle=*](2.159921269396657,2.687579042046982)
		\psdots[dotsize=1pt 0,dotstyle=*](3.485736802706625,4.594043714032869)
		\psdots[dotsize=1pt 0,dotstyle=*](4.618002623124626,2.687579042046982)
		\psdots[dotsize=1pt 0,dotstyle=*](2.159921269396657,0.6032036753901218)
		\psdots[dotsize=1pt 0,dotstyle=*](4.618002623124626,0.6032036753901218)
		\psdots[dotsize=1pt 0,dotstyle=*](3.485736802706625,1.8923926813588636)
		\psdots[dotsize=1pt 0,dotstyle=*](2.5806621879779996,3.2925860563573863)
		\psdots[dotsize=1pt 0,dotstyle=*](2.823051589627747,2.687579042046982)
		\psdots[dotsize=1pt 0,dotstyle=*](2.823051589627747,0.6032036753901218)
		\psdots[dotsize=1pt 0,dotstyle=*](2.5874038478937966,1.01887676609967)
		\psdots[dotsize=1pt 0,dotstyle=*](4.060793259552503,3.4630250163704948)
		\psdots[dotsize=1pt 0,dotstyle=*](4.1289018780680955,1.5501289912163103)
		\end{scriptsize}
		\small
		\rput[tl](1.8,2.85){$A$}
		\rput[tl](4.65,2.85){$B$}
		\rput[tl](3.15,4.8){$C$}
		\rput[tl](2.83,2.6){$B'$}
		\rput[tl](2.3,3.6){$C'$}
		\rput[tl](3.65,3.7){$M$}
		\rput[tl](1.6,0.9){$\pi$}
		\tiny
		\rput[tl](3,1.3){$AM_{pr}$}
		\rput[tl](3.9,1){$ABC_{pr}$}
		\normalsize
		\end{pspicture*}
		\caption{The only original vertices of $c_1$ are $M,A,B$. The small triangles (subtriangles of $ABC$) of the moves inserting vertex $B'$ to $c_1$ are not shown.}
		\label{figure_entering_edge_condition}
	\end{figure}

	If still any one of $\Delta_2,\Delta_3$ is not a good move, then working similarly for it, but this time for the exiting edge, we write it as a composition of submoves which are nice or satisfy the exiting edge condition. Since for the latter moves the entering edge condition is automatically satisfied, they are good, and by Lemma \ref{lemma_triangle_moves}  we are done.
	
	If $\Delta$ is a space slide move, observe that one of the entering or exiting conditions is  automatically satisfied as there exists no such edge for the move. Also, let $AB$ be the edge that become $AC$ performing the move, and $BC$ be part of a rail. The arguments developed above work for the vertex $A$ with the slight modification that $C'B$ is not replaced by the two edges $C'C, CB$, but rather just by the sigle edge $C'C$.
	\end{proof}

\begin{lemma}\label{lemma_good_moves}
	Any good triangle move $\Delta$ between two rail arcs $c_1,c_2$ is a finite composition  of nice moves $\Delta_i$ with triangles satisfying $\Delta_i\subseteq \Delta, \ \forall i$.
\end{lemma}

\begin{proof}
	Let us call $c_{pr}$  the projection of what is left in common after removing the non-common points of $c_1,c_2$  on $\Delta$. We consider two cases:
	
	Case (I). The move $\Delta=ABC$ exchanges  the edge $AB$ with the edges $AC,CB$ between  $c_1$ and $c_2$ (it replaces  $AB$ by $AC,CB$ or vice versa). Then no side of the triangle $\Delta$ can be part of a rail, and the move $\Delta$ is not a slide move.
	
	Since $\Delta$ is good, the entering and exiting edges (if any) on $\pi$,  say $MA_{pr}$ and $BN_{pr}$, do not intersect the interior of $\Delta_{pr}$. But the rails, as well as the  projection $c_{pr}$ might do. As always our projections keep track of over/under data.
	
	Since no rail and no arc of $c_1,c_2$ pierces triangle $\Delta$, the parts of $S=\ell_1 \cup \ell_2 \cup c_{pr}$ that intersect $\Delta_{pr}$ are equipped with data rendering them either entirely over or entirely under $\Delta_{pr}$. The  vertices and crossing points of $S $ in $\Delta_{pr}$ are finite and have to appear in the interior of $\Delta_{pr}$. Thus in the interior of $\Delta_{pr}$ we can consider small enough, disjoint triangles  around each one of these vertices and crossings, and we can extend these triangles to a finite triangulation of the whole triangle $\Delta_{pr}$, taking care of putting no vertex or side of the triangulation on $S$. Then each triangle $\delta_i$ of the triangulation contains a part of $S$  falling to one of four types: (i) $\delta_i$ contains a single crossing point of $S$ with branches through it that intersect two sides of $\delta_i$ at interior points, (ii) $\delta_i$ contains a single vertex of $S$ and parts of the two edges with endpoint this vertex, that intersect one or two sides of $\delta$ at interior points, (iii) $\delta_i$ contains only a part of an edge that intersect two sides of $\delta$ at interior points, (iv) $\delta_i$ contains no no point of $S$. Constructing the triangulation, it is convenient to consider triangles of types (i), (ii) that look as in Figure \ref{figure_cases_of_projected_moves}.

	\begin{figure}[!h]
		\centering
		\newrgbcolor{eqeqeq}{0.8784313725490196 0.8784313725490196 0.8784313725490196}
		\psset{xunit=1.0cm,yunit=1.0cm,algebraic=true,dimen=middle,dotstyle=o,dotsize=5pt 0,linewidth=1.6pt,arrowsize=3pt 2,arrowinset=0.25}
		\begin{pspicture*}(-0.6888621129048845,0.25524452407850734)(11.892586521983258,4.66537336136039)
		\pspolygon[linewidth=0.4pt,linestyle=dashed,dash=1pt 1pt,linecolor=eqeqeq,fillcolor=eqeqeq,fillstyle=solid,opacity=0.1](2.5934914589609326,4.3139750959131975)(1.734227147313643,2.6498302898114767)(3.920456598466873,2.6498302898114767)
		\pspolygon[linewidth=0.pt,linecolor=eqeqeq,fillcolor=eqeqeq,fillstyle=solid,opacity=0.1](5.43232671642856,4.3139750959131975)(4.57306240478127,2.6498302898114767)(6.7592918559345,2.6498302898114767)
		\pspolygon[linewidth=0.pt,linecolor=eqeqeq,fillcolor=eqeqeq,fillstyle=solid,opacity=0.1](8.271161973896188,4.3139750959131975)(7.411897662248897,2.6498302898114767)(9.598127113402128,2.6498302898114767)
		\pspolygon[linewidth=0.pt,linecolor=eqeqeq,fillcolor=eqeqeq,fillstyle=solid,opacity=0.1](1.2556495560164187,2.247390042584267)(0.39638524436912914,0.5832452364825462)(2.5826146955223592,0.5832452364825462)
		\pspolygon[linewidth=0.pt,linecolor=eqeqeq,fillcolor=eqeqeq,fillstyle=solid,opacity=0.1](4.094484813484046,2.247390042584267)(3.235220501836756,0.5832452364825462)(5.421449952989986,0.5832452364825462)
		\pspolygon[linewidth=0.pt,linecolor=eqeqeq,fillcolor=eqeqeq,fillstyle=solid,opacity=0.1](7.36990596801248,2.236534911682713)(6.51064165636519,0.5723901055809923)(8.69687110751842,0.5723901055809923)
		\pspolygon[linewidth=0.pt,linecolor=eqeqeq,fillcolor=eqeqeq,fillstyle=solid,opacity=0.1](10.208741225480107,2.236534911682713)(9.349476913832817,0.5723901055809923)(11.535706364986048,0.5723901055809923)
		\psline[linewidth=0.4pt]{->}(-0.6695375726110528,6.804753923346491)(2.1692976848565744,6.804753923346491)
		\psline[linewidth=0.4pt]{->}(-0.6695375726110528,6.804753923346491)(-0.6695375726110528,5.40563796425608)
		\psline[linewidth=0.4pt]{->}(-0.6695375726110528,6.804753923346491)(-2.0073794755555667,4.73816887001756)
		\psline[linewidth=0.4pt](2.200622763559666,2.317110096225533)(3.3324587669776013,3.9095770312670566)
		\psline[linewidth=0.4pt,linestyle=dashed,dash=1pt 1pt](4.57306240478127,2.6498302898114767)(5.43232671642856,4.3139750959131975)
		\psline[linewidth=0.4pt](5.0394580210272935,2.317110096225533)(6.171294024445229,3.9095770312670566)
		\psline[linewidth=0.4pt,linestyle=dashed,dash=1pt 1pt](7.411897662248897,2.6498302898114767)(8.271161973896188,4.3139750959131975)
		\psline[linewidth=0.4pt,linestyle=dashed,dash=1pt 1pt](8.271161973896188,4.3139750959131975)(9.598127113402128,2.6498302898114767)
		\psline[linewidth=0.4pt,linestyle=dashed,dash=1pt 1pt](9.598127113402128,2.6498302898114767)(7.411897662248897,2.6498302898114767)
		\psline[linewidth=0.4pt,linestyle=dashed,dash=1pt 1pt](2.232455546836078,3.6147536458486025)(2.5934914589609326,4.3139750959131975)
		\psline[linewidth=0.4pt,linestyle=dashed,dash=1pt 1pt](2.5934914589609326,4.3139750959131975)(3.0966175358022983,3.683005507743285)
		\psline[linewidth=0.4pt](1.9637268558675396,3.777968193660319)(2.6457859342417906,3.132124641571423)
		\psline[linewidth=0.4pt](2.7812520661310156,3.003851401640917)(3.4509067208236646,2.369753631268228)
		\psline[linewidth=0.4pt](4.802562113335167,3.777968193660319)(4.972209177201309,3.617328938495033)
		\psline[linewidth=0.4pt](5.641863831893957,2.983231168122345)(5.930207047642776,2.7101982116168246)
		\psline[linewidth=0.4pt](6.081715933957452,2.566734044929476)(6.289741978291292,2.369753631268228)
		\psline[linewidth=0.4pt,linestyle=dashed,dash=1pt 1pt](5.43232671642856,4.3139750959131975)(5.937531969839292,3.680398015816128)
		\psline[linewidth=0.4pt,linestyle=dashed,dash=1pt 1pt](6.034976045056455,3.558193560830669)(6.7592918559345,2.6498302898114767)
		\psline[linewidth=0.4pt](7.641397370802794,3.777968193660319)(7.811044434668936,3.617328938495033)
		\psline[linewidth=0.4pt](8.480699089361584,2.983231168122345)(8.769042305110403,2.7101982116168246)
		\psline[linewidth=0.4pt](8.92055119142508,2.566734044929476)(9.128577235758918,2.369753631268228)
		\psline[linewidth=0.4pt](7.878293278494921,2.317110096225533)(8.05724291185045,2.568888068737383)
		\psline[linewidth=0.4pt](8.192835190700716,2.7596632517708986)(8.76365814037567,3.5627978670112452)
		\psline[linewidth=0.4pt](8.8479639250631,3.6814141454668152)(9.010129281912857,3.9095770312670566)
		\psline[linewidth=0.4pt,linestyle=dashed,dash=1pt 1pt](0.39638524436912914,0.5832452364825462)(2.5826146955223592,0.5832452364825462)
		\psline[linewidth=0.4pt](1.3566863003495,1.522514930265145)(0.19207717365909724,0.9466093181654913)
		\psline[linewidth=0.4pt](1.3566863003495,1.522514930265145)(2.623678646968729,1.0489925380943186)
		\psline[linewidth=0.4pt,linestyle=dashed,dash=1pt 1pt](4.094484813484046,2.247390042584267)(3.235220501836756,0.5832452364825462)
		\psline[linewidth=0.4pt,linestyle=dashed,dash=1pt 1pt](3.235220501836756,0.5832452364825462)(5.421449952989986,0.5832452364825462)
		\psline[linewidth=0.4pt,linestyle=dashed,dash=1pt 1pt](5.421449952989986,0.5832452364825462)(4.094484813484046,2.247390042584267)
		\psline[linewidth=0.4pt]{->}(-0.6695375726110528,6.804753923346491)(2.974178179310998,6.804753923346491)
		\psline[linewidth=0.4pt,linestyle=dashed,dash=1pt 1pt](1.2556495560164187,2.247390042584267)(0.7901900114074163,1.345930418214929)
		\psline[linewidth=0.4pt,linestyle=dashed,dash=1pt 1pt](0.6454399071948038,1.0655916087905009)(0.39638524436912914,0.5832452364825462)
		\psline[linewidth=0.4pt,linestyle=dashed,dash=1pt 1pt](1.2556495560164187,2.247390042584267)(1.923862252464672,1.409385595399158)
		\psline[linewidth=0.4pt,linestyle=dashed,dash=1pt 1pt](2.1251920763906753,1.1568981932624478)(2.5826146955223592,0.5832452364825462)
		\psline[linewidth=0.4pt](3.0309124311267244,0.9466093181654913)(3.415965752933083,1.1370203014763294)
		\psline[linewidth=0.4pt](3.6485132405237843,1.2520163118233802)(4.195521557817127,1.522514930265145)
		\psline[linewidth=0.4pt](4.195521557817127,1.522514930265145)(4.726454301581019,1.3240855209796485)
		\psline[linewidth=0.4pt](4.973387366808421,1.2317974056926393)(5.462513904436356,1.0489925380943186)
		\psline[linewidth=0.4pt,linestyle=dashed,dash=1pt 1pt,linecolor=eqeqeq](2.5934914589609326,4.3139750959131975)(1.734227147313643,2.6498302898114767)
		\psline[linewidth=0.4pt,linestyle=dashed,dash=1pt 1pt,linecolor=eqeqeq](1.734227147313643,2.6498302898114767)(3.920456598466873,2.6498302898114767)
		\psline[linewidth=0.4pt,linestyle=dashed,dash=1pt 1pt,linecolor=eqeqeq](3.920456598466873,2.6498302898114767)(2.5934914589609326,4.3139750959131975)
		\psline[linewidth=0.4pt]{->}(-0.6695375726110528,6.804753923346491)(5.444718839385009,6.793898792444937)
		\psline[linewidth=0.4pt,linestyle=dashed,dash=1pt 1pt](6.51064165636519,0.5723901055809923)(8.69687110751842,0.5723901055809923)
		\psline[linewidth=0.4pt](6.441116455374437,1.134049421376562)(8.322796530247622,1.5908734857966422)
		\psline[linewidth=0.4pt,linestyle=dashed,dash=1pt 1pt](9.349476913832817,0.5723901055809923)(11.535706364986048,0.5723901055809923)
		\psline[linewidth=0.4pt,linestyle=dashed,dash=1pt 1pt](10.208741225480107,2.236534911682713)(9.349476913832817,0.5723901055809923)
		\psline[linewidth=0.4pt,linestyle=dashed,dash=1pt 1pt](10.208741225480107,2.236534911682713)(11.535706364986048,0.5723901055809923)
		\psline[linewidth=0.4pt,linestyle=dashed,dash=1pt 1pt](7.36990596801248,2.236534911682713)(6.890282886806318,1.3076446404859654)
		\psline[linewidth=0.4pt,linestyle=dashed,dash=1pt 1pt](6.816760001185773,1.1652522164360473)(6.51064165636519,0.5723901055809923)
		\psline[linewidth=0.4pt,linestyle=dashed,dash=1pt 1pt](7.36990596801248,2.236534911682713)(7.905779075821217,1.5644973256602774)
		\psline[linewidth=0.4pt,linestyle=dashed,dash=1pt 1pt](7.9979212507865185,1.4489419750890375)(8.69687110751842,0.5723901055809923)
		\psline[linewidth=0.4pt](9.279951712842063,1.134049421376562)(9.595276954303946,1.2106023701707764)
		\psline[linewidth=0.4pt](9.79542031505861,1.259192087810637)(10.73142152606343,1.4864293760314606)
		\psline[linewidth=0.4pt](10.874627039602512,1.521196032497596)(11.16163178771525,1.5908734857966422)
		\rput[tl](1.3109260385352308,3.8177810322731816){$(i)$}
		\rput[tl](-0.3312840990712215,1.6988002095551598){$(ii)$}
		\rput[tl](6.237556451354587,1.8842110315429865){$(iii)$}
		\psline[linewidth=0.4pt,linestyle=dashed,dash=1pt 1pt](1.734227147313643,2.6498302898114767)(2.171312442957512,3.4963372547926443)
		\psline[linewidth=0.4pt,linestyle=dashed,dash=1pt 1pt](3.920456598466873,2.6498302898114767)(3.1832554373296578,3.5743530574671696)
		\psline[linewidth=0.4pt,linestyle=dashed,dash=1pt 1pt](3.920456598466873,2.6498302898114767)(3.2403325806528858,2.6498302898114767)
		\psline[linewidth=0.4pt,linestyle=dashed,dash=1pt 1pt](3.0692410917641277,2.6498302898114767)(2.542805741337181,2.6498302898114767)
		\psline[linewidth=0.4pt,linestyle=dashed,dash=1pt 1pt](2.3453924849270757,2.6498302898114767)(1.734227147313643,2.6498302898114767)
		\psline[linewidth=0.4pt,linestyle=dashed,dash=1pt 1pt](4.57306240478127,2.6498302898114767)(5.161821609711241,2.6498302898114767)
		\psline[linewidth=0.4pt,linestyle=dashed,dash=1pt 1pt](5.411878401164041,2.6498302898114767)(6.7592918559345,2.6498302898114767)
		\psline[linewidth=0.4pt](5.08833177711179,3.50737214034971)(5.460771527750112,3.154707951692181)
		\psline[linewidth=0.4pt](7.927167034579417,3.50737214034971)(8.29960678521774,3.154707951692181)
		\begin{scriptsize}
		\psdots[dotsize=2pt 0,dotstyle=*](-0.6695375726110528,6.804753923346491)
		\psdots[dotsize=2pt 0,dotstyle=*](2.1692976848565744,6.804753923346491)
		\psdots[dotsize=2pt 0,dotstyle=*](-0.6695375726110528,5.40563796425608)
		\psdots[dotsize=2pt 0,dotstyle=*](-2.0073794755555667,4.73816887001756)
		\psdots[dotsize=2pt 0](1.734227147313643,2.6498302898114767)
		\psdots[dotsize=2pt 0](2.5934914589609326,4.3139750959131975)
		\psdots[dotsize=2pt 0](3.920456598466873,2.6498302898114767)
		\psdots[dotsize=2pt 0](4.57306240478127,2.6498302898114767)
		\psdots[dotsize=2pt 0](5.43232671642856,4.3139750959131975)
		\psdots[dotsize=2pt 0](6.7592918559345,2.6498302898114767)
		\psdots[dotsize=2pt 0](7.411897662248897,2.6498302898114767)
		\psdots[dotsize=2pt 0](8.271161973896188,4.3139750959131975)
		\psdots[dotsize=2pt 0](9.598127113402128,2.6498302898114767)
		\psdots[dotsize=2pt 0](1.2556495560164187,2.247390042584267)
		\psdots[dotsize=2pt 0](0.39638524436912914,0.5832452364825462)
		\psdots[dotsize=2pt 0](2.5826146955223592,0.5832452364825462)
		\psdots[dotsize=2pt 0,dotstyle=*](1.3566863003495,1.522514930265145)
		\psdots[dotsize=2pt 0](4.094484813484046,2.247390042584267)
		\psdots[dotsize=2pt 0](3.235220501836756,0.5832452364825462)
		\psdots[dotsize=2pt 0](5.421449952989986,0.5832452364825462)
		\psdots[dotsize=2pt 0,dotstyle=*](4.195521557817127,1.522514930265145)
		\psdots[dotsize=2pt 0,dotstyle=*](4.195521557817127,1.522514930265145)
		\psdots[dotsize=2pt 0,dotstyle=*](2.974178179310998,6.804753923346491)
		\psdots[dotsize=2pt 0,dotstyle=*](5.444718839385009,6.793898792444937)
		\psdots[dotsize=2pt 0](6.51064165636519,0.5723901055809923)
		\psdots[dotsize=2pt 0](7.36990596801248,2.236534911682713)
		\psdots[dotsize=2pt 0](8.69687110751842,0.5723901055809923)
		\psdots[dotsize=2pt 0](9.349476913832817,0.5723901055809923)
		\psdots[dotsize=2pt 0](10.208741225480107,2.236534911682713)
		\psdots[dotsize=2pt 0](11.535706364986048,0.5723901055809923)
		\end{scriptsize}
		\end{pspicture*}
		\caption{A triangle move $\Delta$ in space, decomposes to smaller moves that project on $\pi$ to triangles that contain either no parts of the projected arcs and the rails, or else they do so in the three ways shown here.}
		\label{figure_cases_of_projected_moves}
	\end{figure}
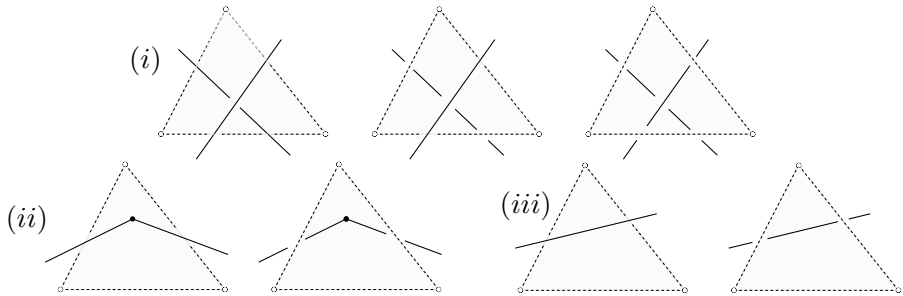
	
	Let us notice that the triangulation by the $\delta_i$'s of $\Delta_{pr}$, implies a triangulation of $\Delta$ by triangles $\Delta_i$'s that project onto the $\delta_i$'s on $\pi$. Let us also notice that by Lemma \ref{lemma_triangle_moves}, it is legitimate to perform consecutive moves in each one of the $\Delta_i$'s. And let for definiteness, the move $\Delta$ replaces $AB$ by $AC,CB$. Then since the triangles of the $\Delta_i$'s have $\Delta$ as their union, we can arrive at the same  result of replacing $AB$ by $AC,CB$   performing moves through the $\Delta_i$'s starting from triangles with sides on $AB$, and ending up with those with sides on $AC,CB$. 
	
	The $\Delta_i$ moves, which are moves between arcs in space, project on $\pi$ to corresponding $\delta_i$ moves between the projected arcs on the plane. But such a move for a triangle $\delta_i$ of type (i) is either an $\Omega_3$ move composed with a planar isotopy 1 move, or it decomposes to an $\Omega_2$  move,  an $\Omega_3$  move,  and some planar isotopies (Figure \ref{figure_decomposing_type_i_moves}). Similary,  for $\delta_i$ of type (ii) the corresponding move is either an $\Omega_2$ move or it decomposes to  some planar isotopies; for $\delta_i$ of type (iii) the corresponding move is an isotopy 1 move, an isotopy 3 move, or decomposes to an $\Omega_2$ and some isotopies; whereas for $\delta_i$ of type (iv) the corresponding move is a planar isotopy. Thus each $\Delta_i$ is a nice move and we are done.
	
	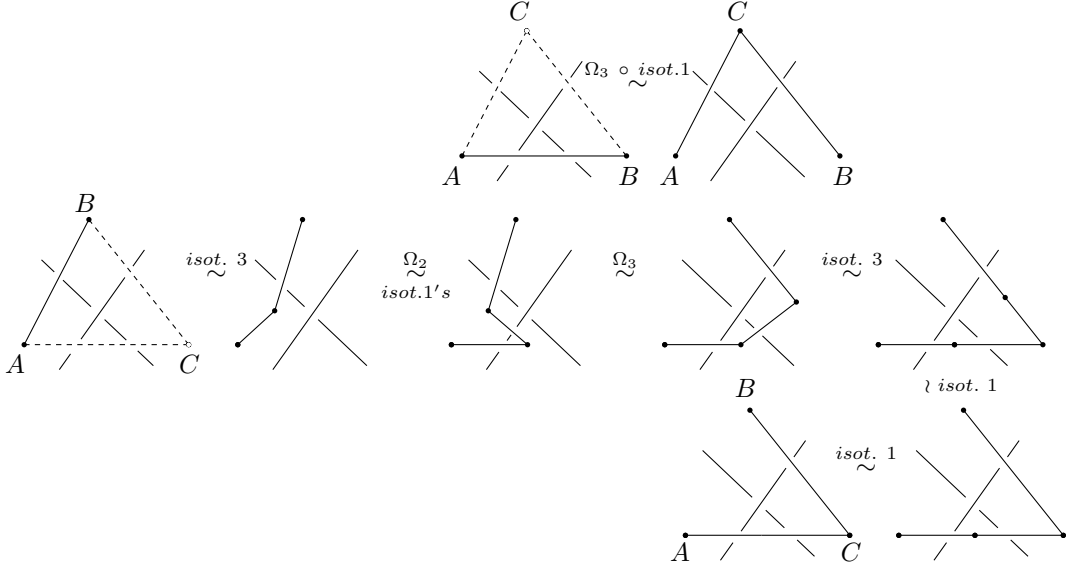
\begin{figure}[!h]
		\centering
		\psset{xunit=1.0cm,yunit=1.0cm,algebraic=true,dimen=middle,dotstyle=o,dotsize=5pt 0,linewidth=1.6pt,arrowsize=3pt 2,arrowinset=0.25}
		\begin{pspicture*}(-1.5034050059077912,-3.146435587457053)(12.885177344844362,4.780607466892581)
		\psline[linewidth=0.4pt]{->}(-0.6695375726110528,6.804753923346491)(2.1692976848565744,6.804753923346491)
		\psline[linewidth=0.4pt]{->}(-0.6695375726110528,6.804753923346491)(-0.6695375726110528,5.40563796425608)
		\psline[linewidth=0.4pt]{->}(-0.6695375726110528,6.804753923346491)(-6.4895721150954095,4.294551717384737)
		\psline[linewidth=0.4pt,linestyle=dashed,dash=2pt 2pt](4.57306240478127,2.6498302898114767)(5.43232671642856,4.3139750959131975)
		\psline[linewidth=0.4pt](8.271161973896188,4.3139750959131975)(9.598127113402128,2.6498302898114767)
		\psline[linewidth=0.4pt](4.802562113335167,3.777968193660319)(4.972209177201309,3.617328938495033)
		\psline[linewidth=0.4pt](5.641863831893957,2.983231168122345)(5.930207047642776,2.7101982116168246)
		\psline[linewidth=0.4pt](6.081715933957452,2.566734044929476)(6.289741978291292,2.369753631268228)
		\psline[linewidth=0.4pt](8.8479639250631,3.6814141454668152)(9.010129281912857,3.9095770312670566)
		\psline[linewidth=0.4pt]{->}(-0.6695375726110528,6.804753923346491)(2.974178179310998,6.804753923346491)
		\psline[linewidth=0.4pt]{->}(-0.6695375726110528,6.804753923346491)(5.444718839385009,6.793898792444937)
		\psline[linewidth=0.4pt](5.08833177711179,3.50737214034971)(5.460771527750112,3.154707951692181)
		\psline[linewidth=0.4pt](4.57306240478127,2.6498302898114767)(6.7592918559345,2.6498302898114767)
		\psline[linewidth=0.4pt,linestyle=dashed,dash=2pt 2pt](5.43232671642856,4.3139750959131975)(6.7592918559345,2.6498302898114767)
		\psline[linewidth=0.4pt](5.0394580210272935,2.317110096225533)(5.188508904121335,2.526821222439243)
		\psline[linewidth=0.4pt](5.327899658869849,2.7229407727249444)(5.91120112646038,3.543632372474418)
		\psline[linewidth=0.4pt](6.018186244120525,3.6941579449962503)(6.171294024445229,3.9095770312670566)
		\psline[linewidth=0.4pt](-1.0174724291491897,1.267765987698565)(-0.8478253652830476,1.107126732533279)
		\psline[linewidth=0.4pt](-0.17817071059039957,0.473028962160591)(0.11017250515841948,0.1999960056550707)
		\psline[linewidth=0.4pt](0.2616813914730951,0.05653183896772207)(0.4697074358069351,-0.14044857469352579)
		\psline[linewidth=0.4pt](-0.7317027653725665,0.9971699343879559)(-0.3592630147342444,0.6445057457304273)
		\psline[linewidth=0.4pt,linestyle=dashed,dash=2pt 2pt](-0.3877078260557969,1.8037728899514436)(0.9392573134501436,0.1396280838497228)
		\psline[linewidth=0.4pt](-0.7805765214570632,-0.19309210973622104)(-0.6315256383630219,0.01661901647748909)
		\psline[linewidth=0.4pt](-0.49213488361450786,0.21273856676319047)(0.0911665839760234,1.0334301665126642)
		\psline[linewidth=0.4pt](0.19815170163616802,1.1839557390344964)(0.35125948196087187,1.3993748253053027)
		\psline[linewidth=0.4pt](-0.3877078260557969,1.8037728899514436)(-1.2469721377030867,0.1396280838497228)
		\psline[linewidth=0.4pt,linestyle=dashed,dash=2pt 2pt](-1.2469721377030867,0.1396280838497228)(0.9392573134501436,0.1396280838497228)
		\psline[linewidth=0.4pt](7.411897662248897,2.6498302898114767)(8.271161973896188,4.3139750959131975)
		\psline[linewidth=0.4pt](8.76365814037567,3.5627978670112452)(7.878293278494921,2.317110096225533)
		\psline[linewidth=0.4pt](7.641397370802794,3.777968193660319)(7.836802784699834,3.592938288465776)
		\psline[linewidth=0.4pt](7.926836830046558,3.507684811898523)(8.33504884347171,3.1211477726375354)
		\psline[linewidth=0.4pt](8.484219869873215,2.9798973317086768)(9.128577235758918,2.369753631268228)
		\psline[linewidth=0.4pt](1.5918631197645405,0.1396280838497228)(2.0829085497142685,0.5891824837112023)
		\psline[linewidth=0.4pt](2.0829085497142685,0.5891824837112023)(2.4511274314118303,1.8037728899514436)
		\psline[linewidth=0.4pt](2.09945299005633,1.0044416752564882)(1.8213628283184375,1.267765987698565)
		\psline[linewidth=0.4pt](2.479572242733383,0.6445057457304273)(2.2478824869148286,0.8638933906205639)
		\psline[linewidth=0.4pt](3.190094739428499,1.3993748253053027)(2.058258736010564,-0.19309210973622104)
		\psline[linewidth=0.4pt](2.6606645468772276,0.473028962160591)(3.3085426932745623,-0.14044857469352579)
		\psline[linewidth=0.4pt](4.921743807181896,0.5891824837112023)(5.2899626888794575,1.8037728899514436)
		\psline[linewidth=0.4pt](4.938288247523957,1.0044416752564882)(4.660198085786065,1.267765987698565)
		\psline[linewidth=0.4pt](5.31840750020101,0.6445057457304273)(5.086717744382456,0.8638933906205639)
		\psline[linewidth=0.4pt](5.499499804344855,0.473028962160591)(6.1473779507421895,-0.14044857469352579)
		\psline[linewidth=0.4pt](4.921743807181896,0.5891824837112023)(5.442938642835327,0.1396280838497228)
		\psline[linewidth=0.4pt](4.430698377232168,0.1396280838497228)(5.442938642835327,0.1396280838497228)
		\psline[linewidth=0.4pt](4.897093993478191,-0.19309210973622104)(5.063864563959763,0.041550204545991454)
		\psline[linewidth=0.4pt](5.171831670622585,0.1934574127576385)(5.225996116665016,0.26966552870105875)
		\psline[linewidth=0.4pt](5.303521564244832,0.37874203052847477)(6.028929996896126,1.3993748253053027)
		\psline[linewidth=0.4pt](7.269533634699795,0.1396280838497228)(8.281773900302955,0.1396280838497228)
		\psline[linewidth=0.4pt](7.735929250945818,-0.19309210973622104)(7.90269982142739,0.041550204545991454)
		\psline[linewidth=0.4pt](8.281773900302955,0.1396280838497228)(9.018718781171952,0.7055279489060885)
		\psline[linewidth=0.4pt](9.018718781171952,0.7055279489060885)(8.128797946347085,1.8037728899514436)
		\psline[linewidth=0.4pt](7.499033343253692,1.267765987698565)(8.157242757668637,0.6445057457304273)
		\psline[linewidth=0.4pt](8.712711023292774,1.1812171281007837)(8.867765254363754,1.3993748253053027)
		\psline[linewidth=0.4pt](8.010666928090213,0.1934574127576385)(8.607534832742825,1.033236673954923)
		\psline[linewidth=0.4pt](8.338335061812483,0.473028962160591)(8.455550279645946,0.36203756120324015)
		\psline[linewidth=0.4pt](8.56690329935411,0.2565970912140936)(8.986213208209817,-0.14044857469352579)
		\psline[linewidth=0.4pt](10.108368892167423,0.1396280838497228)(11.120609157770582,0.1396280838497228)
		\psline[linewidth=0.4pt](10.574764508413445,-0.19309210973622104)(10.741535078895017,0.041550204545991454)
		\psline[linewidth=0.4pt](10.337868600721318,1.267765987698565)(10.996078015136264,0.6445057457304273)
		\psline[linewidth=0.4pt](11.551546280760402,1.1812171281007837)(11.70660051183138,1.3993748253053027)
		\psline[linewidth=0.4pt](10.84950218555784,0.1934574127576385)(11.446370090210452,1.033236673954923)
		\psline[linewidth=0.4pt](10.967633203814712,1.8037728899514436)(12.294598343320652,0.1396280838497228)
		\psline[linewidth=0.4pt](11.120609157770582,0.1396280838497228)(12.294598343320652,0.1396280838497228)
		\psline[linewidth=0.4pt](11.424580284505753,0.2387558092478126)(11.17717031928011,0.473028962160591)
		\psline[linewidth=0.4pt](11.594699520497196,0.0776694530435248)(11.825048465677444,-0.14044857469352579)
		\psline[linewidth=0.4pt]{->}(-0.6695375726110528,6.804753923346491)(-3.23554485688226,4.268492438630979)
		\psline[linewidth=0.4pt](10.381196865363844,-2.396633400865789)(11.393437130967003,-2.396633400865789)
		\psline[linewidth=0.4pt](10.847592481609865,-2.729353594451733)(11.014363052091438,-2.49471128016952)
		\psline[linewidth=0.4pt](10.610696573917739,-1.2684954970169469)(11.268905988332685,-1.8917557389850845)
		\psline[linewidth=0.4pt](11.824374253956822,-1.355044356614728)(11.979428485027801,-1.1368866594102092)
		\psline[linewidth=0.4pt](11.122330158754261,-2.3428040719578735)(11.719198063406873,-1.5030248107605888)
		\psline[linewidth=0.4pt](11.240461177011133,-0.7324885947640682)(12.567426316517073,-2.396633400865789)
		\psline[linewidth=0.4pt](11.393437130967003,-2.396633400865789)(12.567426316517073,-2.396633400865789)
		\psline[linewidth=0.4pt](11.697408257702174,-2.2975056754676992)(11.449998292476531,-2.063232522554921)
		\psline[linewidth=0.4pt](11.867527493693617,-2.458592031671987)(12.097876438873865,-2.6767100594090376)
		\psline[linewidth=0.4pt]{->}(2.1692976848565744,6.804753923346491)(-0.6695375726110528,6.804753923346491)
		\psline[linewidth=0.4pt](7.542361607896217,-2.396633400865789)(8.554601873499376,-2.396633400865789)
		\psline[linewidth=0.4pt](8.008757224142238,-2.729353594451733)(8.175527794623811,-2.49471128016952)
		\psline[linewidth=0.4pt](7.771861316450112,-1.2684954970169469)(8.430070730865058,-1.8917557389850845)
		\psline[linewidth=0.4pt](8.985538996489195,-1.355044356614728)(9.140593227560174,-1.1368866594102092)
		\psline[linewidth=0.4pt](8.283494901286634,-2.3428040719578735)(8.880362805939246,-1.5030248107605888)
		\psline[linewidth=0.4pt](8.401625919543505,-0.7324885947640682)(9.728591059049446,-2.396633400865789)
		\psline[linewidth=0.4pt](8.554601873499376,-2.396633400865789)(9.728591059049446,-2.396633400865789)
		\psline[linewidth=0.4pt](8.858573000234546,-2.2975056754676992)(8.611163035008904,-2.063232522554921)
		\psline[linewidth=0.4pt](9.02869223622599,-2.458592031671987)(9.259041181406237,-2.6767100594090376)
		\begin{scriptsize}
		\psdots[dotsize=2pt 0,dotstyle=*](-0.6695375726110528,6.804753923346491)
		\psdots[dotsize=2pt 0,dotstyle=*](2.1692976848565744,6.804753923346491)
		\psdots[dotsize=2pt 0,dotstyle=*](-0.6695375726110528,5.40563796425608)
		\psdots[dotsize=2pt 0,dotstyle=*](-6.4895721150954095,4.294551717384737)
		\psdots[dotsize=2pt 0,dotstyle=*](4.57306240478127,2.6498302898114767)
		\psdots[dotsize=2pt 0](5.43232671642856,4.3139750959131975)
		\psdots[dotsize=2pt 0,dotstyle=*](6.7592918559345,2.6498302898114767)
		\psdots[dotsize=2pt 0,dotstyle=*](7.411897662248897,2.6498302898114767)
		\psdots[dotsize=2pt 0,dotstyle=*](8.271161973896188,4.3139750959131975)
		\psdots[dotsize=2pt 0,dotstyle=*](9.598127113402128,2.6498302898114767)
		\psdots[dotsize=2pt 0,dotstyle=*](2.974178179310998,6.804753923346491)
		\psdots[dotsize=2pt 0,dotstyle=*](5.444718839385009,6.793898792444937)
		\psdots[dotsize=2pt 0,dotstyle=*](-1.2469721377030867,0.1396280838497228)
		\psdots[dotsize=2pt 0,dotstyle=*](-0.3877078260557969,1.8037728899514436)
		\psdots[dotsize=2pt 0](0.9392573134501436,0.1396280838497228)
		\psdots[dotsize=2pt 0,dotstyle=*](2.4511274314118303,1.8037728899514436)
		\psdots[dotsize=2pt 0,dotstyle=*](1.5918631197645405,0.1396280838497228)
		\psdots[dotsize=2pt 0,dotstyle=*](1.5918631197645405,0.1396280838497228)
		\psdots[dotsize=2pt 0,dotstyle=*](2.0829085497142685,0.5891824837112023)
		\psdots[dotsize=2pt 0,dotstyle=*](4.430698377232168,0.1396280838497228)
		\psdots[dotsize=2pt 0,dotstyle=*](4.921743807181896,0.5891824837112023)
		\psdots[dotsize=1pt 0,dotstyle=*](4.921743807181896,0.5891824837112023)
		\psdots[dotsize=2pt 0,dotstyle=*](5.2899626888794575,1.8037728899514436)
		\psdots[dotsize=2pt 0,dotstyle=*](4.430698377232168,0.1396280838497228)
		\psdots[dotsize=2pt 0,dotstyle=*](5.442938642835327,0.1396280838497228)
		\psdots[dotsize=2pt 0,dotstyle=*](8.128797946347085,1.8037728899514436)
		\psdots[dotsize=2pt 0,dotstyle=*](8.281773900302955,0.1396280838497228)
		\psdots[dotsize=2pt 0,dotstyle=*](7.269533634699795,0.1396280838497228)
		\psdots[dotsize=2pt 0,dotstyle=*](7.269533634699795,0.1396280838497228)
		\psdots[dotsize=2pt 0,dotstyle=*](9.018718781171952,0.7055279489060885)
		\psdots[dotsize=2pt 0,dotstyle=*](10.108368892167423,0.1396280838497228)
		\psdots[dotsize=2pt 0,dotstyle=*](11.120609157770582,0.1396280838497228)
		\psdots[dotsize=2pt 0,dotstyle=*](11.120609157770582,0.1396280838497228)
		\psdots[dotsize=2pt 0,dotstyle=*](10.967633203814712,1.8037728899514436)
		\psdots[dotsize=2pt 0,dotstyle=*](10.108368892167423,0.1396280838497228)
		\psdots[dotsize=2pt 0,dotstyle=*](12.294598343320652,0.1396280838497228)
		\psdots[dotsize=2pt 0,dotstyle=*](11.794180873592163,0.7672008122797176)
		\psdots[dotsize=1pt 0,dotstyle=*](-3.23554485688226,4.268492438630979)
		\psdots[dotsize=2pt 0,dotstyle=*](10.381196865363844,-2.396633400865789)
		\psdots[dotsize=2pt 0,dotstyle=*](11.393437130967003,-2.396633400865789)
		\psdots[dotsize=2pt 0,dotstyle=*](11.240461177011133,-0.7324885947640682)
		\psdots[dotsize=2pt 0,dotstyle=*](12.567426316517073,-2.396633400865789)
		\psdots[dotsize=2pt 0,dotstyle=*](11.393437130967003,-2.396633400865789)
		\psdots[dotsize=2pt 0,dotstyle=*](12.567426316517073,-2.396633400865789)
		\psdots[dotsize=2pt 0,dotstyle=*](10.381196865363844,-2.396633400865789)
		\psdots[dotsize=2pt 0,dotstyle=*](7.542361607896217,-2.396633400865789)
		\psdots[dotsize=2pt 0,dotstyle=*](8.401625919543505,-0.7324885947640682)
		\psdots[dotsize=2pt 0,dotstyle=*](9.728591059049446,-2.396633400865789)
		\psdots[dotsize=2pt 0,dotstyle=*](9.728591059049446,-2.396633400865789)
		\psdots[dotsize=2pt 0,dotstyle=*](7.542361607896217,-2.396633400865789)
		\end{scriptsize}
		\small
		\rput[tl](4.3,2.5){$A$}
		\rput[tl](6.65,2.5){$B$}
		\rput[tl](8.1,4.7){$C$}
		\rput[tl](7.2,2.5){$A$}
		\rput[tl](9.5,2.5){$B$}
		\rput[tl](5.2,4.7){$C$}
		\rput[tl](-1.5,0){$A$}
		\rput[tl](-0.58,2.15){$B$}
		\rput[tl](0.8114247935766407,0){$C$}
		\rput[tl](7.339577897158636,-2.5){$A$}
		\rput[tl](8.2,-0.32){$B$}
		\rput[tl](9.61,-2.5){$C$}
		\rput[tl](6.2,3.9312814253551203){$\overset{\Omega_3\ \circ \ isot.1}{\sim}$}
		\rput[tl](0.9,1.4){$\overset{isot. \ 3}{\sim}$}
		\rput[tl](3.5,1.4){$\underset{isot.1's}{\overset{\Omega_2}{\sim}}$}
		\rput[tl](6.573519114595443,1.4){$\overset{\Omega_3}{\sim}$}
		\rput[tl](9.35,1.4){$\overset{isot. \ 3}{\sim}$}
		\rput[tl](9.537833534079104,-1.1813282756644967){$\overset{isot. \ 1}{\sim}$}
		\tiny
		\rput[tl](10.720228611513598,-0.3){$\wr \ isot. \ 1$}\normalsize
		\end{pspicture*}
		\caption{Two cases of decomposing a move through a triangle $\delta_i$ of type (i) on $\pi$ to a sequence of Reidemeister, slide or planar isotopies moves.}
		\label{figure_decomposing_type_i_moves}
	\end{figure}
	
	Case (II). $\Delta=ABC$ is a space slide move. Then it does not exchange an edge by two others or vice versa as above, but it instead replaces in space one position of the initial or final edge of $c_1,c_2$ by another position, say from $AB$ to $CB$ with $A,C$ on a rail $\ell$.
	
	We spot  all crossing points of $AC$ with   $c_{pr}$. 	If any, we call these points in increasing distance from $A$ as $X_1,X_2,\ldots X_k$. Let $Y_i$ be an interior point in the segment $X_iX_{i+1}$ for $i=1,\ldots,k-1$.  We decompose move $\Delta$ to a sequence of moves through the triangles $BAY_1, BY_1Y_2,\ldots,BY_{k-2}Y_{k-1},BY_{k-1}C$; we are allowed to, by Lemma \ref{lemma_triangle_moves}. We'll prove that each one of these moves is as required, and then by Lemma \ref{lemma_triangle_moves} we'll get that $\Delta$ is also as required and we will be done. It is enough to prove that the move $\delta$ through $BAY_1$ satisfies our Lemma, since the reasoning will apply to all other triangles in this sequence as well. To this end:

	\begin{figure}[!h]
		\centering
		\psset{xunit=1.0cm,yunit=1.0cm,algebraic=true,dimen=middle,dotstyle=o,dotsize=5pt 0,linewidth=1.6pt,arrowsize=3pt 2,arrowinset=0.25}
		\begin{pspicture*}(1.282992947585356,0.12137923193432268)(4.861448000854534,5.558119716374894)
		\psaxes[labelFontSize=\scriptstyle,xAxis=true,yAxis=true,Dx=1.,Dy=1.,ticksize=-2pt 0,subticks=2]{->}(0,0)(1.282992947585356,0.12137923193432268)(4.861448000854534,5.558119716374894)
		\psline[linewidth=0.4pt]{->}(-0.6695375726110528,6.804753923346491)(2.1692976848565744,6.804753923346491)
		\psline[linewidth=0.4pt]{->}(-0.6695375726110528,6.804753923346491)(-0.6695375726110528,5.40563796425608)
		\psline[linewidth=0.4pt]{->}(-0.6695375726110528,6.804753923346491)(-6.4895721150954095,4.294551717384737)
		\psline[linewidth=0.4pt]{->}(-0.6695375726110528,6.804753923346491)(2.974178179310998,6.804753923346491)
		\psline[linewidth=0.4pt]{->}(-0.6695375726110528,6.804753923346491)(5.444718839385009,6.793898792444937)
		\psline[linewidth=0.4pt]{->}(-0.6695375726110528,6.804753923346491)(-3.23554485688226,4.268492438630979)
		\psline[linewidth=0.4pt]{->}(2.1692976848565744,6.804753923346491)(-0.6695375726110528,6.804753923346491)
		\psline[linewidth=0.4pt](4.341956574085272,3.4483313233044046)(1.9771664192162834,1.)
		\psline[linewidth=0.4pt,linestyle=dashed,dash=1pt 1pt](4.341956574085272,3.4483313233044046)(1.9771664192162834,4.863874725866838)
		\psline[linewidth=0.8pt,linestyle=dotted](1.9771664192162834,3.)(4.341956574085272,3.4483313233044046)
		\psline[linewidth=0.8pt,linestyle=dotted](2.726930951536077,2.594907812615137)(1.9771664192162834,1.)
		\psline[linewidth=0.8pt,linestyle=dotted](2.726930951536077,2.594907812615137)(1.9771664192162834,3.)
		\psline[linewidth=0.8pt,linestyle=dotted](2.726930951536077,2.594907812615137)(4.341956574085272,3.4483313233044046)
		\psline[linewidth=0.8pt,linestyle=dotted](4.341956574085272,3.4483313233044046)(1.9771664192162834,3.8480141663808562)
		\psline[linewidth=0.4pt](1.5907186699453741,3.8593192131308)(2.2901874311203367,3.1455755792787916)
		\psline[linewidth=0.4pt](1.4354071949408849,2.4816466632381213)(1.8655205286487384,2.412071847315856)
		\psline[linewidth=0.4pt](2.1141954322587564,2.3718463781493835)(2.4118136137345134,2.3237038804202625)
		\psline[linewidth=0.4pt](1.3907887521832185,4.686064513559656)(1.8250813377748645,4.521737589281734)
		\psline[linewidth=0.4pt](2.0821540507771648,4.424466833010593)(2.480345410635691,4.273799831983041)
		\psline[linewidth=0.8pt,linecolor=red](1.9771664192162834,0.38119853327726183)(1.9771664192162834,3.335710788555077)
		\psline[linewidth=0.8pt,linecolor=red](1.9771664192162834,3.6244983893452227)(1.9771664192162834,5.477154782502023)
		\begin{scriptsize}
		\psdots[dotsize=2pt 0,dotstyle=*](-0.6695375726110528,6.804753923346491)
		\psdots[dotsize=2pt 0,dotstyle=*](2.1692976848565744,6.804753923346491)
		\psdots[dotsize=2pt 0,dotstyle=*](-0.6695375726110528,5.40563796425608)
		\psdots[dotsize=2pt 0,dotstyle=*](-6.4895721150954095,4.294551717384737)
		\psdots[dotsize=2pt 0,dotstyle=*](2.974178179310998,6.804753923346491)
		\psdots[dotsize=2pt 0,dotstyle=*](5.444718839385009,6.793898792444937)
		\psdots[dotsize=1pt 0,dotstyle=*](-3.23554485688226,4.268492438630979)
		\psdots[dotsize=2pt 0](1.9771664192162834,4.863874725866838)
		\psdots[dotsize=2pt 0,dotstyle=*](1.9771664192162834,1.)
		\psdots[dotsize=2pt 0,dotstyle=*](4.341956574085272,3.4483313233044046)
		\psdots[dotsize=2pt 0](1.9771664192162834,3.)
		\psdots[dotsize=1pt 0](1.9771664192162834,3.4649847750992566)
		\psdots[dotsize=1pt 0](1.9771664192162834,4.464191882790386)
		\psdots[dotsize=2pt 0](2.726930951536077,2.594907812615137)
		\psdots[dotsize=1pt 0](1.9771664192162834,2.3940120903263407)
		\psdots[dotsize=2pt 0](1.9771664192162834,3.8480141663808562)
		\end{scriptsize}
		\rput[tl](1.52,1.2){$A$}
		\rput[tl](1.52,5.1){$C$}
		\rput[tl](4.45,3.7){$B'$}
		\tiny
		\rput[tl](1.55,2.3){$X_1$}
		\rput[tl](1.55,3.5){$X_2$}
		\rput[tl](1.55,4.4){$X_3$}
		\rput[tl](2,2.9){$Y_1$}
		\rput[tl](2,4){$Y_2$}
		\rput[tl](2.67,2.55){$M'$}
		\rput[tl](3.1,2.6){$\delta_{1pr}$}
		\rput[tl](2.0112399408822417,2.1){$\delta_{2pr}$}
		\rput[tl](2.6641510383208287,3.05){$\delta_{3pr}$}
		\normalsize
		\rput[tl](3.6058497365495596,4.553641104930908){$\Delta_{pr}$}
		\rput[tl](1.7098963574490476,0.6738424682285147){$\ell$}
		\end{pspicture*}
		\caption{Decomposing a space slide move $\Delta=ABC$ to smaller moves. $ABC=AY_1B\cup Y_1Y_2B \cup \cdots$, and $AY_1B=\delta_1 \cup \delta_2 \cup \delta_3$ are depicted here as projected on $\pi$. Solid black lines are part of the projected arc. Triangle moves via the $\delta_i$'a are good.}
		\label{figure_decomposing_type_i_moves_2}
	\end{figure}
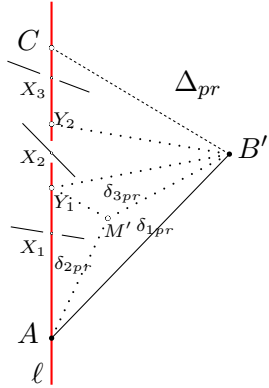

	$A$ and $Y_1$ project on $\pi$ onto themselves. Let us call the projection of $B$ as $B'$. We triangulate $B'AY_1$ so that one of the triangles is $AY_1M'$, where $M'$ is the projection of a point $M$ on $AX_1B$ chosen so close to the rail $\ell$ so that $AY_1M'$ contains no crossings of arcs of $pr_c$. Let us consider the following be moves in space: $\delta_1$ in triangle $AMB$ that replaces $AB$ by $AM,MB$, $\delta_2$ in triangle $MY_1B$ that replaces  $AM$ by $Y_1M$ (a slide move), and $\delta_3$ in triangle $CMB$ that replaces $Y_1M,MB$ by $Y_1B$. Then $\delta=\delta_3 \circ \delta_2, \delta_1$. But by Case (I), we know that $\delta_1,\delta_3$ are as required by our Lemma, whereas $\delta_2$ projects to a slide move thus it is also as required by our Lemma. Since the triangles of these moves are part of the triangle $ABC$ of move $\Delta$, Lemma \ref{lemma_triangle_moves} implies that $\delta$ is as required and we are done.
	\end{proof}

So if $\Delta$ is a triangle move between two rail arcs $c_1,c_2$ and $\Delta=\Delta_k \circ \cdots \circ \Delta_i$ a decomposition to nice and good moves as assured in Lemma \ref{lemma_nice_and_good_moves}, then Lemma \ref{lemma_triangle_moves} allows us to replace in the decomposition of $\Delta$, any good move $\Delta_i$ with a product of nice moves as in Lemma \ref{lemma_good_moves}. Thus we proved:

 \begin{lemma}
 		Any triangle move $\Delta$  between two rail arcs $c_1,c_2$ is a finite composition $\Delta_k \circ \cdots \circ \Delta_i$ of nice moves with triangles satisfying $\Delta_i\subseteq \Delta, \ \forall i$.
 \end{lemma}

So if the two rail arcs $c_1,c_2$ are rail isotopic, thus related by a finite sequence of triangle moves, we can replace by Lemma \ref{lemma_triangle_moves} any one of these moves by a sequence of moves as assured in Lemma \ref{lemma_good_moves}, and we get:

\begin{corollary}\label{corollary_nice_moves}
		If the two rail arcs $c_1,c_2$ are rail isotopic, then they are related by a finite sequence of nice moves.
\end{corollary}

The following is a special case of (one part of) the main Theorem that follows.
\begin{lemma}\label{lemma_arcs_with_same projection}
	If $c_1,c_2$ are two rail arcs with exactly the same (pointwise) projection $c_{1pr}=c_{2pr}$  on the plane $\pi$ of the rails, then there exists a rail isotopy between the arcs.
\end{lemma}

\begin{proof}
	If necessary, we subdivide the two arcs so that they get the same number of vertices and every vertex of each one of them lies on the same vertical line (with respect to $\pi$) with a vertex of the other. Since this subdivision can be performed via triangle moves, the resulting arcs $c'_1,c'_2$ are related to the original ones via isotopies. Pointwise the subdivided arcs have not changed.
	
	We perform a further subdivision of $c'_1,c'_2$ as thus: if $e_1,f_1$ are two edges of $c_1'$ whose projections have a crossing point $A$ on $\pi$, and $e_2,f_2$ the corresponding edges of $c_2$ with the same projections on $\pi$, then we chose on $e_{1pr}=e_{2pr}$ two nearby points on each of the two sides of $A$ and lift them vertically to two points on each of $e_1,e_2$ as new vertices; and similarly for $f_1,f_2$. We do so for every crossing point on the projections on $\pi$, and remaining very close to each crossing point, the segments on the old edges between the new vertices, remain disjoint, even if they happen to appear on the same old edge. Since the insertion of all these new vertices can be performed via triangle moves, the new arcs $c_1'',c_2''$ are related to $c_1',c_2'$, and then to the original, ones via isotopies. Pointwise the subdivided arcs have not changed.
	
	So actually we need to show that there exists an isotopy between $c_1'',c_2''$. This is not that hard to see, but it is rather technical:
	
	$c_1'',c_2''$ have the same number of vertices, paired so that any one of the first, along with its pair from the second lie on the same vertical line. Let's call such vertices as corresponding. Let us also call two edges, one from each arc, as corresponding whenever their endpoints are corresponding. 	We cannot just slide each vertex of $c_1''$ on its vertical  (w.r.t. to $\pi$) line to make it take the place of the corresponding vertex of $c_2''$ because this  causes a sliding of the issuing edges from these vertices, and these edges may be obstructed by other edges below them. Even if we try to slide all vertices at once in order to prevent such obstructions, it is not clear at al that such a simultaneous slide will have the desired result. So we deal first with the possible obstructions. Throughout below, we think of $c_2''$ as the fixed ideal position to push $c_1''$ to. And we perform isotopies on space changing the position of $c_1''$, but the ideal position of $c_2''$ remains fixed. 
	
	Obstructions occur whenever edges of $c_1$ project on $\pi$ forming crossing points. We'll deal first with the crossing points of two projected edges, ignoring any crossing points of a projected edge with the rails.
	
\begin{figure}[!h]
	\centering
	\psset{xunit=1.0cm,yunit=1.0cm,algebraic=true,dimen=middle,dotstyle=o,dotsize=5pt 0,linewidth=1.6pt,arrowsize=3pt 2,arrowinset=0.25}
	\begin{pspicture*}(0.81607641007183,0.7537939510585828)(9.854059197199275,9.844941342816208)
	\psaxes[labelFontSize=\scriptstyle,xAxis=true,yAxis=true,Dx=1.,Dy=1.,ticksize=-2pt 0,subticks=2]{->}(0,0)(0.81607641007183,0.7537939510585828)(9.854059197199275,9.844941342816208)
	\rput{-1}(4.98366553214101,2.290913617237049){\psellipse[linewidth=0.4pt](0,0)(0.605282899166271,0.45350950072752944)}
	\rput{-1}(4.98366553214101,2.2909136172370497){\psellipse[linewidth=0.4pt](0,0)(1.1574925601195851,0.8537493459564961)}
	\psline[linewidth=0.4pt](0.494620070686265,0.7537939510585828)(0.494620070686265,9.844941342816208)
	\psline[linewidth=0.4pt]{->}(0.494620070686265,3.529052040127685)(0.494620070686265,10.097223006506315)
	\rput{-1}(4.98366553214101,8.85908458361568){\psellipse[linewidth=0.4pt](0,0)(1.1574925601195851,0.8537493459564961)}
	\rput{-1}(4.98366553214101,8.85908458361568){\psellipse[linewidth=0.4pt](0,0)(0.605282899166271,0.45350950072752944)}
	\psline[linewidth=0.8pt,linestyle=dotted](6.140529608425473,8.820946207474433)(6.140529608425473,2.252775241095803)
	\psline[linewidth=0.4pt]{->}(0.494620070686265,3.529052040127685)(0.494620070686265,8.964779736441034)
	\psline[linewidth=0.8pt,linestyle=dotted](2.4490214371841574,2.374473312675188)(2.4490214371841574,7.810201008988538)
	\psline[linewidth=0.8pt,linestyle=dotted](7.706912908210799,2.2011362312127702)(7.706912908210799,7.63686392752612)
	\psline[linewidth=0.8pt,linestyle=dotted](2.4490214371841574,7.810201008988538)(7.706912908210799,7.63686392752612)
	\psline[linewidth=0.8pt,linestyle=dotted](3.2171542549463363,6.860311123390323)(3.2171542549463363,1.4245834270769735)
	\psline[linewidth=0.8pt](2.4490214371841574,5.329367386579625)(3.011132511908649,5.530006001089392)
	\psline[linewidth=0.8pt](3.262585755796804,5.619759134831568)(3.826801455856547,5.821148969394302)
	\psline[linewidth=0.8pt,linecolor=cyan](3.826801455856547,5.821148969394302)(4.423618912155867,6.0341755987056525)
	\psline[linewidth=0.8pt,linecolor=cyan](5.5847900445058745,6.4486413114077745)(6.1405296084254735,6.647005693728016)
	\psline[linewidth=0.8pt](6.1405296084254735,6.647005693728016)(7.706912908210799,7.206106891323531)
	\psline[linewidth=0.8pt](2.4490214371841574,6.454893959152483)(3.034371800253787,6.407724638325015)
	\psline[linewidth=0.8pt](3.3547356696196218,6.381908738841497)(3.815949767808518,6.344742697543533)
	\psline[linewidth=0.8pt](7.706912908210799,6.031196998630801)(6.151621755627318,6.156527111986529)
	\psline[linewidth=0.8pt,linecolor=cyan](3.815949767808518,6.344742697543533)(4.463994526861864,6.292521270721406)
	\psline[linewidth=0.8pt,linecolor=cyan](5.603586609210737,6.200689458637329)(6.151621755627318,6.156527111986529)
	\psline[linewidth=0.8pt,linecolor=red](5.603586609210737,6.200689458637329)(4.463994526861864,6.292521270721406)
	\psline[linewidth=0.8pt](3.2171542549463363,3.431075078413389)(4.0289336185377955,3.8352458741295363)
	\psline[linewidth=0.8pt,linecolor=cyan](4.0289336185377955,3.8352458741295363)(4.544466555221577,4.091920737500699)
	\psline[linewidth=0.8pt,linecolor=red](4.544466555221577,4.091920737500699)(5.422864509060439,4.52925976813065)
	\psline[linewidth=0.8pt,linecolor=cyan](5.422864509060439,4.52925976813065)(5.938397445744228,4.785934631501815)
	\psline[linewidth=0.8pt](3.2171542549463363,4.352322868313289)(4.026538144933626,4.32117259498423)
	\psline[linewidth=0.8pt,linecolor=cyan](4.026538144933626,4.32117259498423)(4.528244677256722,4.301863715963814)
	\psline[linewidth=0.8pt,linecolor=cyan](5.441229842739471,4.266726202353787)(5.8498086375070875,4.251001474828424)
	\psline[linewidth=0.8pt](2.4490214371841574,2.374473312675188)(3.054550058582813,2.3545108306510563)
	\psline[linewidth=0.8pt](3.3949988414513874,2.343287244402642)(7.706912908210799,2.2011362312127702)
	\psline[linewidth=0.8pt,linestyle=dotted](3.8268014558565473,2.329051993378296)(3.8268014558565473,3.469122079132818)
	\psline[linewidth=0.8pt,linestyle=dotted](3.8268014558565473,3.828665289815815)(3.8268014558565473,4.141311559974942)
	\psline[linewidth=0.8pt,linestyle=dotted](3.8268014558565473,4.485222457149981)(3.8268014558565473,5.642013656738753)
	\psline[linewidth=0.8pt,linestyle=dotted](3.8268014558565473,6.2360415700410945)(3.8268014558565473,6.)
	\psline[linewidth=0.8pt,linestyle=dotted](3.8268014558565473,6.533055526692265)(3.8268014558565473,7.)
	\psline[linewidth=0.8pt,linestyle=dotted](3.8268014558565473,7.2990388885821265)(3.8268014558565473,8.897222959756927)
	\psline[linewidth=0.8pt,linestyle=dotted](4.423618912155867,2.309376692621175)(4.423618912155867,3.8084614580867795)
	\psline[linewidth=0.8pt,linestyle=dotted](4.423618912155867,7.812075438449528)(4.423618912155867,8.010601751525368)
	\psline[linewidth=0.8pt,linestyle=dotted](4.423618912155867,8.877547658999806)(4.423618912155867,8.391110518254058)
	\psline[linewidth=0.8pt,linestyle=dotted](5.9383974457442275,9.3273029816195)(5.9383974457442275,8.67235612844483)
	\psline[linewidth=0.8pt,linestyle=dotted](5.9383974457442275,2.759132015240869)(5.9383974457442275,3.8603083612497224)
	\psline[linewidth=0.8pt](7.075193683313262,3.316638165006541)(6.307580117345769,2.9401861065571078)
	\psline[linewidth=0.8pt](3.2171542549463363,1.4245834270769735)(6.083280314111992,2.8301852853356464)
	\psline[linewidth=0.8pt,linestyle=dotted](5.543712152126153,8.840621508231553)(5.543712152126153,8.324935080562112)
	\psline[linewidth=0.8pt,linestyle=dotted](5.543712152126153,2.2724505418529235)(5.543712152126153,3.6622006922483377)
	\psline[linewidth=0.8pt,linestyle=dotted](7.075193683313262,8.752365861319891)(6.266985203049154,8.356005297175567)
	\psline[linewidth=0.8pt,linestyle=dotted](4.032075169073548,7.259963590722661)(3.2171542549463363,6.860311123390323)
	\psline[linewidth=0.8pt,linestyle=dotted](7.075193683313262,8.752365861319891)(7.075193683313262,7.87198865852776)
	\psline[linewidth=0.8pt,linestyle=dotted](7.075193683313262,7.492282292941773)(7.075193683313262,7.178611817022914)
	\psline[linewidth=0.8pt,linestyle=dotted](7.075193683313262,6.4191990858509405)(7.075193683313262,6.798905451436927)
	\psline[linewidth=0.8pt,linestyle=dotted](7.075193683313262,5.808367106430003)(7.075193683313262,3.316638165006541)
	\psline[linewidth=0.8pt](7.075193683313262,5.351925663686552)(6.2809742201172005,4.956497636306009)
	\psline[linewidth=0.8pt](5.938397445744228,4.785934631501815)(6.084326838872081,4.858590332772563)
	\psline[linewidth=0.8pt,linestyle=dotted](4.983665532141008,8.85908458361568)(4.983665532141008,8.449802693115132)
	\psline[linewidth=0.8pt,linestyle=dotted](4.983665532141008,8.268203996530529)(4.983665532141008,8.136132217196273)
	\psline[linewidth=0.8pt,linestyle=dotted](4.983665532141008,2.2909136172370497)(4.983665532141008,6.815414423853709)
	\psline[linewidth=0.8pt](6.064107858255647,4.242753868985137)(5.8498086375070875,4.251001474828424)
	\psline[linewidth=0.8pt](7.075193683313262,4.203840814006767)(6.262556954780726,4.235116277385979)
	\psline[linewidth=0.8pt](1.1417871745693557,1.136327912480684)(1.9507268229916772,3.3155122714959147)
	\psline[linewidth=0.8pt](9.54485413471143,3.3155122714959147)(8.735914486289108,1.136327912480684)
	\psline[linewidth=0.8pt](8.735914486289108,1.136327912480684)(1.1417871745693557,1.136327912480684)
	\psline[linewidth=0.8pt](1.9507268229916772,3.3155122714959147)(2.330433188577665,3.3155122714959147)
	\psline[linewidth=0.8pt](2.594576747246178,3.3155122714959147)(3.1228638645832043,3.3155122714959147)
	\psline[linewidth=0.8pt](7.547268472280798,3.3155122714959147)(6.309095541022144,3.3155122714959147)
	\psline[linewidth=0.8pt](9.54485413471143,3.3155122714959147)(7.827921003366094,3.3155122714959147)
	\psline[linewidth=0.8pt,linestyle=dotted](3.815949767808518,6.344742697543533)(3.815949767808518,7.016076779935211)
	\psline[linewidth=0.8pt,linestyle=dotted](4.463994526861864,6.292521270721406)(4.463994526861864,6.963855353113083)
	\psline[linewidth=0.8pt,linestyle=dotted](5.603586609210737,6.200689458637329)(5.603586609210737,6.872023541029006)
	\psline[linewidth=0.8pt,linestyle=dotted](6.151621755627318,6.156527111986529)(6.151621755627318,6.827861194378206)
	\psline[linewidth=0.8pt,linestyle=dotted](3.815949767808518,7.016076779935211)(4.463994526861864,6.963855353113083)
	\psline[linewidth=0.8pt,linestyle=dotted](4.463994526861864,6.963855353113083)(5.603586609210737,6.872023541029006)
	\psline[linewidth=0.8pt,linestyle=dotted](5.603586609210737,6.872023541029006)(6.151621755627318,6.827861194378206)
	\psline[linewidth=0.8pt,linestyle=dotted](4.423618912155867,6.0341755987056525)(4.423618912155867,5.362841516313975)
	\psline[linewidth=0.8pt,linestyle=dotted](5.5847900445058745,6.4486413114077745)(5.5847900445058745,5.777307229016097)
	\psline[linewidth=0.8pt,linestyle=dotted](3.826801455856547,5.149814887002624)(6.1405296084254735,5.975671611336339)
	\psline[linewidth=0.8pt,linecolor=red](4.423618912155867,6.0341755987056525)(5.5847900445058745,6.4486413114077745)
	\psline[linewidth=0.8pt,linecolor=red](4.528244677256722,4.301863715963814)(5.441229842739471,4.266726202353787)
	\begin{scriptsize}
	\psdots[dotsize=2pt 0,dotstyle=*](4.983665532141008,2.2909136172370497)
	\psdots[dotsize=1pt 0,dotstyle=*](0.494620070686265,3.529052040127685)
	\psdots[dotsize=1pt 0,dotstyle=*](0.494620070686265,10.097223006506315)
	\psdots[dotsize=1pt 0](2.4490214371841574,2.374473312675188)
	\psdots[dotsize=1pt 0,dotstyle=*](0.494620070686265,8.964779736441034)
	\psdots[dotsize=1pt 0](7.706912908210799,2.2011362312127702)
	\psdots[dotsize=1pt 0](3.2171542549463363,1.4245834270769735)
	\psdots[dotsize=1pt 0](7.075193683313262,3.316638165006541)
	\psdots[dotsize=1pt 0,dotstyle=*](2.4490214371841574,6.454893959152483)
	\psdots[dotsize=1pt 0,dotstyle=*](2.4490214371841574,5.329367386579625)
	\psdots[dotsize=1pt 0,dotstyle=*](7.706912908210799,7.206106891323531)
	\psdots[dotsize=1pt 0,dotstyle=*](7.706912908210799,6.031196998630801)
	\psdots[dotsize=1pt 0,dotstyle=*](3.826801455856547,5.821148969394302)
	\psdots[dotsize=1pt 0,dotstyle=*](6.1405296084254735,6.647005693728016)
	\psdots[dotsize=1pt 0,dotstyle=*](4.423618912155867,6.0341755987056525)
	\psdots[dotsize=1pt 0,dotstyle=*](5.5847900445058745,6.4486413114077745)
	\psdots[dotsize=1pt 0,dotstyle=*](4.463994526861864,6.292521270721406)
	\psdots[dotsize=1pt 0,dotstyle=*](5.603586609210737,6.200689458637329)
	\psdots[dotsize=1pt 0,dotstyle=*](6.151621755627318,6.156527111986529)
	\psdots[dotsize=1pt 0,dotstyle=*](3.815949767808518,6.344742697543533)
	\psdots[dotsize=1pt 0,dotstyle=*](3.2171542549463363,4.352322868313289)
	\psdots[dotsize=1pt 0,dotstyle=*](3.2171542549463363,3.431075078413389)
	\psdots[dotsize=1pt 0,dotstyle=*](7.075193683313262,5.351925663686552)
	\psdots[dotsize=1pt 0,dotstyle=*](7.075193683313262,4.203840814006767)
	\psdots[dotsize=1pt 0,dotstyle=*](4.0289336185377955,3.8352458741295363)
	\psdots[dotsize=1pt 0,dotstyle=*](4.544466555221577,4.091920737500699)
	\psdots[dotsize=1pt 0,dotstyle=*](5.422864509060439,4.52925976813065)
	\psdots[dotsize=1pt 0,dotstyle=*](5.938397445744228,4.785934631501815)
	\psdots[dotsize=1pt 0,dotstyle=*](4.026538144933626,4.32117259498423)
	\psdots[dotsize=1pt 0,dotstyle=*](5.441229842739471,4.266726202353787)
	\psdots[dotsize=1pt 0,dotstyle=*](4.528244677256722,4.301863715963814)
	\psdots[dotsize=1pt 0,dotstyle=*](5.8498086375070875,4.251001474828424)
	\end{scriptsize}
	\rput[tl](4.1477249668952805,7.310761855445174){$S_1$}
	\rput[tl](4.8211432922106585,7.293040320568453){$S_2$}
	\rput[tl](5.60089082678636,7.222154181061571){$S_3$}
	\rput[tl](4.023674222758237,5.45){$T_1$}
	\rput[tl](5.069244780484746,5.75){$T_2$}
	\rput[tl](5.671776966293242,5.915){$T_3$}
	\rput[tl](2.5705083628671574,7.935){$\pi_1$}
	\rput[tl](3.3,7.15){$\pi_2$}
	\rput[tl](1.6,1.55){$\pi$}
	\rput[tl](4.431269524922808,9.63){$C_A$}
	\rput[tl](3.45,9.667725994049002){$D_A$}
	\rput[tl](2.588229897743878,6.974052692787484){$c_1''$}
	\rput[tl](2.5527868279904373,5.999368274567855){$c_2''$}
	\rput[tl](3.3148128276894178,4.865190042457742){$c_2''$}
	\rput[tl](3.208483618429095,4.08544250788204){$c_1''$}
	\rput[tl](4.838864827087379,2.2){$A$}
	\rput[tl](4.094560362265119,6.735){$e_1$}
	\rput[tl](4.643927943443454,6.65){$e_2$}
	\rput[tl](5.6,6.45){$e_3$}
	\rput[tl](4.236332641278883,4.209493252019083){$f_1$}
	\rput[tl](4.679371013196895,4.42215167053973){$f_2$}
	\rput[tl](5.724941570923403,5.006962321471507){$f_3$}
	\rput[tl](4.041395757634958,6.176583623335061){$E_1$}
	\rput[tl](4.626206408566733,6.3183559023488245){$E_2$}
	\rput[tl](5.8,6.7){$E_3$}
	\rput[tl](3.9705096181280757,4.85){$F_1$}
	\rput[tl](4.555320269059852,4.77658236807414){$F_2$}
	\rput[tl](5.547726222156198,4.55){$F_3$}
	\rput[tl](5.0338017107313044,5.25){$\epsilon$}
	\end{pspicture*}
	\caption{The parts of $c_1'',c_2''$ over a neighborhood of each crossing $A$ of their projections are isolated via  two cylinders $C_A,D_A$, the first inside the second. Red edges are inside $C_A$, blue are between $C_A,D_A$.   We move $c_1''$ to a new position, by a p.l. isotopy fixing everything outside $D_A$ and pushing vertically everything inside $D_A$, so that the $e_2$ red arc assumes the position $E_2$. The blue arcs $e_1,e_3$ remain segments and retain their endpoints on the surface of $D_A$, but have moved their endpoints on the surface of $C_A$  vertically. The over/under relation of the red arcs $e_2,f_2$ has not changed, thus $(e_2\equiv E_2, f_2)$, is isotopic to $(E_2,F_2)$. We perform this isotopy vertically, in a similar way to the previous one, and now $c_1''$ is made to coincide with $c_2''$ over the neighborhood of the crossing $A$.}
	\label{figure_providing_a_rail_isotopy}
\end{figure}
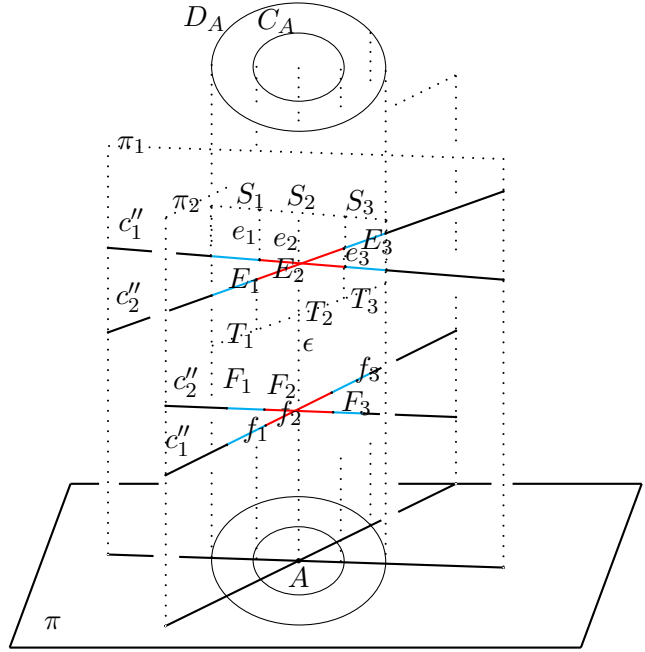
	
	If $A$ is a  crossing point of the common projection of $c_1'',c_2''$ on $\pi$ (Figure \ref{figure_providing_a_rail_isotopy}), then on the first arc there exist  consecutive edges $e_1,e_2,e_3$ on a line segment and also exist consecutive edges $f_1,f_2,f_3$ on another line segment, with the projections of  $e_2,f_2$ containing $A$. Let $E_1,E_2,E_3$ on a line segment, and $F_1,F_2,F_3$ on another line segment be the corresponding edges on the second arc of the edges $e_1,e_2,e_3$ and $f_1,f_2,f_3$ respectively. Let $\pi_1$ be the vertical plane to $\pi$ containing the edges $e_i,E_i$, and $\pi_2$ the vertical plane containing the corresponding edges $f_i,F_i$. Call $S_1,S_2,S_3$ the zones on $\pi_1$  that lie on and between the vertical lines passing from the endpoints of $e_i,E_i$ for $i=1,2,3$. And call $T_1,T_2,T_3$ the zones on $\pi_2$  that lie on and between the vertical lines passing from the endpoints of $f_i,F_i$ for $i=1,2,3$.   Because of our construction of the second subdivision above, a small solid infinite cylinder $C_A$ (surface union its interior) contains $S_2 \cup T_2$ and nothing more of $c_1'',c_2''$. Similarly, an infinite solid  cylinder $D_A$ contains $(\cup_i S_i) \cup (\cup_i T_i)$ and nothing more of $c_1'',c_2''$. Now, $e_2,f_2,E_2,F_2$ are segments intersecting the vertical line $\epsilon$ through $A$ at points $1,2,3,4$ respectively. Make $\epsilon$ an axis and compare relative positions of the points $1,2,3,4$ using their coordinates. For definiteness, let $1$ be above $2$ (i.e. have greater coordinate). Since $c''_{1pr},c''_{2pr}$ carry the same over/under data, $3$ should be above $4$ too. In $S_2 \cup T_2$ we perform a vertical (w.r.t. $\pi$) push in $S_2,T_2$ up until $e_2$ is made to coincide with $E_2$. Then $1$ coincides with $3$, and  the points $2,4$ of $\epsilon$ both lie below $1\equiv3$ ((a) of Figure \ref{figure_providing_a_rail_isotopy_2} cannot happen). Then the edges $f_2,F_2$ define in the zone $T_2$ on their plane $\pi_2$, a quadrangle $q$ either as edges or as diagonals. In both cases, we can push the edge $f_2$ vertically inside $q$ (Figure \ref{figure_providing_a_rail_isotopy_2} (b)), until it coincides with $F_2$ without disturbing edge $e_2$ that lies above.
	
	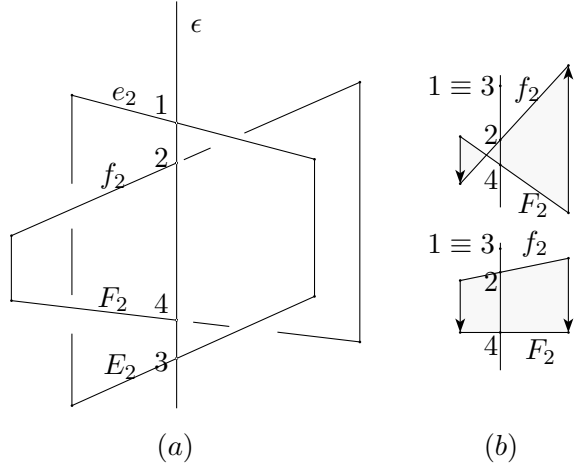
\begin{figure}[!h]
		\centering
		\newrgbcolor{eqeqeq}{0.8784313725490196 0.8784313725490196 0.8784313725490196}
		\psset{xunit=1.0cm,yunit=1.0cm,algebraic=true,dimen=middle,dotstyle=o,dotsize=5pt 0,linewidth=1.6pt,arrowsize=3pt 2,arrowinset=0.25}
		\begin{pspicture*}(0.6571315462090326,0.42817791707609426)(9.121245047781064,7.1806246808534615)
		\psaxes[labelFontSize=\scriptstyle,xAxis=true,yAxis=true,Dx=1.,Dy=1.,ticksize=-2pt 0,subticks=2]{->}(0,0)(0.6571315462090326,0.42817791707609426)(9.121245047781064,7.1806246808534615)
		\pspolygon[linewidth=0.pt,linecolor=eqeqeq,fillcolor=eqeqeq,fillstyle=solid,opacity=0.2](6.687223539907845,5.202000745421019)(8.124992565053862,4.187104962964994)(8.124992565053862,6.146232031910747)(6.687223539907845,4.573866162744054)
		\pspolygon[linewidth=0.pt,linecolor=eqeqeq,fillcolor=eqeqeq,fillstyle=solid,opacity=0.2](6.687223539907845,2.595242227311618)(6.687223539907845,3.2861902682562776)(8.124992565053862,3.582201538139282)(8.124992565053862,2.595242227311614)
		\psline[linewidth=0.4pt](1.5249378051406488,5.751266786293768)(4.750257152703776,4.9006331122111835)
		\psline[linewidth=0.4pt](1.5249378051406488,1.6221491600178881)(4.750257152703776,3.0753150199089703)
		\psline[linewidth=0.4pt](3.8177084836263337,5.2480342312551596)(5.352437551719306,5.924355515499525)
		\psline[linewidth=0.4pt](4.257437047067973,2.5995306248873975)(5.352437551719306,2.4696153107762213)
		\psline[linewidth=0.4pt](2.9171342742312705,7.)(2.9171342742312705,1.586706090264447)
		\psline[linewidth=0.4pt](0.7199450044767285,3.882918121799391)(0.7199450044767285,3.019233070618565)
		\psline[linewidth=0.4pt](5.352437551719306,5.924355515499525)(5.352437551719306,2.4696153107762213)
		\psline[linewidth=0.4pt](4.750257152703776,4.9006331122111835)(4.750257152703776,3.0753150199089703)
		\psline[linewidth=0.4pt](0.7199450044767285,3.882918121799391)(2.9171342742312705,4.8511710203353)
		\psline[linewidth=0.4pt](0.7199450044767285,3.019233070618565)(2.9171342742312705,2.7585495979358208)
		\psline[linewidth=0.4pt](1.5249378051406488,5.751266786293768)(1.5249378051406488,4.573866162744052)
		\psline[linewidth=0.4pt](1.5249378051406488,3.977138309200936)(1.5249378051406488,3.093036554785691)
		\psline[linewidth=0.4pt](1.5249378051406488,2.5795388627446902)(1.5249378051406488,1.6221491600178881)
		\psline[linewidth=0.4pt](3.6168595305039055,2.675531347191609)(3.136977518219329,2.7324665011914746)
		\psline[linewidth=0.4pt](3.0128269833462307,4.89334068875884)(3.368389023667661,5.050029045510659)
		\psline[linewidth=0.4pt](6.687223539907845,5.202000745421019)(8.124992565053862,4.187104962964994)
		\psline[linewidth=0.4pt](6.687223539907845,4.573866162744054)(8.124992565053862,6.146232031910747)
		\psline[linewidth=0.4pt](6.687223539907845,3.2861902682562776)(8.124992565053862,3.582201538139282)
		\psline[linewidth=0.4pt](6.687223539907845,2.595242227311618)(8.124992565053862,2.595242227311614)
		\psline[linewidth=0.4pt](7.221137935183261,6.018575702901075)(7.221137935183261,4.2597988714055735)
		\psline[linewidth=0.4pt](7.221137935183261,3.78869793439785)(7.221137935183261,2.092734561170046)
		\psline[linewidth=0.4pt]{->}(6.687223539907845,5.202000745421019)(6.687223539907845,4.573866162744054)
		\psline[linewidth=0.4pt]{->}(8.124992565053862,4.187104962964994)(8.124992565053862,6.146232031910747)
		\psline[linewidth=0.4pt]{->}(6.687223539907845,3.2861902682562776)(6.687223539907845,2.595242227311618)
		\psline[linewidth=0.4pt]{->}(8.124992565053862,3.582201538139282)(8.124992565053862,2.595242227311614)
		\begin{scriptsize}
		\psdots[dotsize=1pt 0,dotstyle=*](1.5249378051406488,5.751266786293768)
		\psdots[dotsize=1pt 0,dotstyle=*](4.750257152703776,4.9006331122111835)
		\psdots[dotsize=1pt 0,dotstyle=*](1.5249378051406488,1.6221491600178881)
		\psdots[dotsize=1pt 0,dotstyle=*](4.750257152703776,3.0753150199089703)
		\psdots[dotsize=1pt 0,dotstyle=*](0.7199450044767285,3.882918121799391)
		\psdots[dotsize=1pt 0,dotstyle=*](5.352437551719306,5.924355515499525)
		\psdots[dotsize=1pt 0,dotstyle=*](0.7199450044767285,3.019233070618565)
		\psdots[dotsize=1pt 0,dotstyle=*](5.352437551719306,2.4696153107762213)
		\psdots[dotsize=1pt 0](2.9171342742312705,2.249402514223554)
		\psdots[dotsize=1pt 0](2.9171342742312705,5.384094091148988)
		\psdots[dotsize=1pt 0](2.9171342742312705,4.8511710203353)
		\psdots[dotsize=1pt 0](2.9171342742312705,2.7585495979358208)
		\psdots[dotsize=1pt 0,dotstyle=*](7.221137935183261,5.157762990054394)
		\psdots[dotsize=1pt 0,dotstyle=*](7.221137935183261,4.82511999581484)
		\psdots[dotsize=1pt 0,dotstyle=*](7.221137935183261,3.3961138202247456)
		\psdots[dotsize=1pt 0,dotstyle=*](7.221137935183261,2.595242227311616)
		\psdots[dotsize=1pt 0,dotstyle=*](6.687223539907845,4.573866162744054)
		\psdots[dotsize=1pt 0,dotstyle=*](8.124992565053862,6.146232031910747)
		\psdots[dotsize=1pt 0,dotstyle=*](6.687223539907845,5.202000745421019)
		\psdots[dotsize=1pt 0,dotstyle=*](8.124992565053862,4.187104962964994)
		\psdots[dotsize=1pt 0,dotstyle=*](6.687223539907845,3.2861902682562776)
		\psdots[dotsize=1pt 0,dotstyle=*](6.687223539907845,2.595242227311618)
		\psdots[dotsize=1pt 0,dotstyle=*](8.124992565053862,3.582201538139282)
		\psdots[dotsize=1pt 0,dotstyle=*](8.124992565053862,2.595242227311614)
		\psdots[dotsize=1pt 0,dotstyle=*](7.221137935183261,5.877245421798759)
		\psdots[dotsize=1pt 0,dotstyle=*](7.221137935183261,3.7101811115632324)
		\end{scriptsize}
		\rput[tl](3.1,6.75663383754651){$\epsilon$}
		\rput[tl](2.62,5.75){$1$}
		\rput[tl](2.62,5.05){$2$}
		\rput[tl](2.62,2.33){$3$}
		\rput[tl](2.62,3.1){$4$}
		\rput[tl](2.054730992665268,5.85){$e_2$}
		\rput[tl](1.8976973469960279,4.85){$f_2$}
		\rput[tl](1.9448074406968,2.3){$E_2$}
		\rput[tl](1.8662906178621799,3.2){$F_2$}
		\rput[tl](7,5.35){$2$}
		\rput[tl](7,4.75){$4$}
		\rput[tl](7,3.4){$2$}
		\rput[tl](7,2.55){$4$}
		\rput[tl](6.25,6){$1\equiv 3$}
		\rput[tl](6.25,3.9){$1\equiv 3$}
				\rput[tl](7.4,6){$f_2$}
				\rput[tl](7.45,4.45){$F_2$}
					\rput[tl](7.5,3.95){$f_2$}
				\rput[tl](7.55,2.5){$F_2$}
		\rput[tl](2.65,1.25){$(a)$}
		\rput[tl](7,1.25){$(b)$}
		\end{pspicture*}
		\caption{The situation in (a) cannot happen because $c_1'',c_2''$ carry the same over/under data. After pushing $e_2$ to assume position $E_2$, we make $f_2$ to assume position $F_2$ by moving it on $\pi_2$ as shown in (b).}
		\label{figure_providing_a_rail_isotopy_2}
	\end{figure}
	
	We take care to perform the above vertical pushes as p.l. isotopies in space fixing everything on the surface and outside cylinder $D_A$. As a result, in their new positions $e_2,f_2$ coincide with $E_2,F_2$, and $e_1,e_3,f_1,f_3,E_1,E_3,F_1,F_3$ change their carrier lines, all projections on $\pi$ remain as before the isotopy. Performing for each crossing point such isotopies consecutively, and even all at once, we get a new $c_1''$, say $c_1'''$. Some of its vertices still may not coincide with those of $c''$. So at the exterior of the union of the $C_A$'s we perform a simultaneous vertical pushing of all vertices of $c_1'''$ that do not already coincide with their corresponding vertices of $c_2''$ until they do coincide. This push is not obstructed by edges crossing the moving ones, but it can in principal be obstructed by the rails. Nevertheless, for corresponding edges, say $e,f$ of $c_1''',c_2''$, the common projection to $\pi$ is for both an underpassing or an overpassing for the crossing rail. This means that both pairs of corresponding endpoints-vertices, form vertical segments that do not intersect the rails, and as parallel segments, they form a quadrilateral $q$ not containing the rails. So the push of such vertices takes place in these quadrilaterals and is not obstructed by the rails. As above we perform this push as a p.l. isotopy which does not disturb those vertices which were previously made to coincide with their corresponding of $c_2''$, and we are done.  
	\end{proof}

We can now prove the following:

\begin{theorem}
	Two rail arcs in $\mathbb{R}^3$ are rail isotopic iff their rail knotoid diagram projections on the plane $\pi$ of the rails are rail equivalent. In other words, rail isotopy in $\mathbb{R}^3$ corresponds to rail equivalence on $\pi$ (rail arcs are isotopic iff they correspond to the same rail knotoid).
\end{theorem}

\begin{proof}

	$(\Rightarrow)$   If  $c_1,c_2$ are rail arcs which are rail isotopic in $\mathbb{R}^3$, then by Corollary \ref{corollary_nice_moves} there exists a rail isotopy between them expressed by a finite sequence of nice triangle moves. By definition, the projection of any such move is a Reidemeister or slide move or a planar isotopy move. Thus $c_{pr1},c_{pr2}$ are equivalent rail knotoid diagram projections as wanted.
	
	$(\Leftarrow)$  Let for two arcs $c_1,c_2$  in $\mathbb{R}^3$, their corresponding rail knotoid diagram projections  $c_{1pr},c_{2pr}$ on $\pi$ differ by a single Reidemeister move, slide move or planar isotopy move $\delta$. In each case one can readily check that there exist a few obvious triangles $\Delta_1,\ldots$ in space so that for their projected moves on $\pi$ it is $(\cdots \circ \Delta_1)_{pr}=\delta$ wehereas in space $(\cdots \circ \Delta_1)(c_1)$ coincides with $c_2$, thus $c_1,c_2$ are rail isotopic as wanted. Let us notice that no matter what kind of move $\delta$ is, the $\Delta_i$'s indeed provide us  with triangle moves in space: when the rails are not involved in $\delta$ the result is immediate, and whenever a part of a rail is involved in $\delta$, the $\Delta_i$'s either do not have common points with the rails, or they have a whole side on a rail, thus by the definition of a triangle move, we get the required result.

	In the general case of two rail arcs $c_1,c_2$ in $\mathbb{R}^3$ with equivalent rail knotoid diagram projections $c_{1pr},c_{2pr}$ on $\pi$, we note that $c_{2pr}$ comes from $c_{1pr}$ via a finite sequence of Reidemeister moves, slide moves or planar isotopies. Let on $\pi$ be $d_0=c_{1pr},d_1,\ldots,d_k=c_{2pr}$ the diagrams obtained consecutively by such a sequence of moves, and let in space, $f_0=c_1,f_1,\ldots,f_k$ be arcs with projections the $d_i$'s respectively. Then by what we proved just above, there exist a space isotopy between any $f_i$ and the next one $f_{i+1}$, thus there exists such an isotopy between $c_1=f_0$ and $f_k$. But $f_k$ and $c_2$ share the same projection $c_{2pr}$ on $\pi$. Thus Lemma \ref{lemma_arcs_with_same projection} assures the existence of an isotopy between $f_k$ and $c_2$, and we finally get a desired isotopy between $c_1$ and $c_2$.
\end{proof}

\section{Rail knotoids and theta-curves}

Closing this article, it is worth mentioning that a rail arc together with the two rails is a kind of a trivalent graph  embedded in the 3-space $\mathbb{R}^3$, containing the end points of the rail arc as its only two vertices, the rail arc as an edge, and the  four half-lines  of the rails emanating from the two vertices, as infinitely extended edges. Clearly there is a direct connection of these graphs with the $\theta$-curves, where a theta-curve is  a graph with the form of the Greek letter $\theta$, embedded in the 3-sphere or the 3-space as a trivalent graph with only two vertices and exactly three edges (upper, middle, lower) each one of which joins the two vertices \cite{Wol}, \cite{Kau}  \cite{Yam}. From our infinitely extended trivalent graph, one gets a theta-curve as in Figure \ref{figure_theta-curves}, where the two horizontal line segments joining the two rails are chosen far away from the rail arc (say outside a 3-disk containing the rail arc), one segment on either side  (on each rail, the corresponding endpoint of the rail arc lies between the endpoints of the two horizontal segments). Conversely, from a $\theta$-curve we can get a rail knotoid as follows: from each one of the upper and lower edges of the $\theta$-curve, we first delete an arc from its interior leaving two small disjoint arcs touching one vertex each, then  we make the two arcs around each vertex into a line segment, and finally make the two segments parallel  extending them indefinitely to form the parallel lines of the rails; the middle edge of the theta-curve becomes a rail arc.  These remarks suggest the possibility of exchanging information between the theory of rail knotoids and the theory of theta-curves. Theta-curves exhibit a richness of properties and recently, Turaev has connected them to the spherical knotoids as well \cite{Tu}. 

	\begin{figure}[!h]
	\centering
	\psset{xunit=1.0cm,yunit=1.0cm,algebraic=true,dimen=middle,dotstyle=o,dotsize=5pt 0,linewidth=1.6pt,arrowsize=3pt 2,arrowinset=0.25}
	\begin{pspicture*}(5.360057220771243,2.553288575002056)(7.778380719300864,7.389935572061283)
	\psaxes[labelFontSize=\scriptstyle,xAxis=true,yAxis=true,Dx=0.5,Dy=0.5,ticksize=-2pt 0,subticks=2]{->}(0,0)(5.360057220771243,2.553288575002056)(7.778380719300864,7.389935572061283)
	\psline[linewidth=0.8pt,linecolor=red](5.7897163177173985,3.579614158436346)(5.7897163177173985,6.007590970387144)
	\psline[linewidth=0.8pt,linecolor=red](7.348194144442746,3.579614158436346)(7.348194144442746,4.56234264157902)
	\psline[linewidth=0.8pt,linecolor=red](7.348194144442746,4.750538252321563)(7.348194144442746,6.007590970387144)
	\parametricplot[linewidth=0.8pt]{1.3124896428446133}{2.0491738013646073}{1.*2.0441623491706924*cos(t)+0.*2.0441623491706924*sin(t)+6.730724779917389|0.*2.0441623491706924*cos(t)+1.*2.0441623491706924*sin(t)+3.123699349255933}
	\parametricplot[linewidth=0.8pt]{-0.897963385319354}{1.212895220174818}{1.*0.21069737658829024*cos(t)+0.*0.21069737658829024*sin(t)+7.336666585591969|0.*0.21069737658829024*cos(t)+1.*0.21069737658829024*sin(t)+4.843622530429749}
	\parametricplot[linewidth=0.8pt]{3.456529644292434}{5.190203848893322}{1.*0.643976505465421*cos(t)+0.*0.643976505465421*sin(t)+7.171847961037518|0.*0.643976505465421*cos(t)+1.*0.643976505465421*sin(t)+5.250697016198361}
	\parametricplot[linewidth=0.8pt]{1.727268012851006}{2.949250896956655}{1.*0.9302669017999806*cos(t)+0.*0.9302669017999806*sin(t)+7.493161332981229|0.*0.9302669017999806*cos(t)+1.*0.9302669017999806*sin(t)+5.088688895717715}
	\psline[linewidth=0.8pt,linecolor=red](5.7897163177173985,5.966145079751118)(5.7897163177173985,6.562810939433099)
	\psline[linewidth=0.8pt]{->}(0.,7.532357133936048)(3.868811445485548,7.532357133936048)
	\psline[linewidth=0.8pt]{->}(3.868811445485548,7.532357133936048)(0.,7.532357133936048)
	\rput[tl](5.816344673324002,3.064330521861144){$\ell_1$}
	\rput[tl](7.413350757258657,3.036953274707979){$\ell_2$}
	\psline[linewidth=0.8pt,linestyle=dotted,linecolor=red](5.7897163177173985,6.562810939433099)(5.7897163177173985,7.2367330934889305)
	\psline[linewidth=0.8pt](5.7897163177173985,3.579614158436346)(7.348194144442746,3.579614158436346)
	\psline[linewidth=0.8pt,linestyle=dotted,linecolor=red](5.7897163177173985,3.579614158436346)(5.7897163177173985,2.9808685315747065)
	\psline[linewidth=0.8pt](7.348194144442746,6.562810939433099)(5.7897163177173985,6.562810939433099)
	\psline[linewidth=0.8pt,linestyle=dotted,linecolor=red](7.348194144442746,7.2367330934889305)(7.348194144442746,6.562810939433099)
	\psline[linewidth=0.8pt,linecolor=red](7.348194144442746,6.562810939433099)(7.348194144442746,6.007590970387144)
	\psline[linewidth=0.8pt,linestyle=dotted,linecolor=red](7.348194144442746,3.579614158436346)(7.348194144442746,2.9808685315747065)
	\begin{scriptsize}
	\psdots[dotsize=2pt 0,dotstyle=*](7.348194144442746,6.007590970387144)
	\psdots[dotsize=2pt 0,dotstyle=*](5.7897163177173985,4.938390172463857)
	\psdots[dotsize=1pt 0,dotstyle=*](3.868811445485548,7.532357133936048)
	\psdots[dotsize=3pt 0,dotstyle=*,linecolor=darkgray](0.,7.532357133936048)
	\end{scriptsize}
	\end{pspicture*}
	\caption{The solid red segments on the two rails, together with the two horizontal line segments joining the two rails and also together with the pictured rail arc, form a theta-curve with the endpoints of the rail arc as its vertices, and  the rail knotoid as its middle edge.}
	\label{figure_theta-curves}
\end{figure}
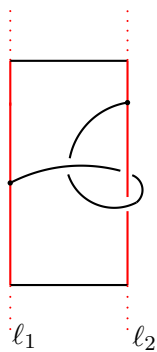


\begin{thebibliography}{00}

\bibitem{AdHen} Adams, C., Henrich, A., Kearney, K., and Scoville, N.,  Knots related by knotoids. To appear in {\em American Mathematical Monthly} (2019).

\bibitem{Ba} Bartholomew, A. Andrew Bartholomew's Mathematics Page:

 Knotoids. http://www.layer8.co.uk/maths/knotoids/index.htm.

\bibitem{BBHL}
Barbensi, A.,  Buck, D.,  Harrington, H.A., and  Lackenby M., Double branched covers of knotoids. ArXiv:1811.09121 [math.GT] (2018).


		\bibitem{DST}				  Daikoku, K.,  Sakai, K.,  and   Takase, M., On a move
reducing the genus of a knot diagram. In {\it Indiana University
Mathematics Journal} 61.3 (2012), pp. 1111--1127. ISSN: 00222518,
19435258. url: http://www.jstor.org/stable/24904076.


\bibitem{GDBS}
Goundaroulis, D., Dorier,  J., Benedetti,  F., and Stasiak, A., Studies of global and local entanglements of individual protein chains using the concept of knotoids. {\em Sci. Reports} (2017), {\em 7}, 6309.

\bibitem{GGLDSK}
Goundaroulis, D., G\"ug\"umc\"u, N., Lambropoulou, S., Dorier, J., Stasiak A., and Kauffman L.H.,  Topological models for open knotted protein chains using the concepts of knotoids and bonded knotoids. {\em  Polymers},  Special issue on Knotted and Catenated Polymers, Dusan Racko and Andrzej Stasiak Eds. (2017), {\em 9(9)}, 444, DOI: 10.3390/polym9090444.

	\bibitem{DRGDSMRSS}  Dabrowski-Tumanski, P., Rubach, P.,   Goundaroulis, D.,  Dorier, J., Sułkowski, P., Millett, K.C.,  Rawdon, E.J.,  Stasiak, A., and  Sułkowska, J.I., KnotProt 2.0: a database of proteins with knots and other entangled structures, {\em Nucleic Acids Research} {\bf 1} (2018). DOI: 10.1093/nar/gky1140.

\bibitem{Gthesis} 
G\"ug\"umc\"u, N., On knotoids, braidoids and their applications, {\em PhD thesis}, National Technical U. Athens, (2017).

\bibitem{GK1}  G\"ug\"umc\"u, N., Kauffman, L.H., New invariants of knotoids, {\em European J. of Combinatorics} {\bf 2017}, {\em 65C}, 186-229. 

\bibitem{GKL} G\"ug\"umc\"u, N., Kauffman, L.H., and Lambropoulou, S., A survey on knotoids, braidoids and their applications. To appear in ``Knots, Low-dimensional Topology and Applications" with subtitle ``Proceedings of the   International Conference on Knots, Low-dimensional Topology and Applications - Knots in Hellas 2016",  {\it Springer Proceedings in Mathematics \& Statistics (PROMS)}; 2019; C. Adams et al, Eds.
		
	\bibitem{GL1} G\"ug\"umc\"u, N., Lambropoulou, S., Knotoids, Braidoids and Applications. {\em Symmetry} (2017), {\bf 9(12)}, 315,	 https://doi.org/10.3390/sym9120315.
		
			\bibitem{GL2} 
	G\"ug\"umc\"u, N., Lambropoulou, S., Braidoids, submitted for publication.
		
			\bibitem{GN1}	G\"ug\"umc\"u, N., Nelson, S., Biquandle coloring invariants of knotoids. To appear in  {\em J. Knot Theory  Ramifications}. ArXiv preprint: https://arxiv.org/abs/1803.11308.
			
\bibitem{Kau}	Kauffman, L. Invariants of Graphs in Three-space, Transactions of the American Mathematical Society, Volume 311, No. 2, 1989.		

\bibitem{KoLa1} Kodokostas, D., Lambropoulou, S.,  A spanning set and potential basis of the mixed Hecke algebra on two fixed strands. {\em Mediterranean J. Math.} (2018), {\bf 15:192},  https://doi.org/10.1007/s00009-018-1240-7 .


\bibitem{KMT} Korablev, P.G., May, Y.K., Tarkaev, V. Classification of low complexity knotoids, {\em Siberian Electronic Mathematical Reports} (2018), {\bf 15}, 1237-–1244. [Russian, English abstract].
DOI: 10.17377/semi.2018.15.100.

\bibitem{Tu} Turaev, V., Knotoids, {\em Osaka Journal of Mathematics}  (2012), {\bf 49}, 195--223. See also arXiv:1002.4133v4.

\bibitem{Wol} Wolcott, K., The knotting of theta-curves and other graphs in $S^3$,  {\em Geometry and Topology: Manifolds, Varieties, and Knots},  McCrory, C., Shifrin, T., Eds., 1987, 235--346, Marcel Dekker.

\bibitem{Yam} Yamada, S., An Invariant of Spatial Graphs, {\em Journal of Graph Theory} (1989), Vol. 13, No. 5, 537--551. 

\end{thebibliography}
\end{document}